\documentclass[11pt]{amsart}
\usepackage{amsmath, amssymb, graphicx, geometry, xr}
\graphicspath{{.}{figs/}}

\usepackage{times}
\usepackage[percent]{overpic}
\usepackage{varioref}
\usepackage{placeins}
\usepackage[matrix,arrow]{xy}
\usepackage{array}
\usepackage{longtable}
\usepackage{dcolumn}
\usepackage{booktabs}
\usepackage{fullpage}
\usepackage{epic,eepic}
\usepackage{fp}
\usepackage{units}
\usepackage{url}
\usepackage{hyperref}

\theoremstyle{plain}
\newtheorem{theorem}{Theorem}[section]

\theoremstyle{definition}
\newtheorem{definition}[theorem]{Definition}

\theoremstyle{definition}
\newtheorem{example}[theorem]{Example}

\theoremstyle{definition}

\theoremstyle{plain}
\newtheorem{proposition}[theorem]{Proposition}

\theoremstyle{plain}
\newtheorem{lemma}[theorem]{Lemma}

\theoremstyle{plain}
\newtheorem{corollary}[theorem]{Corollary}

\theoremstyle{plain}
\newtheorem{question}[theorem]{Question}

\theoremstyle{plain}
\newtheorem{claim}[theorem]{Claim}

\newcommand\cross{\times}
\newcommand{\MCG}{\operatorname{MCG}}
\newcommand{\Sym}{\operatorname{Sym}}
\newcommand{\Z}{\mathbf{Z}}
\newcommand{\id}{e}  
\newcommand{\s}[2]{\Sigma_{#1,#2}}
\newcommand{\co}{\operatorname{:}}
\newcommand{\Lk}{\operatorname{Lk}}

\renewcommand{\wr}{\operatorname{wr}}
\newcommand{\Jones}{\operatorname{Jones}}
\newcommand{\minus}{\hbox{-}} 

\everymath{\displaystyle}

\begin{document}

\title{Intrinsic Symmetry Groups of Links with 8 and fewer crossings}
\date{Revised: \today}

\author{Michael Berglund}

\author{Jason Cantarella}

\author{Meredith Perrie Casey}

\author{Ellie Dannenberg}

\author{Whitney George}

\author{Aja Johnson}

\author{Amelia Kelly}

\author{Al LaPointe}

\author{Matt Mastin}


\author{Jason Parsley}

\author{Jacob Rooney}

\author{Rachel Whitaker}

\begin{abstract}
We present an elementary derivation of the ``intrinsic'' symmetry groups for knots and links of 8 or fewer crossings. The standard symmetry group for a link is the mapping class group $\MCG(S^3,L)$ or $\Sym(L)$ of the pair $(S^3,L)$. Elements in this symmetry group can (and often do) fix the link and act nontrivially only on its complement. We ignore such elements and focus on the ``intrinsic'' symmetry group of a link, defined to be the image $\Sigma(L)$ of the natural homomorphism $\MCG(S^3,L) \rightarrow \MCG(S^3) \cross \MCG(L)$. This different symmetry group, first defined by Whitten in 1969, records directly whether $L$ is isotopic to a link $L'$ obtained from $L$ by permuting components or reversing orientations. 

For hyperbolic links both $\Sym(L)$ and $\Sigma(L)$ can be obtained using the output of \texttt{SnapPea}, but this proof does not give any hints about how to actually construct isotopies realizing $\Sigma(L)$. We show that standard invariants are enough to rule out all the isotopies outside $\Sigma(L)$ for all links except  $7^2_6$, $8^2_{13}$ and $8^3_5$ where an additional construction is needed to use the Jones polynomial to rule out ``component exchange'' symmetries. On the other hand, we present explicit isotopies starting with the positions in Cerf's table of oriented links which generate $\Sigma(L)$ for each link in our table.  Our approach gives a constructive proof of the $\Sigma(L)$ groups.
\end{abstract}

\keywords{knot, symmetry group of knot, link symmetry, Whitten group}

\maketitle

{\small
\setcounter{tocdepth}{2}


}

\bigskip

\section{Introduction}
The symmetry group of a link $L$ is defined to be the mapping class group $\MCG(L)$ (or $\Sym(L)$) of the pair $(S^3,L)$. The study of this symmetry group is a classical topic in knot theory, and these groups have now been computed for prime knots and links in several ways. Kodama and Sakuma~\cite{MR1177431} used a method in Bonahon and Siebenmann~\cite{bs} to compute these groups for all but three of the knots of 10 and fewer crossings in 1992. In the same year, Weeks and Henry used the program SnapPea to compute the symmetry groups for hyperbolic knots and links of 9 and fewer crossings~\cite{MR1164115}. These efforts followed earlier tabulations of symmetry groups by Boileau and Zimmermann~\cite{MR891592}, who found symmetry groups for nonelliptic Montesinos links with 11 or fewer crossings. 

We consider a different group of symmetries of a link $L$ given by the image of the natural homomorphism
\begin{equation*}
\pi \co \Sym(L) = \MCG(S^3,L) \rightarrow \MCG(S^3) \cross \MCG(L).
\end{equation*}
Since these symmetries record an action on $L$ itself (and only record the orientation of the ambient $S^3$), we will call them ``intrinsic'' symmetries of $L$ to distinguish them from the standard symmetry group. 

Unlike the elements in the $\Sym$ group, which may be somewhat difficult to describe explicitly, each of the elements in $\Sigma(L)$ corresponds to an isotopy of $L$ which may exchange the position of some components, which may mirror crossings, and which may reverse orientations of some components. Neither the Boileau-Zimmerman or the Henry-Weeks-\texttt{SnapPea} method gives much insight into what those isotopies might look like. In addition, it is worth noting that \texttt{SnapPea} is a large and complicated computer program, and while its results are accurate for the links in our table, it is always worthwhile to have alternate proofs for results that depend essentially on nontrivial computer calculations.

In this spirit, the present paper presents an elementary and explicit derivation of the $\Sigma(L)$ groups for all links of 8 and fewer crossings. We rule out certain isotopies using elementary and polynomial invariants to provide an upper bound on the size of $\Sigma(L)$ for each link in our table and then present explicit isotopies generating $\Sigma(L)$ starting with the configurations of the link in Cerf's table of alternating oriented links~\cite{cerf} or (for nonalternating links) Doll and Hoste's table~\cite{MR1094946}. For three links in our table, $7^2_6$, $8^2_{13}$ and $8^3_5$, an additional construction is needed to rule out certain ``component exchange'' symmetries using satellites and the Jones polynomial. This shows that the polynomial invariants are powerful enough to compute $\Sigma(L)$ for these links. We give the first comprehensive list of $\Sigma(L)$ groups that we know of, though Hillman~\cite{MR867798} provides examples of various two-component links (including some split links) with symmetry groups equal to 12 different subgroups of $\Gamma_2$.

Why are the $\Sigma(L)$ groups interesting?  First, it is often more natural to consider the restricted group $\Sigma(L)$ than the generally larger $\Sym(L)$.  Sakuma~\cite{MR1006701} has shown that
a knot $K$ has a finite symmetry group if and only if $K$ is a hyperbolic knot, a torus knot, or a cable of a torus knot. Thus for many\footnote{Sakuma gives as an explicit example the untwisted double of a two-bridge knot.} knots, the group $\Sym(L)$ contains infinitely many elements which act nontrivially on the complement of $L$ but fix the link itself. We ignore such elements, which lie in the kernel of the natural homomorphism $\pi:\MCG(S^3,L) \rightarrow \MCG(S^3) \cross \MCG(L)$. In fact, even when $\MCG(S^3,L)$ is finite, we give various examples below where $\pi$ has nontrivial kernel. It is often difficult to describe an element of $\Sym(L)$ in $\ker \pi$ explicitly, but it is always simple to understand the exact meaning of the statement $\gamma \in \Sigma(L)$.

As an application, if one is interested in classifying knots and links up to oriented, labeled ambient isotopy, it is important to know the symmetry group $\Sigma(L)$ for each prime link type $L$, since links related by an element in $\Gamma(L)$ outside $\Sigma(L)$ are not (oriented, labeled) ambient isotopic. The number of different links related by an element of $\Gamma(L)$ to a given link of prime link type $L$ is given by the number of cosets of $\Sigma(L)$ in $\Gamma(L)$. If we count these cosets instead of prime link types, the number of actual knots and link types of a given crossing number is actually quite a bit larger than the usual table of prime knot and link types suggests. (See Table~\ref{tab:linktypes}.)

\begin{table}
\begin{tabular}{@{}rcccccccccc@{}}\toprule
&\multicolumn{9}{c} {Number of link types by components and crossings} \\
\cmidrule {2-11}
Crossings & 1-U & 1-OS & 2-U & 2-OS & 3-U & 3-OS & 4-U & 4-OS & All Links-U & All Links-OS \\
\midrule
0 & 1 &1& 1&1&  && && 1&1 \\
2 &  && 1&2&  && && 1&2 \\
3 & 1 &2 & &&  && && & \\
4 & 1 &1 &  1&4& &&&& 1&4\\
5 & 2 &4&   1&2& &&&& 1&2 \\
6 & 3 &5 &  3&10& 3&18&&& 6&44 \\
7 & 7 &14&  8 &38 & 1 &8&  &&  9&40 \\
8 & 21 &38 & 16 & 78& 10 & 200& 3& 120& 29& 398 \\
\bottomrule
\end{tabular}
\medskip
\caption[Number of isotopy classes for links, for both unoriented/unlabeled and oriented/labeled cases]{This table shows the number of distinct isotopy classes of links by crossing number and number of components in the link.  The columns labeled ``$n$-U'' count link types for $n$-component links in the usual way, where the components are unlabeled and unoriented, and the mirror image of a given link is considered to have the same link type, regardless of whether the two are ambient isotopic. The columns labeled ``$n$-OS'' give finer information, considering two $n$-component links to have the same type if and only if they are oriented, labeled ambient isotopic. As we can see, this stricter definition leads to a much larger collection of link types.
\label{tab:linktypes}
}
\end{table}

Second, $\Sigma(L)$ seems likely to be eventually relevant in applications. For instance, when studying DNA links, each loop of the link has generally has a unique sequence of base pairs which provide an orientation and an unambiguous labeling of each component of the link. In such a case, the question of whether two components in a link can be interchanged may prove to be of real significance.

Last, we are interested in the topic of tabulating composite knots and links. Since the connect sum of different symmetry versions of the same knot type can produce different knots (such as the square knot, which is the connect sum of a trefoil and its mirror image, and the granny knot, which is the connect sum of two trefoils with the same handedness), keeping track of the action of $\Sigma(L)$ is a crucial element in this calculation. We treat this topic in a forthcoming manuscript~\cite{composite}.

\section{The symmetry and intrinsic symmetry groups}

As we will describe below, the group $\MCG(S^3) \cross \MCG(L)$ was first studied by Whitten in 1969~\cite{MR0242146}, following ideas of Fox. They denoted this group $\Gamma(L)$ or $\Gamma_\mu$, where $\mu$ is the number of components of $L$. We can write this group as a semidirect product of $\Z_2$ groups encoding the orientation of each component of the link $L$ with the permutation group $S_\mu$ exchanging components of $L$ (cf. Definition~\ref{def:whittengroup}), finally crossed with another $\Z_2$ recording the orientation of $S^3$:
\begin{equation*}
\Gamma(L) = \Gamma_\mu = \Z_2 \cross (\Z_2^{\mu}\rtimes S_\mu).
\end{equation*}
It is clear that an element $\gamma = (\epsilon_0, \epsilon_1, \dots ,\epsilon_\mu, p) \in \Gamma(L)$ acts on $L$ to produce a new link $L^\gamma$. If $\epsilon_0 = +1$, then $L^\gamma$ and $L$ are the same as sets (but the components of $L$ have been renumbered and reoriented), while if $\epsilon_0 = -1$ the new link $L^\gamma$ is the mirror image of $L$ (again with renumbering and reorientation). We can then define the symmetry subgroup $\Sigma(L)$ by
\begin{equation*}
\gamma \in \Sigma(L) \iff \text{ there is an isotopy from } L \text { to } L^\gamma \text{ preserving component numbering and orientation.}
\end{equation*}

For knots, $\Sigma(L) < \Gamma_1 = \Z_2 \cross \Z_2$.  Here the five subgroups of $\Z_2 \cross \Z_2$ correspond to the standard descriptions of the possible symmetries for a knot, as shown in Table~\ref{tab:symmetrytypes}.

\begin{table}[ht]
\begin{tabular}{clc}
\toprule
Symmetry subgroup of $\Gamma_1$ & Name & Example(s)\\
\midrule
 $\{(1,1)\}$ & No symmetry  & $9_{32}, 9_{33}$ \\
$\{(1,1),(-1,1)\}$ & (+) amphichiral symmetry &  $12_{427}$ \\
$\{(1,1),(1,-1)\}$ & invertible\footnotemark[1]
symmetry & $3_1$ \\
$\{(1,1),(-1,-1)\}$ & (-) amphichiral\footnotemark[2]
symmetry  &  $8_{17}$ \\
$\Gamma_1$ & full symmetry &  $4_1$ \\
\bottomrule
\end{tabular}
\medskip
\caption[The 5 standard symmetry types for knots correspond to the 5 subgroups of the Whitten group $\Gamma_1$.]{The five standard symmetry types for knots correspond to the five subgroups of the Whitten group $\Gamma_1$.
\label{tab:symmetrytypes}
}
\end{table}
\footnotetext[1]{Conway~\cite{MR0258014} calls this \emph{reversible} symmetry.}
\footnotetext[2]{Conway~\cite{MR0258014} calls this \emph{invertible} symmetry.}
\addtocounter{footnote}{2}

For links, the situation is more interesting, as the group $\Gamma(L)$ is more complicated. In the case of two-component links, the group $\Gamma_2 = \Z_2 \cross (\Z_2 \cross \Z_2 \rtimes S_2)$ is a nonabelian 16 element group isomorphic to $\Z_2 \cross D_4$. The various subgroups of $\Gamma_2$ do not all have standard names, but we will call a link \emph{purely invertible} if $(1,-1,\dots,-1,e) \in \Sigma(L)$, and say that components $(i,j)$ have a \emph{pure exchange} symmetry if $(1,1,\dots,1,(ij)) \in \Sigma(L)$. For two-component links, we will say that $L$ has pure exchange symmetry if its two components have that symmetry. The question of which links have this symmetry goes back at least to Fox's 1962 problem list in knot theory~\cite[Problem 11]{MR0140100}\footnote{Actually, Fox was willing to call $L$ \emph{interchangeable} if $\Sigma(L)$ contained any element in the form $(\epsilon_0,\epsilon_1,\epsilon_2,(12))$, but we find it more useful to focus on the ``pure exchange'' symmetry.}. For example, $2^2_1$ (the Hopf link) has pure exchange symmetry while we will show that $7^2_4$ does not.

We similarly focus on the pure invertibility symmetry:  we show that 45 out of the 47 prime links with 8 crossings or less are purely invertible.  The two exceptions are $6^3_2$ (the Borromean rings) and $8^3_5$; we show that both of these are invertible using some nontrivial permutation, i.e., $(1,-1,-1,-1,p)\in\Sigma(L)$ for some $p\neq e$. (Whitten found examples of more complicated links which are not invertible even when a nontrivial permutation is allowed~\cite{MR0295331}.)

Increasing the number of components in a link greatly increases the number of possible types of symmetry.  Table~\ref{tab:conj} lists the number of subgroups of $\Gamma_\mu$; each different subgroup represents a different intrinsic symmetry group that a $\mu$-component link might have.  We note that if $\Sigma(L)$ is the symmetry subgroup of link $L$, then the symmetry subgroup of $L^\gamma$ is the conjugate subgroup $\Sigma(L^\gamma) = \gamma \Sigma(L) \gamma^{-1}$.  Therefore, it suffices to only examine the number of mutually nonconjugate subgroups of $\Gamma_\mu$ in order to specify all of the different intrinsic symmetry groups.  Table~\ref{tab:conj} also lists the number of conjugacy classes of subgroups of $\Gamma_\mu$, and the number of these which appear for prime links of 8 or fewer crossings.

\begin{table}[ht]
\begin{center}
\begin{tabular}{crrrr}
\toprule
&&& \# subgroups  & nonconjugate ones\\
$\mu$ & $\left|\Gamma_\mu\right|$ & \# subgroups & (up to conjugacy) & for $\leq 8$ crossings\\
\midrule
1 &      4        &       5    & 5 & 3\\
2  &    16       &       35 & 27   &  5                     \\
3  &     96      &       420 & 131 & 7                        \\
4  &     768    &     9417 &    994  & 3                        \\
5  &     7680  &      270131 &    6382 & 0                        \\
\bottomrule
\end{tabular}
\medskip
\caption[Number of subgroups of $\Gamma_\mu$; this equals the number of symmetry groups for a $\mu$-component link.] {The number of subgroups of $\Gamma_\mu$; each one represents a different intrinsic symmetry group possible for a $\mu$-component link.  The number of nonconjugate symmetries (the fourth column) is given by the number of conjugacy classes of subgroups..  This article computes the symmetry group for all prime links of 8 or fewer crossings; the last column summarizes our results.}
\label{tab:conj}
\end{center}
\end{table}

\section{Methods and Notation}

We started from the excellent table of oriented alternating links provided by Cerf~\cite{cerf}. Cerf considers the effect of reversing orientations of components of her links, but does not compute the effect of permutations. Our more detailed calculation does not depend on this information, so our paper also provides a check on Cerf's symmetry calculations. For alternating links, we have kept component numbers and orientations consistent with her table. For nonalternating links, we use component numbers and orientations consistent with those of Doll and Hoste~\cite{MR1094946}. Like Cerf, Doll and Hoste considered the effect of reversing the orientation of individual components of these links, but not the effect of permutations of components. We note that these component numbers and orientations are \emph{not} consistent with those used in~\texttt{SnapPea}. The component numbers and orientations in~\texttt{SnapPea} seem to have been chosen arbitrarily sometime in the $1980s$ by Joe Christy when he digitized the Rolfsen table~\cite{nmemail}.


To check our results against the \texttt{SnapPea} calculations, we used the Python interface provided by \texttt{SnapPy} to compute the image of $\pi:\MCG(S^3,L) \rightarrow \MCG(S^3) \cross \MCG(L)$. The results agreed\footnote{We had to redraw several links where the component numbering and orientation given by default in~\texttt{SnapPea} differed from our choice of numbering and orientation to make the groups agree. The default link data in~\texttt{SnapPea} results in several cases in subgroups conjugate to those we give below.} with the tables we give below, meaning that our results serve as an independent verification of \texttt{SnapPea}'s accuracy for these links. 

Finally, some notational comments:  we henceforth use the term `link' to refer to a prime, multicomponent link unless otherwise indicated.  We use the Rolfsen notation for links but provide the Thistlethwaite notation as well in our summary tables (Tables~\ref{tab:2bylink}, \ref{tab:3symm}, and \ref{tab:4symm}).

\section{The Whitten group}
We begin by giving the details of our construction of the Whitten group $\Gamma(L)$ and the symmetry group~$\Sigma(L)$. Consider operations on an oriented, labeled link $L$ with $\mu$ components. We may reverse the orientation of any of the components of $L$ or permute the components of $L$ by any element of the permutation group~$S_\mu$. However, these operations must interact with each as well: if we reverse component 3 and exchange components 3 and 5, we must decide whether the orientation is reversed before or after the permutation. Further, we can reverse the orientation on the ambient $S^3$ as well, a process which is clearly unaffected by the permutation. To formalize our choices, we follow~\cite{MR0242146} to introduce the Whitten group of a $\mu$-component link.
\begin{definition}
Consider the homomorphism given by
\begin{equation*}
\omega:S_\mu\longmapsto~\operatorname{Aut}(\mathbf{Z}_2^{\mu+1}),\hspace{20pt}p\longmapsto\omega(p)
\end{equation*}
where $\omega(p)$ is defined as
\begin{equation*}
\omega(p)(\epsilon_0,\epsilon_1,\epsilon_2...\epsilon_\mu)=(\epsilon_0,\epsilon_{p(1)},\epsilon_{p(2)}...\epsilon_{p(\mu)}).
\end{equation*}
\\For $\gamma=(\epsilon_0,\epsilon_1,...\epsilon_\mu, p),$ and $\gamma'=(\epsilon'_0,\epsilon'_1,...\epsilon'_\mu, q)\in \mathbf{Z}_2^{\mu+1}\rtimes_\omega S_\mu$, we define the \emph{Whitten group} $\Gamma_\mu$ as the semidirect product $\Gamma_\mu=\mathbf{Z}_2^{\mu+1}\rtimes_\omega S_\mu$
with the group operation
\begin{align*}
\gamma\ast\gamma'&=\left(\epsilon_0,\epsilon_1,\epsilon_2...\epsilon_\mu,p\right)\ast (\epsilon'_0,\epsilon'_1,\epsilon'_2...\epsilon'_\mu,q)\\
& =((\epsilon_0,\epsilon_1,\epsilon_2...\epsilon_\mu)\cdot\omega(p)(\epsilon'_0,\epsilon'_1,\epsilon'_2...\epsilon'_\mu),qp)\\
&=(\epsilon_0\epsilon'_0,\epsilon_1\epsilon'_{p(1)},\epsilon_2\epsilon'_{p(2)}...\epsilon_\mu\epsilon'_{p(\mu)},qp)
\end{align*}
We will also use the notation $\Gamma(L)$ to refer to the Whitten group $\Gamma_\mu$.
\label{def:whittengroup}
\end{definition}

\subsection{Link operations}
Given a link $L$ consisting of $\mu$ oriented knots in $S^3$, we may order the knots and write
\begin{equation*}
L=K_1\cup K_2\cup \cdots \cup K_\mu.
\end{equation*}
Consider the following operations on $L$:
\begin{enumerate}
\item Permuting the $K_i$.
\item Reversing the orientation of any set of $K_i$'s
\item Reversing the orientation on $S^3$ (mirroring $L$).
\end{enumerate}
Let $\gamma$ be a combination of any of the moves (1), (2), or (3). We think of $\gamma=(\epsilon_0,\epsilon_1,...\epsilon_\mu, p)$ as an element of the set $\mathbf{Z}_2^{\mu+1}\times S_\mu$ in the following way. Let
\begin{equation*}\epsilon_0=
\begin{cases}
-1, &\text{if $\gamma$ mirrors $L$}\\
+1, &\text{if $\gamma$ does not mirror $L$}
\end{cases}\end{equation*}
and
\begin{equation*}\epsilon_i=
\begin{cases}
-1, &\text{if $\gamma$ reverses the orientation of $K_{p(i)}$}\\
+1, &\text{if $\gamma$ does not reverse the orientation of $K_{p(i)}$}
\end{cases}\end{equation*}
Lastly, let $p\in S_\mu$ be the permutation of the $K_i$ associated to $\gamma$.  To be explicitly clear, permutation $p$ permutes the labels of the components; the component originally labeled $i$ will be labeled $p(i)$ after the action of $\gamma$.  

For each element, $\gamma$ in $\mathbf{Z}_2^{\mu+1}\times S_\mu$, we define
\begin{equation}
\label{eq:Lgamma}
L^\gamma=\gamma(L)=\epsilon_1K_{p(1)}^{(*)}\cup \epsilon_2K_{p(2)}^{(*)}\cup\cdot\cdot\cdot \cup \epsilon_\mu K_{p(\mu)}^{(*)} =\bigcup_{i=1}^{\mu}\epsilon_i K_{p(i)}^{(*)}
\end{equation}
\\where $-K_i$ is $K_i$ with orientation reversed, $K_i^*$ is the mirror image of $K_i$ and the $(*)$ appears above if and only if $\epsilon_0 = -1$. Note that the $i$th component of $\gamma(L)$ is $\epsilon_i K_{p(i)}^{(*)}$ the possibly reversed or mirrored $p(i)$th component of $L$.  Since we are applying $\epsilon_i$ instead of $\epsilon_{p(i)}$ to $K_{p(i)}$ we are taking the convention of first permuting and then reversing the appropriate components.

\begin{example}
Let $L=K_1\cup K_2\cup K_3$ and $\gamma=(1,1,-1,1,(123))$.  Then, $\gamma(L)=K_2\cup -K_3\cup K_1$.  
\end{example}

\begin{example}Let $L=K_1\cup K_2\cup K_3\cup K_4$ and $\gamma=(-1, 1,1,-1,-1, (14)(23))$.
Then, $\gamma(L)=(K_4^*\cup K_3^*\cup -K_2^* \cup -K_1^*)$.  Since we have reversed the orientation on $S^3$, note that $\gamma(L)$ will be the mirror image of $L$ as well.
\end{example}

We now confirm that this operation defines a group action of the Whitten group $\Gamma(L)$ on the set of links obtained from $L$ by such transformations.

\begin{proposition}
The Whitten group $\Gamma(L)$ is isomorphic to the group $\MCG(S^3) \cross \MCG(L)$.
\end{proposition}

\begin{proof}
We know that $L$ is a disjoint union of $\mu$ copies of $S^1$ denoted $L = K_1 \sqcup \cdots \sqcup K_\mu$. Further, the mapping class groups of $S^1$ and $S^3$ are both $\Z_2$, where the elements $\pm 1$ correspond to orientation preserving and reversing diffeomorphisms of $S^1$ and $S^3$. In general, the mapping class group of $\mu$ disjoint copies of a space is the semidirect product of the individual mapping class groups with the permutation group $S_\mu$. This means that $\MCG(L) = (\Z_2)^{\mu} \rtimes S_\mu$ and the Whitten group $\Gamma(L)$ has a bijective map to $\MCG(S^3) \cross \MCG(L)$.

It remains to show that the group operation $*$ in the Whitten group maps to the group operation (composition of maps) in $\MCG(S^3) \cross \MCG(L)$. To do so, we introduce some notation.
Let $\gamma=(\epsilon_0,\epsilon_1,...\epsilon_\mu, p)$ and $\gamma'=(\epsilon'_0,\epsilon'_1,...\epsilon'_\mu, q)$. We must show
\begin{equation*}
L^{\gamma\ast\gamma'}=(\gamma \circ \gamma')(L)=\gamma(\gamma'(L))
\end{equation*}
where $\gamma\ast\gamma'$ is the operation of the Whitten Group, $\Gamma_\mu$.

Then,
\begin{equation*}
\gamma'(L)=\gamma'\left(\bigcup_{i=1}^{\mu}K_i\right)=\bigcup_{i=1}^{\mu} \epsilon'_iK_{q(i)}=\bigcup_{i=1}^{\mu}X_i,
\end{equation*}
where $X_i$ denotes the $i$-th component $\epsilon_i' K_{q(i)}$ of $\gamma'(L)$.
\begin{equation*}
\gamma(\gamma'(L))=\gamma\left(\bigcup_{i=1}^{\mu}\epsilon'_iK_{q(i)}\right) =\gamma\left(\bigcup_{i=1}^{\mu}X_i\right)=\bigcup_{j=1}^{\mu}\epsilon_i X_{p(i)} .
\end{equation*}
Note that $X_{p(i)}=\epsilon'_{p(i)}K_{q(p(i))}$, which implies
\begin{equation*}
\gamma(\gamma'(L))=\bigcup_{i=1}^{\mu}\epsilon_i (\epsilon'_{p(i)} K_{q(p(i))}).\end{equation*}
Now, $\gamma\ast\gamma'=(\epsilon_0\epsilon'_0,\epsilon_1\epsilon'_{p(1)},\epsilon_2\epsilon'_{p(2)}, ..., \epsilon_\mu\epsilon'_{p(\mu)},qp)$ and acts on $L$ as
\begin{align*}
L^{\gamma\ast\gamma'} &=\epsilon_1\epsilon'_{p(1)}K_{qp(1)}\, \cup \,  \epsilon_2\epsilon'_{p(2)}K_{qp(2)}\, \cup \ldots \cup \,  \epsilon_\mu\epsilon'_{p(\mu)}K_{qp(\mu)} \\
L^{\gamma\ast\gamma'} & =\bigcup_{i=1}^{\mu}\epsilon_{p(i)} (\epsilon'_i K_{qp(i)}) =\gamma(\gamma'(L))
\end{align*}

We have dropped the notation for mirroring throughout the proof, because the two links clearly agree in this regard.  The element $\gamma \ast \gamma'$ preserves the orientation of $S^3$ if and only if $\epsilon_0 \epsilon_0' =1$, i.e., if either both or neither of $\gamma$ and $\gamma'$ mirror $L$.
\end{proof}

We can now define the subgroup of $\Gamma(L)$ which corresponds to the symmetries of the link $L$.

\begin{definition}Given a link, $L$ and $\gamma\in \Gamma(L)$, we say that \emph{$L$ admits $\gamma$} when there exists an isotopy taking each component of $L$ to the corresponding component of $L^{\gamma}$ which respects the orientations of the components.  We define as the \emph{Whitten symmetry group} of $L$,
\begin{equation*}
\Sigma(L):=\{\gamma\in \Gamma(L) |~\text{$L$ admits $\gamma$}\}.
\end{equation*}
\end{definition}

The Whitten symmetry group $\Sigma(L)$ is a subgroup of $\Gamma_\mu$, and its left cosets represent the different isotopy classes of links $L^\gamma$ among all symmetries $\gamma$.  By counting the number of cosets, we determine the number of (labeled, oriented) isotopy classes of a particular prime link.  
%

Next, we provide a few examples of symmetry subgroups.  Recall that the first Whitten group $\Gamma_1 = \Z_2 \times \Z_2$ has order four and that $\Gamma_2 = \Z_2 \cross (\Z_2 \cross \Z_2 \rtimes S_2)$ is a nonabelian 16 element group.  

\begin{example}
Let $L=4_1$, the figure eight knot. Since $L\sim-L\sim L^*\sim -L^*$, we have $\Sigma(4_1)=\Gamma_1$, so the figure eight knot has \emph{full symmetry}. There is only one coset of $\Sigma(4_1)$ and hence only one isotopy class of $4_1$ knots.
\end{example}

\begin{example} Let $L=3_1$, a trefoil knot. It is well known that $L\sim-L$ and $L^*\sim-L^*$, but $L\nsim L^*$, so we have $\Sigma(3_1)=\{(1,1, \id), (1,-1, e)\}$. This means that the two cosets of $\Sigma(3_1)$ are $\{(1,1,\id), (1,-1, \id)\}$ and $\{(-1,-1, \id), (-1,1,\id)\}$, and there are two isotopy classes of $3_1$ knots.  A trefoil knot is thus \emph{invertible}.
\end{example}

\begin{example}
Let $L=7^2_5$, whose components are an unknot $K_1$ and a trefoil $K_2$.  In section~\ref{sec2a}, we determine all symmetry groups for two-component links, but we provide details for $7^2_5$ here. Since the components $K_1$ and $K_2$ are of different knot types, we conclude that no symmetry in $\Sigma(7^2_5)$ can contain the permutation $(12)$.  Since $K_2 \nsim K_2^*$, we cannot mirror $L$, i.e., the first entry of $\gamma \in \Sigma(7^2_5)$ cannot equal $-1$.  The linking number of $L$ is nonzero, so we can rule out the symmetries $(1,-1,1,e)$ and $(1,1,-1,e)$ by Lemma~\ref{lemma:selfwrithe}.  Last, $L$ is purely invertible, meaning isotopic to $-L=-K_1\cup-K_2$. Thus, $\Sigma(L)$ is the two element group $\s{2}{1}=\{(1,1,1,e), (1,-1,-1,e)\}$.  There are 8 cosets of this two element group in the 16 element group $\Gamma_2$, so there are 8 isotopy classes of $7^2_5$ links.  
\end{example}

We now prove
\begin{proposition}
\label{prop:image}
The Whitten symmetry group $\Sigma(L)$ is the image of $\Sym(L)$ under the map $\pi \co\, \Sym(L) = \MCG(S^3,L) \rightarrow \MCG(S^3) \cross \MCG(L) = \Gamma(L)$.
\end{proposition}

\begin{proof}
Given a map $f \co S^3 \rightarrow S^3 \in \MCG(S^3,L)$, we see that if $f$ is orientation-preserving on $S^3$, then it is homotopic to the identity on $S^3$ since $\MCG(S^3) = \Z_2$. This homotopy yields an ambient isotopy between $L$ and $f(L)$, proving that $f|_L = \pi(f) \in \Sigma(L)$. If $f$ is orientation-reversing on $S^3$, it is homotopic to a standard reflection $r$. Composing the homotopy with $r$ provides an ambient isotopy between $L$ and $r(f(L))$, proving that $\pi(f) \in \Sigma(L)$. This shows $\pi(\Sym(L)) \subset \Sigma(L)$.

Now suppose $g \in \Sigma(L)$. The isotopy from $L$ to $g(L)$ generates an orientation-preserving (since it is homotopic to the identity) diffeomorphism $f \co S^3 \rightarrow S^3$ which either fixes $L$ or takes $L$ to $rL$. In the first case, $f \in\Sym(L)$. In the second, the map $rf \in \Sym(L)$.
\end{proof}

\section{The linking matrix}
For each link $L$, our overall strategy will be to explicitly give isotopies for certain elements of the symmetry subgroup $\Sigma(L)$, generate the subgroup containing those elements, and then rule out the remainder of $\Gamma(L)$ using invariants. For three- and four-component links, a great deal of information about $\Sigma(L)$ can usually be obtained by considering the collection of pairwise linking numbers of the components of the link.

We recall a few definitions:

\begin{definition}
Given a $n$-component link, with components $L_1,~L_2,\dots,L_n$, we let the $n \times n$ \emph{linking matrix} of $L$ be the matrix $\Lk(L)$ so that $\Lk(L)_{ij} = \Lk(L)_{ji} = \Lk(L_i,L_j)$ and $\Lk(L)_{ii} = 0$ where $\Lk(L_i,L_j)$ is the linking number of $L_i$ and $L_j$. We let $\Lk(n)$ denote the set of $n\cross n$ symmetric, integer-valued matrices with zeros on the diagonal.
\end{definition}

The linking number can also be computed by counting the signed crossings of one knot over another.  
Among minimal crossing number diagrams of alternating links, the following three numbers are also useful link invariants:
\begin{definition}
The (overall) \emph{linking number} $\ell k(L)$ of $L$ is half the sum of the entries of the linking matrix. (This is half of the intercomponent signed crossings of $L$.) The \emph{writhe} $\wr(L)$ is the sum of all signed crossings of $L$. The \emph{self-writhe} $s(L)$ is the sum of the intracomponent signed crossings of $L$.  Clearly, $\wr = 2\ell k + s$.
\end{definition}

Murasugi and Thistlethwaite separately showed that writhe was an invariant of reduced alternating link diagrams \cite{MR898151, MR963633}.  Since linking number is an invariant for all links, self-writhe is also an invariant of reduced alternating diagrams.  We will utilize this invariance to rule out certain symmetries of links, cf. Lemma~\ref{lemma:selfwrithe}. 

The Whitten group $\Gamma_n$ acts on the set of linking matrices $\Lk(n)$. Further, for a given link $L$, the symmetry subgroup $\Sigma(L)$ must be a subgroup of the stabilizer of $\Lk(L)$ under this action. This means that it is worthwhile for us to understand this action and make a classification of linking matrices according to their orbit types. We start by writing down the action:

\begin{proposition}\label{prop:linkingmatrix}
The action of the Whitten group $\Gamma_n$ on $n$-component links gives rise to the following $\Gamma_n$ action on the set of $n \cross n$ linking matrices $\Lk(n)$:
\begin{equation*}
\Lk(\gamma(L))_{ij} = \epsilon_0 \epsilon_i \epsilon_j \Lk(L)_{p(i) p(j)}.
\end{equation*}
\end{proposition}

\begin{proof}
Equation~\eqref{eq:Lgamma} reminds us that
\begin{equation*}
\gamma(L)=\epsilon_1K_{p(1)}^{(*)}\cup \epsilon_2K_{p(2)}^{(*)}\cup\cdot\cdot\cdot \cup \epsilon_\mu K_{p(\mu)}^{(*)} =\bigcup_{i=1}^{\mu}\epsilon_i K_{p(i)}^{(*)}
\end{equation*}
This means that the $i$th component of $\gamma(L)$ is component $p(i)$ of $L$.  Recall that linking number is reversed by changing the orientation of either curve or the ambient $S^3$, which proves that we should multiply by $\epsilon_0 \epsilon_i \epsilon_j$ as claimed.
\end{proof}
\begin{corollary}\label{cor:3matrix}
If $\gamma = (\epsilon_0,\epsilon_1,\dots,\epsilon_3,p)$, let $\epsilon = \epsilon_0 \epsilon_1 \epsilon_2 \epsilon_3$. The action of $\Gamma_3$ on $\Lk(3)$ can be written
\begin{equation*}
\Lk(\gamma(L))_{ij} = \epsilon_k \epsilon \Lk(L)_{p(i)p(j)}.
\end{equation*}
where $\{i,j,k\} = \{1,2,3\}$ as sets.
\end{corollary}

In principle, this description of the action provides all the information one needs to compute orbits and stabilizers for any given matrix (for instance, by computer). However, that brute force approach doesn't yield much insight into the structure of the problem. We now develop enough theory to understand the situation without computer assistance in the case of three-component links.

\subsection{Linking matrix for three-component links}

We first observe that there is a bijection between $\Lk(3)$ and $\Z^3$ given by
\begin{equation}
\label{eq:correspondance}
\left(
\begin{array}{lll}
0 & z_3 & z_2 \\
z_3 & 0 & z_1 \\
z_2 & z_1 & 0
\end{array}
\right) \leftrightarrow (z_1, z_2, z_3)
\end{equation}
We would like to understand the action of $\Gamma_3$ on $\Lk(3)$ by reducing it to the natural action of the simpler group $(\Z_2)^3 \rtimes S_3$ on $\Z^3$.

\begin{proposition}
\label{prop:fhom}
The action of $\Gamma_3$ on $\Lk(3)$ descends to the natural action of $(\Z_2)^3 \rtimes S_3$ on $\Z^3$ via the surjective homomorphism $\Gamma_3 \rightarrow (\Z_2)^3 \rtimes S_3$ defined by 
\begin{equation*}
f: (\epsilon_0,\epsilon_1,\epsilon_2,\epsilon_3,p) \mapsto (\epsilon_1 \epsilon,\epsilon_2 \epsilon,\epsilon_3 \epsilon, p)
\end{equation*}
where
\begin{equation*}
\epsilon = \epsilon_0 \epsilon_1 \epsilon_2 \epsilon_3.
\end{equation*}
\end{proposition}

\begin{proof}
We first check that $f(\gamma \ast \gamma') = f(\gamma) \ast f(\gamma')$. Now $\gamma\ast\gamma'=(\epsilon_0\epsilon'_0,\epsilon_1\epsilon'_{p(1)},\epsilon_2\epsilon'_{p(2)}...\epsilon_\mu\epsilon'_{p(\mu)},qp)$. This means that
\begin{align*}
f(\gamma \ast \gamma') &= (\epsilon_1 \epsilon'_{p(1)} \epsilon_0 \epsilon'_0 \epsilon_1 \epsilon'_{p(1)} \epsilon_2 \epsilon'_{p(2)} \epsilon_3 \epsilon'_{p(3)}, \dots, qp) \\
&= (\epsilon_1 \epsilon'_{p(1)} \epsilon \epsilon', \epsilon_2 \epsilon'_{p(2)} \epsilon \epsilon',
\epsilon_3 \epsilon'_{p(3)} \epsilon \epsilon',qp),
\end{align*}
since $\epsilon'_{p(1)}\epsilon'_{p(2)}\epsilon'_{p(3)} = \epsilon'_1\epsilon'_2\epsilon'_3$ for any permutation $p$. But
\begin{align*}
f(\gamma) \ast f(\gamma') &= (\epsilon_1 \epsilon,\dots,\epsilon_3 \epsilon,p) \ast (\epsilon'_1 \epsilon', \dots, \epsilon'_3 \epsilon',q) \\
&= (\epsilon_1 \epsilon \, \epsilon'_{p(1)} \epsilon', \dots, \epsilon_3 \epsilon \, \epsilon'_{p(3)} \epsilon',qp).
\end{align*}
This proves that $f$ is a homomorphism.
In order to show that $f$ is surjective we will compute the kernel. 
Consider the image  $(\delta_1, \delta_2, \delta_3, p) = f(\epsilon_0,\epsilon_1,\epsilon_2,\epsilon_3,p)$.  Then the product $\delta_1 \delta_2 \delta_3 = \epsilon_0$ since $\delta_1 \delta_2 \delta_3 = \epsilon_1 \epsilon \, \epsilon_2 \epsilon \, \epsilon_3 \epsilon = \epsilon_0^3 \epsilon_1^4 \epsilon_2^4 \epsilon_3^4 = \epsilon_0$.

Note that direct computation shows that the preimage of a general element $(\delta_1, \delta_2, \delta_3, p)$ is given by $(\delta_1 \delta_2 \delta_3, \delta_1, \delta_2, \delta_3, p)$ and $(\delta_1 \delta_2 \delta_3, -\delta_1, -\delta_2, -\delta_3, p)$.
So suppose that $f(\epsilon_0,\epsilon_1,\epsilon_2,\epsilon_3,p) = (1,1,1,e)$.   It is clear that $p = e$ and $\epsilon_0=1$.  Since $\epsilon_1 \epsilon = \epsilon_2 \epsilon = \epsilon_3 \epsilon = 1$, if $\epsilon_1 = 1$, then $\epsilon_2 = \epsilon_3 = 1$ and $\gamma = (1,1,1,1,e)$.  Likewise, if $\epsilon_1 = -1$, then $\epsilon_2 = \epsilon_3 = -1$ as well, and $\gamma = (1,-1,-1,-1,e)$.

Since $\Gamma_3$ is a group of order 96 and the kernel of $f$ has order 2, the image of $f$ has order 48. Since the target group $(\Z_2)^3 \rtimes S_3$ also has order 48, we conclude that $f$ is surjective, as claimed. 

 By Corollary~\ref{cor:3matrix}, the $\Gamma_3$ action on $\Lk(3)$ maps each entry $z_k = \Lk(L)_{ij}$ in the linking matrix to $\epsilon \epsilon_k \Lk(L)_{p(i)p(j)}=\epsilon \epsilon_k \, z_{p(k)}$, i.e.,
\begin{equation} \label{eq:actionproof}
\gamma * \left(
\begin{array}{ccc}
0 & z_3 & z_2 \\
z_3 & 0 & z_1 \\
z_2 & z_1 & 0
\end{array}
\right) = 
\left(
\begin{array}{ccc}
0 & \epsilon \epsilon_3 \, z_{p(3)} & \epsilon \epsilon_2 \, z_{p(2)} \\
\epsilon \epsilon_3 \, z_{p(3)} & 0 & \epsilon \epsilon_1 \, z_{p(1)} \\
\epsilon \epsilon_2 \, z_{p(2)} & \epsilon \epsilon_1  \, z_{p(1)} & 0
\end{array}
\right).
\end{equation}

By the definition of $f$, the natural action of $(\Z_2)^3 \rtimes S_3$ on $\Z^3$ is
\begin{equation*}
f(\gamma) * (z_1, z_2, z_3) = (\epsilon_1 \epsilon \, z_{p(1)}, \, \epsilon_2 \epsilon \, z_{p(2)}, \, \epsilon_3 \epsilon \, z_{p(3)}).
\end{equation*}
This triple corresponds precisely to the new linking matrix \eqref{eq:actionproof} obtained from the $\Gamma_3$ action, so we have shown the two actions correspond.
\end{proof}

We are now in a position to classify $3 \cross 3$ linking matrices according to their orbit types, and compute their stabilizers in $\Gamma_3$. The stabilizer of a linking matrix $A \in \Lk(3)$ as a subgroup of $\Gamma(3)$ under the group action of Proposition~\ref{prop:linkingmatrix} is the preimage under the homomorphism $f$ of Proposition~\ref{prop:fhom} of the stabilizer of the corresponding triple $(z_1, z_2, z_3) \in \Z^3$ under the natural action of $(\Z_2)^3 \rtimes S_3$ on $\Z^3$. Since the kernel of $f$ has order $2$, stabilizers in $\Gamma_3$ are twice the size of the corresponding stabilizers in $(\Z_2)^3 \rtimes S_3$.

There are 10 orbit types of triples $(z_1, z_2, z_3)$ under this action. To list the orbit types, we write a representative triple in terms of variables $a$, $b$, and $c$ which are assumed to be integers with distinct nonzero magnitudes. To list the stabilizers, we either give the group explicitly as a subgroup of $(\Z_2)^3 \rtimes S_3$ or provide a list of generators in the form $\left< g_1, g_2, \dots, g_n \right>$.  One of these groups, $S(a, a,-a)$, is more complicated and is described below.


\begin{table}[ht]
\begin{tabular}{llllll}
\toprule
$(z_1,z_2,z_3)$ & Stabilizer in $(\Z_2)^3 \rtimes S_3$ & Stab. in $\Gamma_3$ & Order\\
\midrule
$(0,0,0)$ & $(\Z_2)^3 \rtimes S_3$ & $\Gamma_3$ & 96 \\
$(a,0,0)$ & $(\{+1\} \cross \Z_2 \cross \Z_2) \rtimes \{e,(23)\}$ & $D_4 \cross \Z_2$ & 16 \\
$(a,-a,0)$ & $\left< (1,1,-1,e), (-1,-1,1,(12)) \right>$ & $(\Z_2)^3$ & 8 \\
$(a,a,0)$ & $(\{+1\} \cross \{+1\} \cross \Z_2 ) \rtimes \{e,(12)\}$ & $(\Z_2)^3$ & 8 \\
$(a,b,0)$ & $(\{+1\} \cross \{+1\} \cross \Z_2) \rtimes \{e\}$ & $D_2$ & 4 \\
$(a,a,-a)$ & $S(a,a,-a)$ & $\Z_2 \times S_3$ & 12 \\
$(a,a,a)$ & $(\{+1\} \cross \{+1\} \cross \{+1\}) \rtimes S_3$ & $\Z_2 \times S_3$ & 12 \\
$(a,b,-b)$ & $\left< (1,-1,-1,(23)) \right>$ & $D_2$ & 4 \\
$(a,b,b)$ & $(\{+1\}\cross \{+1\} \cross \{+1\}) \rtimes \{e,(23)\}$ & $D_2$ & 4 \\
$(a,b,c)$ & $\left\{ (1,1,1,e) \right\}$ & $D_1$ & 2 \\
\bottomrule
\end{tabular}
\smallskip
\caption[Stabilizers for linking matrices in the group $(\Z_2)^3 \rtimes S_3$]{This table gives stabilizers for triples in $\Z^3$ in the group $(\Z_2)^3 \rtimes S_3$ under the natural action of this group on $\Z^3$. This list of examples covers all the orbit types of this action. As we have shown above, the preimages of these stabilizers in $\Gamma_3$ are the stabilizers of the corresponding linking matrices in $\Lk(3)$. For convenience, the order of these preimages are given in the right-hand column of the table.
\label{tab:stabs}}
\end{table}

The group $S(a,a,-a)$ is a $6$ element group isomorphic to $S_3$ (or $D_3$) given by
\begin{equation*}
S(a,a,-a) = \left\{
\begin{array}{lll}
(1,1,1,e) & (1,1,1,(12)) & (1,-1,-1,(23)) \\
(-1,1,-1,(13)) & (1,-1,-1,(123)) & (-1,1,-1,(132))
\end{array} \right\}
\end{equation*}

Using the preimage formula in the proof of Proposition~\ref{prop:fhom}, it is easy to compute the stabilizer of a given linking matrix in $\Gamma_3$ directly from the table above; we simply conjugate by a permutation to bring the corresponding triple into one of the forms above and then apply the preimage formula.

We can now draw some amusing conclusions which might not be obvious otherwise, such as:

\begin{lemma}
If $L$ is a three-component link and any element of $\Sigma(L)$ reverses orientation on $S^3$ then at least one pair of components of $L$ has linking number zero.
\end{lemma}

\begin{proof}
The stabilizer of $\Lk(L)$ includes an element of the form $(-1,\epsilon_1,\epsilon_2,\epsilon_3, p)$ if and only if some element $(\delta_1,\delta_2,\delta_3,p)$ in the stabilizer of the corresponding triple $(z_1,z_2,z_3)$ has $\delta_1 \delta_2 \delta_3 = -1$ since we showed in the proof of Proposition~\ref{prop:fhom} that $\epsilon_0$ equaled $\delta_1 \delta_2 \delta_3$.  

A negative $\delta_i$ will switch the sign of the linking number $z_{p(i)}$; to stabilize the triple $(z_1,z_2,z_3)$ there must be an even number of sign changes unless some $z_i=0$.  Hence, if $L$ has some mirror symmetry, i.e., one with $\epsilon_0=-1$, then $\delta_1 \delta_2 \delta_3 = -1$ which produces an odd number of sign changes, so some linking number $z_i=0$.
\end{proof}

\begin{example} \label{ex:731stab}
We will see that the linking matrix for $7^3_1$ has corresponding triple in the form 
$(a,a, -a)$. This means that the stabilizer of this linking matrix is  
a group of order 12 isomorphic to $\Z_2 \cross S_3$ in $\Gamma_3$ conjugate to the preimage of the stabilizer $S(a,a,-a)$.  Using the preimage formula of Proposition~\ref{prop:fhom}, we can explicitly compute 
\begin{equation*}
f^{-1}(S(a,a,-a)) = \left\{
\begin{array}{lll}
(1,1,1,1,e) & (1,1,1,1,(12)) & (1,1,-1,-1,(23)) \\
(1,-1,-1,-1,e) & (1,-1,-1,-1,(12)) & (1,-1,1,1,(23) \\
(1,-1,1,-1,(13)) & (1,1,-1,-1,(123)) & (1,-1,1,-1,(132)) \\
(1,1,-1,1,(13)) & (1,-1,1,1,(123)) & (1,1,-1,1,(132))
\end{array}
\right\}.
\label{eq:731stab}
\end{equation*}

Conjugating this subgroup by $(1,1,1,1,(13))$, we obtain the stabilizer of $\Lk(7^3_1)$:
\begin{equation*}
\left\{
\begin{array}{lll}
(1,1,1,1,e) & (1,1,1,1,(23)) & (1,-1,-1,1,(12)) \\
(1,-1,-1,-1,e) & (1,-1,-1,-1,(23)) & (1,1,1,-1,(12) \\
(1,-1,1,-1,(13)) & (1,-1,-1,1,(132)) & (1,-1,1,-1,(123)) \\
(1,1,-1,1,(13)) & (1,1,1,-1,(132)) & (1,1,-1,1,(123))
\end{array}
\right\}
\end{equation*}
We know that $\Sigma(7^3_1)$ is a subgroup of this stabilizer; actually, it equals the stabilizer, which we show in section~\ref{claim:731}.
\end{example}

\medskip
\subsection{Linking matrix for four-component links}
For four-component links, we will need to develop a different observation.  There are only three prime four-component links with 8 or fewer crossings, and fortunately they all possess a particular type of linking matrix.  While the situation seems too complicated to make a full analysis of the $S_4$ action on the 6 nonzero elements of a general $4 \cross 4$ linking matrix, it is relatively simple to come up with a theory which covers our cases.  We first give a correspondence between certain $4 \cross 4$ linking matrices and elements of $\Z^4$:
\begin{equation}
\left(
\begin{array}{cccc}
0 & z_1 & z_3 & 0 \\
z_1 & 0 & 0 & z_2 \\
z_3 & 0 & 0 & z_4 \\
0   & z_2 & z_4 & 0
\end{array}
\right)
\leftrightarrow
(z_1,z_2,z_3,z_4)
\label{eq:gamma4tolink4}
\end{equation}
Equivalently, we let
\begin{equation*}
z_1 = A_{12}, \quad z_2 = A_{24}, \quad z_3 = A_{31}, \quad z_4 = A_{43}.
\end{equation*}
If we let $\Gamma_4$ act on the matrix $A_{ij}$ as usual, matrices in this form are fixed by the subgroup with permutations in an 8 element subgroup of $S_4$ isomorphic to $D_4$ which we will call $G_0$. 

\begin{proposition}
\label{prop:gamma4}
Let $G_0$ denote the subgroup $\{e,(14),(23),(14)(23),(12)(34),(13)(24),(1243),(1342)\}$ (isomorphic to $D_4$) of $S_4$. The action of the subgroup $G = \{ (\epsilon_0, \epsilon_1, \dots, \epsilon_4,p) | p \in G_0 \} < \Gamma_4$ on $\Lk(4)$ descends to the natural action of $(\Z_2)^4 \rtimes S_4$ on $\Z^4$ via the homomorphism
\begin{equation*}
f: (\epsilon_0, \epsilon_1, \dots, \epsilon_4,p) \mapsto (\epsilon_0 \epsilon_1 \epsilon_2, ~\epsilon_0 \epsilon_2 \epsilon_4, ~\epsilon_0 \epsilon_1 \epsilon_3, ~\epsilon_0 \epsilon_3 \epsilon_4, ~f_0(p)).
\end{equation*}
where $f_0:G_0 \rightarrow G_0$ is defined as
\begin{equation*}
\begin{array}{llll}
e \mapsto e & \quad (14) \mapsto (12)(34) \quad & \quad (13)(24) \mapsto (14) \quad & (1243) \mapsto (1243) \\
 & \quad (23) \mapsto (13)(24) & \quad (12)(34) \mapsto (23) & (14)(23) \mapsto (14)(23) \\
 &  & & (1342) \mapsto (1342)
\end{array}
\end{equation*}
\end{proposition} 

\begin{proof}
A series of easy but lengthy direct computations show that this homomorphism has a kernel of order 4 given by
\begin{equation*}
\operatorname{ker}(f) = \{(1,1,1,1,1,e), (1,-1,-1,-1,-1,e), (-1,-1,1,1,-1,e), (-1,1,-1,-1,1,e)\},
\end{equation*}
and the image of $f$ in $(\Z_2)^4 \rtimes S_4$ is $\{ (\delta_1,\delta_2,\delta_3,\delta_4,f_0(p)): ~\delta_4 = \delta_1\delta_2\delta_3, p \in G_0 \}$.  The preimage of such an element is the 4 element set $\{(\epsilon_1\epsilon_2\delta_1,  ~\epsilon_1, ~\epsilon_2, ~\epsilon_2 \delta_1 \delta_3,  ~\epsilon_1 \delta_1\delta_2, ~p)\}$, where $\epsilon_1$ and $\epsilon_2$ are arbitrary.

To show that the action of $G$ on $\Lk(4)$ descends to the natural action of $(\Z_2)^4 \rtimes S_4$ on $\Z^4$, i.e.
\begin{equation}
\gamma * \left(
\begin{array}{cccc}
0 & z_1 & z_3 & 0 \\
z_1 & 0 & 0 & z_2 \\
z_3 & 0 & 0 & z_4 \\
0   & z_2 & z_4 & 0
\end{array}
\right)
\leftrightarrow
f(\gamma) * (z_1,z_2,z_3,z_4)
\label{eq:gamma4tolink4action}
\end{equation}
we need only check that
\begin{equation*}
z_{f_0(p)(1)} = A_{p(1)p(2)}, \quad z_{f_0(p)(2)} = A_{p(2)p(4)}, \quad z_{f_0(p)(3)} = A_{p(3)p(1)}, \quad z_{f_0(p)(4)} = A_{p(4)p(3)}
\end{equation*}
for all $p \in G_0$. This is another straightforward, if lengthy, computation.
\end{proof}

We now need to find stabilizers for a few carefully chosen linking matrix types.

\begin{lemma}
\label{lem:stab4}
The stabilizers in $f(G) < (\Z_2)^4 \rtimes S_4$ of $(a,a,a,a)$, $(a,a,-a,a)$, and $(a,-a,-a,a)$ are all 8 element groups isomorphic to $D_4$.  These produce 32-element stabilizers for the corresponding linking matrices in $\Gamma_4$ isomorphic to $\Z_2 \cross \Z_2 \cross D_4$. The individual stabilizers in $f(G)$ are in the form $(\delta_1, \dots, \delta_4, p)$ where the $\delta_i$ are determined uniquely by $p$. For $(a,a,a,a)$, the $\delta_i = 1$. In the other two cases, the pattern of signs is more intricate. We give the subgroups explicitly below in Table~\ref{tab:4stabs}.

\medskip

\begin{table}[ht]
\begin{center}
\begin{tabular}{lp{4.5in}} \toprule
$(z_1,z_2,z_3,z_4)$ \quad & \qquad Stabilizer subgroup $S(z_1,z_2,z_3,z_4)$ of $f(G)< (\Z_2)^4 \rtimes S_4$ \\ 
\midrule
$(a,a,a,a)$  & $(\{+1\} \cross \{+1\} \cross \{+1\} \cross \{+1\})\rtimes G_0$ \\
\midrule 
$(a,a,-a,a)$ & $\begin{array}{lll}
(1,1,1,1,e) & (1,1,-1,-1,(12)(34)) & (1,1,-1,-1,(1243))   \\
 (1,1,1,1,(14))  & (-1,1,-1,1,(13)(24)) & (-1,1,-1,1,(1342)) \\
(1,-1,-1,1,(23)) \hspace*{0.5em} &   (1,-1,-1,1,(14)(23)) & 
\end{array}$ \\ 
\midrule
$(a,-a,-a,a)$ & $\begin{array}{lll}
(1,1,1,1,e) & (-1,-1,-1,-1,(12)(34)) & (-1,-1,-1,-1,(1243))   \\
(1,1,1,1,(14))    & (-1,-1,-1,-1,(13)(24)) &  (-1,-1,-1,-1,(1342)) \\
 (1,1,1,1,(23)) \hspace*{0.5em}  & (1,1,1,1,(14)(23))                 &
\end{array}
$\\
\bottomrule
\end{tabular}
\end{center}
\caption{Stabilizer subgroups for certain four-component links
\label{tab:4stabs}}
\end{table}
\end{lemma}

As before, we can draw some conclusions about links from this theory which we might not have noticed otherwise. For example,
\begin{corollary}
If $L$ is a four-component link with linking matrix in the form of~\eqref{eq:gamma4tolink4}, then no element of $\Sigma(L)$ exchanges components 1 and 3 without also exchanging components 2 and 4.
\end{corollary}

\medskip

\section{The Satellite Lemma}

We begin with two definitions.  
\begin{definition}
A link $L$ is \emph{invertible} if reversing the orientation of all of its components produces a link isotopic to $L$, i.e., if $(1, -1, \ldots, -1, p) \in \Sigma(L)$ for some permutation $p$.  If the trivial permutation suffices, we call $L$ \emph{purely invertible}.
\end{definition}

\begin{definition}  Let $L$ be a link with $\mu$ components.  If swapping the $i$th and $j$th components produces a link isotopic to $L$, i.e., if the element $\gamma = (1, \ldots, 1, (ij))\in \Sigma(L)$, then components $(i,j)$ have a \emph{pure exchange symmetry}.
\end{definition}

We note that if components $K_i$ and $K_j$ are not of the same knot type, then it is impossible for them to have a pure exchange symmetry. We also note that if a link admits all pure exchange symmetries, the terms \emph{invertible} and \emph{purely invertible} are equivalent.



The most difficult part of our work below will be in ruling out pure exchange symmetries. So far, we have two (crude) tools; we can rule out pure exchange when the two components have different knot types or when the pure exchange does not preserve the linking matrix. In a few cases, these tools will not be enough and we will need the following:

\begin{lemma}
\label{lem:satellitelemma}
Suppose that $L(K,i)$ is a satellite of $L$ constructed by replacing component $i$ with a knot or link $K$. Then $L$ cannot have a pure exchange symmetry exchanging components $i$ and $j$ unless $L(K,i)$ and $L(K,j)$ are isotopic.
\end{lemma}

\begin{proof}
Such a pure exchange would carry an oriented solid tube around $L_i$ to a corresponding oriented solid tube around $L_j$. If we imagine $K$ embedded in this tube, this generates an isotopy between~$L(K,i)$ and~$L(K,j)$.
\end{proof}

The point of this lemma is that we can often distinguish $L(K,i)$ and $L(K,j)$ using classical invariants which are insensitive to the original labeling of the link. This seems like a general technique, and it would be interesting to explore this topic further.

\section{Two-component links}

This section records the symmetry group $\Sigma(L)$ for all prime two-component links with eight or fewer crossings; there are 30 such links to consider.  Our results are summarized in section~\ref{sec:2results}, which names and lists the symmetry groups which appear (see Table~\ref{tab:2name}).  We count how frequently each group appears by crossing number in Table~\ref{tab:2symm}.  The symmetry group for each link is listed in Tables \ref{tab:2bygroup} and \ref{tab:2bylink}, by group and by link, respectively.  Proofs of these assertions appear in section~\ref{sec2a} 

\subsection{Symmetry names and results}\label{sec:2results}

The Whitten group $\Gamma_2 = \Z_2 \cross (\Z_2 \cross \Z_2 \rtimes S_2)$ of all possible symmetries for two-component links is a nonabelian 16 element group  isomorphic to $\Z_2 \times D_4$.  The symmetry group $\Sigma(L)$ of a given link must form a subgroup of $\Gamma_2$.  There are 27 mutually nonconjugate subgroups of $\Gamma_2$; of these possibilities, only seven are realized as the symmetry subgroup of a prime link with 9 or fewer crossings (see Table~\ref{tab:2name}).  An eighth appears as the symmetry subgroup of a 10-crossing link.

\begin{question}
Do all 27 nonconjugate subgroups of $\Gamma_2$ appear as the symmetry group of some (possibly composite, split) link?  Of some prime, non-split link?
\end{question}

Hillmann~\cite{MR867798} provided examples for a few of these symmetry subgroups, but some of his examples were split links.  Here are the groups we found among links with 8 or fewer crossings.

 \begin{table}[ht]
\begin{center}
\begin{tabular}{lccc}
\toprule
\emph{Symmetry name} & \; \emph{Notation} \; & \emph{Subgroup of} & \; \emph{Isomorphic to} \; \\
&&  $\Gamma_2 = \Z_2 \cross (\Z_2 \cross \Z_2 \rtimes S_2)$  & \\
\midrule
No symmetry & $\{ \id \}$ & $\{ (1,1,1,\id )\}$ & $\{ \id \}$ \\
\midrule
Purely Invertible & $\s{2}{1}$ & $\{1\} \times \Delta(\Z_2 \times \Z_2) \times \{ \id \}$ & $\Z_2 \cong D_1$ \\
\midrule
Invertible with pure exchange &  $\s{4}{1}$ & $\{1\} \times (\Delta(\Z_2 \times \Z_2) \rtimes S_2 )$ & $\Z_2 \times \Z_2 \cong D_2$\\
\midrule
Individually invertible & $\s{4}{2}$ & $\{1\} \times \Z_2 \times \Z_2\times \{ \id \}$ & $\Z_2 \times \Z_2 \cong D_2$ \\
\midrule
Even number of operations & $\s{4}{3}$ & $\{ (\epsilon_0, \epsilon_1, \epsilon_2, \id ) : \epsilon_0  \epsilon_1 \epsilon_2 = 1\} $ & $\Z_2 \times \Z_2 \cong D_2$\\
\midrule
Full orientation-preserving & $\s{8}{1}$ & $\{1\}\times (\Z_2 \times \Z_2\rtimes S_2) $ & $D_4$\\
\midrule
Even ops \& pure exchange &  $\s{8}{2}$ & $\{ (\epsilon_0, \epsilon_1, \epsilon_2, p ) : \epsilon_0  \epsilon_1 \epsilon_2 = 1\}$  & $D_4$\\
\midrule
No exchanges & $\s{8}{3}$ & $\Z_2 \times \Z_2 \times \Z_2\times \{ \id \} $ & $\Z_2 \times \Z_2 \times \Z_2 $\\
\midrule
Full symmetry & $\Gamma_2 $ & $\Gamma_2$ & $\Z_2 \times D_4$ \\
\bottomrule
\end{tabular}
\smallskip
\caption{Certain Whitten symmetry groups of two-component links (subgroups of $\Gamma_2$). $\Delta(\Z_2 \times \Z_2)$ refers to the diagonal subgroup of $\Z_2 \times \Z_2$}
\label{tab:2name}
\end{center}
\end{table}

The first seven nontrivial groups in Table~\ref{tab:2name} are realized as the symmetry group of a link with nine or fewer crossings, while $\s{8}{3}$ appears to be the symmetry group of a 10-crossing link.  We know of no nontrivial links with full symmetry but speculate that they exist.

The next table records the frequency of each group.

\begin{table}[ht]
\begin{center}
\begin{tabular}{cccccccccc}
\toprule
& 2-component & \multicolumn{8}{c}{Link symmetry group} \\
\cmidrule{3-10}
Crossings & links count& $\s{2}{1}$  & $\s{4}{1}$ & $\s{4}{2}$ & $\s{4}{3}$
&$ \s{8}{1}$ & $\s{8}{2}$ & $\s{8}{3}$ & Full\\
\midrule
0a & 1 & &&&&&&& 1 \\
\addlinespace[0.5em]
2a & 1 & &&&&& 1 && \\
\addlinespace[0.25em]
4a & 1 & & 1&&&&&& \\
\addlinespace[0.25em]
5a & 1 & &&&& 1 &&& \\
\addlinespace[0.25em]
6a & 3 & & 2 &&&&1 &&\\
\addlinespace[0.25em]
7a & 6 & 1& 2 & 2  && 1 &&& \\
7n & 2 & 1 && 1 &&&&& \\
\addlinespace[0.25em]
8a & 14 & 3& 7&3 && &1&& \\
8n & 2 & 1 &&1& & &&&  \\
\midrule
Total & 31 & 6 & 12 & 7 && 2 & 3 & & 1 \\
\midrule
\multicolumn{2}{c}{First example} & $7^2_5$   &   $4^2_1$  & $7^2_4$   & $9^2_{61}$ &   $5^2_1$   & $2^2_1$  & -- & $0^2_1$ \\
\bottomrule
\end{tabular}
\end{center}
\vspace*{0.5em}
\caption{Whitten symmetry groups, by crossing number and by alternation}
\label{tab:2symm}
\end{table}
\bigskip

\newpage

Next, Table~\ref{tab:2bygroup} lists the prime two-component links of eight or fewer crossings by symmetry group.  

\bigskip

\begin{table}[ht]
\begin{center}
\begin{tabular}{lcc}
\toprule
\emph{Symmetry name} & \emph{Notation} & \emph{Prime links} \\
\midrule
No symmetry & $\{ \id \}$ &  none\\
\midrule
Invertible & $\s{2}{1}$ & $7^2_5, \, 7^2_7, \,8^2_9, \,8^2_{11}, \,8^2_{14}, \, 8^2_{16}$ \\
\midrule
Invertible with pure exchange &  $\s{4}{1}$ & $4^2_1, \,6^2_1, \,6^2_3, \, 7^2_1,  \, 7^2_2, \, 8^2_1, \, 8^2_2, \, 8^2_3, \, 8^2_4, \, 8^2_5, \, 8^2_6, \, 8^2_7$\\
\midrule
Individually invertible & $\s{4}{2}$ & $7^2_4, \, 7^2_6, \, 7^2_8,  \, 8^2_{10}, \, 8^2_{12}, \,8^2_{13}, \, 8^2_{15}$ \\
\midrule
Even number of operations & $\s{4}{3}$ & none\\
\midrule
Full orientation-preserving & $\s{8}{1}$ & $5^2_1, \, 7^2_3 $\\
\midrule
Even operations with pure exchange &  $\s{8}{2}$ & $2^2_1, \, 6^2_2, \, 8^2_8 $\\
\midrule
No exchanges & $\s{8}{3}$ & none\\
\midrule
Full symmetry & $\Gamma_2 $ & $0^2_1$ \\
\bottomrule
\end{tabular}
\smallskip
\caption{List of Whitten symmetry groups, and links possessing that symmetry} 
\label{tab:2bygroup}
\end{center}
\end{table}

We conclude our tables of results by listing each two-component link and its corresponding Whitten symmetry group:

\begin{table}[ht]
\begin{center}
\renewcommand{\arraystretch}{1.3}
\begin{tabular}{cccccccc}\toprule
Link&Symmetry Group & \hspace*{0.2in} & Link&Symmetry Group & \hspace*{0.2in} & Link &Symmetry Group\\
\cmidrule{1-2} \cmidrule{4-5} \cmidrule{7-8}
$ 2^2_1\; (2a1)$&$\Sigma_{8,2}$ && $7^2_5\; (7a2)$&$\Sigma_{2,1}$ && $8^2_7\; (8a8)$&$\Sigma_{4,1}$\\
\cmidrule{1-2} \cmidrule{4-5} \cmidrule{7-8}
$ 4^2_1\; (4a1)$&$\Sigma_{4,1}$ && $7^2_6\; (7a1)$&$\Sigma_{4,2}$ && $8^2_8\; (8a9)$&$\Sigma_{8,2}$\\
\cmidrule{1-2} \cmidrule{4-5} \cmidrule{7-8}
$5^2_1\; (5a1)$&$\Sigma_{8,1}$ && $7^2_7\; (7n1)$&$\Sigma_{2,1}$ && $8^2_9\; (8a3) $&$\Sigma_{2,1}$\\
\cmidrule{1-2} \cmidrule{4-5} \cmidrule{7-8}
$6^2_1\; (6a3)$&$\Sigma_{4,1}$ && $7^2_8\; (7n2)$&$\Sigma_{4,2}$ && $8^2_{10}\; (8a2)$&$\Sigma_{4,2}$ \\
\cmidrule{1-2} \cmidrule{4-5} \cmidrule{7-8}
$6^2_2\; (6a2)$&$\Sigma_{8,2}$ && $8^2_1\; (8a14)$&$\Sigma_{4,1}$ && $8^2_{11}\; (8a5)$&$\Sigma_{2,1}$\\
\cmidrule{1-2} \cmidrule{4-5} \cmidrule{7-8}
$6^2_3\; (6a1)$&$\Sigma_{4,1}$ && $8^2_2\; (8a12)$&$\Sigma_{4,1}$ && $8^2_{12}\; (8a4)$&$\Sigma_{4,2}$\\
\cmidrule{1-2} \cmidrule{4-5} \cmidrule{7-8}
$7^2_1\; (7a6)$&$\Sigma_{4,1}$ && $8^2_3\; (8a11)$&$\Sigma_{4,1}$ && $8^2_{13}\; (8a1)$&$\Sigma_{4,2}$ \\
\cmidrule{1-2} \cmidrule{4-5} \cmidrule{7-8}
$7^2_2\; (7a5)$&$\Sigma_{4,1}$ && $8^2_4\; (8a13)$&$\Sigma_{4,1}$ &&  $8^2_{14}\; (8a7)$&$\Sigma_{2,1}$\\
\cmidrule{1-2} \cmidrule{4-5} \cmidrule{7-8}
$7^2_3\; (7a4)$&$\Sigma_{8,1}$ && $8^2_5\; (8a10)$&$\Sigma_{4,1}$ && $8^2_{15}\; (8n2)$&$\Sigma_{4,2}$ \\
\cmidrule{1-2} \cmidrule{4-5} \cmidrule{7-8}
$7^2_4\; (7a3)$&$\Sigma_{4,2}$ && $8^2_6\; (8a6)$&$\Sigma_{4,1}$ && $8^2_{16}\; (8n1)$&$\Sigma_{2,1}$ \\
\bottomrule
\end{tabular}
\smallskip
\caption{The Whitten symmetry group for each two-component link}
\label{tab:2bylink}
\end{center}
\end{table}

\newpage
\begin{theorem} \label{thm:2comp}
The symmetry groups for all prime two-component links up to 8 crossings are as listed in Table~\ref{tab:2bylink}.
\end{theorem}

The proof of this theorem is divided into five cases, based on the five symmetry groups that appear in Table~\ref{tab:2bylink}; these proofs are found in section~\ref{sec2a}.  Many of our arguments generalize to various families of links.  As this paper focuses on these first examples, we ask for the reader's understanding when we eschew the most general argument in favor of a simpler, more expedient one.

\bigskip

\subsection{Proofs for two-component links} \label{sec2a}

Below, we attempt to provide a general framework for determining symmetry groups for two-component links.  Since $\Sigma(L)$ is a subgroup of $\Gamma_2$, the order of the symmetry group must divide $|\Gamma_2 | = 16$.  Our strategy begins by exhibiting certain symmetries via explicit isotopies.  With these in hand, we next use various techniques to rule out some symmetries until we can finally determine the symmetry group $\Sigma(L)$.  These techniques generally involve using some link invariant to show $L^\gamma \nsim L$.  Among link invariants, the linking number and self-writhe (for alternating links) are easily applied since they count signed crossings; we also use polynomials and other methods.  

We focus on the 30 prime links with eight or fewer crossings.  Our first results indicate which of these 30 links have either a \emph{pure invertibility} or a \emph{pure exchange} symmetries, which we prove explicitly by exhibiting isotopies.  Recall that a link is purely invertible if reversing all components' orientations produces an isotopic link; a link has pure exchange symmetry if swapping its two components is an isotopy.

\begin{lemma}\label{lemma:pe}
Via the isotopies exhibited in Figures~\ref{fig:pe-y}, \ref{fig:pe-z}, and \ref{fig:721pe}-\ref{fig:twisttype} in Appendices~\ref{app:rotate-pe} and \ref{app:pe}, the following 17 links have pure exchange symmetry:
 \begin{equation} \label{eq:pe}
2^2_1,~~4_1^2, ~~5_1^2, ~~6_1^2, ~~6_2^2,~~6_3^2,~~ 7_1^2, ~~7_2^2, ~~ 7_3^2,~~ 8_1^2,~~ 8_2^2, ~~8_3^2, ~~8^2_4,~~8^2_5,~~8^2_6,~~8^2_7, ~~8^2_8
\end{equation}
i.e., $(1,1,1,(12))$ belongs to the symmetry group of each of these links.
\end{lemma}

As we determine symmetry groups, we will establish that the remaining 13 links in consideration do not have pure exchange symmetry.  
 
Cerf~\cite{cerf} states that all prime, alternating two-component links of 8 or fewer crossings are invertible, though this may involve exchanging components.  Via the isotopies exhibited in Figures~\ref{fig:pi-y} and \ref{fig:622pi}-\ref{fig:8215pi} of Appendices~\ref{app:rotate-pi} and \ref{app:pi}, we extend Cerf's result to non-alternating links, and we show that the invertibility is pure (i.e., without exchanging components).  To obtain invertibility for $7^2_8$, combine the results of Figures~\ref{fig:rotate-y} and \ref{fig:rotate-z}, which show that each of its components can be individually inverted.

\begin{lemma}
\label{lemma:pi}
All 30 prime two-component links with eight or fewer crossings are purely invertible.
\end{lemma}

\medskip

We note that the pure exchange and pure invertibility symmetries, corresponding to Whitten elements $(1,1,1,(12))$ and $(1,-1,-1,\id)$, respectively, generate the subgroup $\s{4}{1}$ of $\Gamma_2$.  This implies our first result about link symmetry groups.

\begin{lemma}
Any two-component link, such as those listed in \eqref{eq:pe}, that has both pure exchange symmetry and (pure) invertibility, must have $\Sigma_{4,1}$ as a subgroup of its symmetry group $\Sigma(L)$.
\end{lemma}

By examining signed crossings of a link, we calculate its linking number and self-writhe; if one of these is nonzero, we may rule out some symmetries.

\begin{lemma}  \label{lemma:selfwrithe}
\begin{enumerate}
\item If the linking number $\Lk(L)\neq 0$, then $\Sigma(L) < \s{8}{2}$.  
\item For $L$ alternating, if the self-writhe $s(L) \neq 0$, then $\Sigma(L) < \s{8}{1}$.
\item For $L$ alternating, if $\Lk(L)\neq 0$ and $s(L) \neq 0$, then $\Sigma(L) < \s{4}{1}$.
\end{enumerate}
\end{lemma}

\begin{proof}
Consider the effect of each symmetry operation upon linking number (see  Proposition~\ref{prop:linkingmatrix}):  mirroring a link or inverting one of its components will swap the sign of the linking number, while exchanging its components fixes the linking number.  As for the self-writhe of a link, it is fixed by inverting any component or exchanging the two components; however, mirroring a link swaps the sign of $s(L)$.

Thus, the elements of $\Gamma_2$ that will swap the sign of a linking number are precisely those of the form $\gamma = (\epsilon_0, \epsilon_1, \epsilon_2, p)$ with $\epsilon:= \epsilon_0 \epsilon_1\epsilon_2 = -1$.   If the linking number $\Lk(L)$ is nonzero, these cannot possibly produce a link $L^\gamma$ isotopic to the original link $L$, so these eight elements are not part of $\Sigma(L)$.  The remaining eight symmetry elements form $\s{8}{2}= \{ (\epsilon_0, \epsilon_1, \epsilon_2, p ) : \epsilon_0  \epsilon_1 \epsilon_2 = 1\}$, which proves the first assertion.

Self-writhe is an invariant of reduced diagrams of alternating links, and any symmetry operation which mirrors the link will swap the sign of $s(L)$.  If the self-writhe $s(L)$ is nonzero, then no element which mirrors, i.e., $(-1, \epsilon_1, \epsilon_2, p)$, can lie in $\Sigma(L)$.  The remaining elements form $\s{8}{1}=\{1\}\times (\Z_2 \times \Z_2\rtimes S_2) $.

The last assertion follows as an immediate corollary of the first two.  If both hypotheses are satisfied, then $\Sigma(L) \subset \s{8}{1} \cap \s{8}{2} = \s{4}{1}$.
\end{proof}

\begin{lemma} \label{lemma:knottype}
Let $L$ be a two-component link. 
\begin{enumerate}
\item If $L$ is purely invertible, then $\s{2}{1}<\Sigma(L)$.
\item If the components of $L$ are different knot types, then $\Sigma(L) < \s{8}{3}$.
\item If both hypotheses above are true, and
\begin{enumerate}
\item if $lk(L) \neq 0$, then $\Sigma(L)$ is either $\s{2}{1}$ or $\s{4}{3}$.
\item if $L$ is alternating and $s(L) \neq 0$, then $\Sigma(L)$ is either $\s{2}{1}$ or $\s{4}{2}$.
\end{enumerate}
\end{enumerate}
\end{lemma}

\begin{proof}  The first assertion is immediate, since the purely invertible symmetry generates $\s{2}{1}$.  

If the components of $L$ have different knot types, then no exchange symmetries are permissible; the permutation $p=(12)$ never appears in $\Sigma(L)$.  Hence the symmetry group $\Sigma(L)$ is contained in the `No exchanges' group $\s{8}{3}$.

Combining these two results with the previous lemma proves the third assertion.  If the linking number is nonzero and the components of $L$ have different knot types, then $\s{2}{1} < \Sigma(L) < \s{8}{2} \cap \s{8}{3} = \s{4}{3}$.  If $L$ is also purely invertible, then $\s{2}{1} < \Sigma(L) < \s{4}{3}$.  This implies that the order of $\Sigma(L)$ equals 2 or 4, so it is either $\s{2}{1}$ or  $\s{4}{3}$.

If instead self-writhe is nonzero and the first two hypotheses hold, then $\s{2}{1} < \Sigma(L) < \s{8}{1} \cap \s{8}{3} = \s{4}{2}$.  This implies that the order of $\Sigma(L)$ equals 2 or 4, so it is either $\s{2}{1}$ or $\s{4}{2}$.
\end{proof}

With these five lemmas in hand, we are now prepared to begin proving Theorem~\ref{thm:2comp}, which we treat by each symmetry group.

\subsubsection{Links with symmetry group $\s{2}{1}$}

\begin{claim}
Links $7^2_5, \; 7^2_7,  \; 8^2_9, \; 8^2_{11}, \; 8^2_{14},$ and $8^2_{16}$ have symmetry group $\s{2}{1}$.
\end{claim}

\begin{proof}
All of these links are purely invertible, so $\s{2}{1} \subset \Sigma(L)$.  Also, all of them have components of different knot types and nonzero linking numbers; thus by Lemma~\ref{lemma:knottype}, their symmetry groups are either $\s{2}{1}$ or $\s{4}{3}$.  

Three of the alternating links $(7^2_5, \; 8^2_{11}, \; 8^2_{14})$  have nonzero self-writhe, so we apply Lemma~\ref{lemma:knottype} again.  We conclude that they have only the purely invertible symmetry, and $\s{2}{1}$ is their symmetry group.

For the remaining three links in this case $\left( 7^2_7,  \; 8^2_{11}, \; 8^2_{16}\right)$, consider the action of the Whitten element $\gamma = (-1, -1, 1, \id)$.  We consider the Jones polynomials of $L$ and $L^\gamma$.  They are unequal, as demonstrated below, which implies $L^\gamma$ is not isotopic to $L$.  Thus $\Sigma(L) \neq \s{4}{3}$, so it must be $\s{2}{1}$. 

\begin{align*}
\Jones(7^2_7) = & \, z^{-15/2} - z^{-13/2} - z^{-9/2} - z^{-5/2} \\
\Jones\left((7^2_7)^\gamma \right) = & \,-{z^{-7/2}} -{z^{-3/2}} -\sqrt{z} +z^{3/2} \\
& \\
\Jones(8^2_{11}) = & \,-z^{9/2}+3 z^{7/2}-\frac{1}{z^{7/2}}-4 z^{5/2}+\frac{1}{z^{5/2}}+5 z^{3/2}-\frac{4}{z^{3/2}}-5 \sqrt{z}+\frac{4}{\sqrt{z}}\\
\Jones\left((8^2_{11})^\gamma \right) = & \,-\frac{4}{z^{9/2}}+\frac{1}{z^{7/2}}-\frac{1}{z^{5/2}}-\frac{1}{z^{21/2}}+\frac{3}{z^{19/2}}-\frac{4}{z^{17/2}}+\frac{5}{z^{15/2}}-\frac{5}{z^{13/2}}+\frac{4}{z^
   {11/2}} \\
& \\
\Jones(8^2_{16}) = & \,\frac{2}{z^{9/2}}-\frac{2}{z^{7/2}}+\frac{2}{z^{5/2}}-\frac{2}{z^{3/2}}-\frac{2}{z^{11/2}}-\sqrt{z}+\frac{1}{\sqrt{z}} \\
 \Jones\left((8^2_{16})^\gamma \right) = & \,-\frac{2}{z^{9/2}}+\frac{2}{z^{7/2}}-\frac{2}{z^{5/2}}+\frac{2}{z^{3/2}}-\frac{1}{z^{13/2}}+\frac{1}{z^{11/2}}-\frac{2}{\sqrt{z}} \qedhere
\end{align*}
\end{proof}

\smallskip

\subsubsection{Links with symmetry group $\s{4}{1}$}

\begin{claim}
\label{s41claim}
Links $4^2_1,~6^2_1,~8^2_1,~8^2_2,$ and $8^2_4$ have symmetry group $\s{4}{1}$.
\end{claim}
\begin{proof}
All five of these links appear in our list \eqref{eq:pe} of pure exchange symmetry links; also, they all are purely invertible and have nonzero linking numbers.  Lemmas~\ref{lemma:pe} and \ref{lemma:selfwrithe} imply that $\s{4}{1} < \Sigma(L) < \s{8}{2}$.  

For each link, the Conway polynomials differ for $L$ and $L^\gamma$, where $\gamma=(-1,-1,1, \id) \in \s{8}{2}$.  Thus each link cannot have $\s{8}{2}$ as its symmetry group and must therefore have $\Sigma(L)=\s{4}{1}$.  We display the Conway polynomials in Table~\ref{s41table} below.
\end{proof}

\renewcommand{\arraystretch}{1.1}
\begin{center}
\begin{table}[ht]
\begin{tabular}{ccc}
\toprule
Link $L$ & Conway$(L)$ & Conway$(L^\gamma)$ \\
\midrule
$4^2_1$ & $2z$ & $z^3+2z$ \\
$6^2_1$ & $-z^5-4z^3-3z$ & $-3z$ \\
$8^2_1$ & $-z^7-6z^5-10z^3-4z$ & $-4z$ \\
 $8^2_2$ & $2z^5+7z^3+4z$ & $3z^3+4z$ \\
$8^2_4$ & $4z^3+4z$ & $2z^5+6z^3+4z$ \\
\bottomrule
\end{tabular}
\smallskip
\caption{Conway polynomials for Claim~\ref{s41claim}}
\label{s41table}
\end{table}
\end{center}
\renewcommand{\arraystretch}{1}

\begin{claim} 
Links  $6^2_3,~7^2_1,~7^2_2,~8^2_3,~8^2_5,~8^2_6,$ and $8^2_7$ have symmetry group $\s{4}{1}$.
\end{claim} 

\begin{proof}
These links have both pure exchange and pure invertibility symmetries; they also have nonzero self-writhes and linking numbers.  By Lemmas~\ref{lemma:pe} and \ref{lemma:selfwrithe}, their symmetry group must be $\s{4}{1}$.  
\end{proof}


\subsubsection{Links with symmetry group $\s{4}{2}$}
\begin{claim} 
Links  $7^2_4$ and $8^2_{12}$ have symmetry group $\s{4}{2}$.
\end{claim}
\begin{proof}
These links are purely invertible, comprised of different knot types, and have self-writhe $s(L) \neq 0$; thus, Lemma~\ref{lemma:selfwrithe} implies $\s{2}{1} < \Sigma(L) <  \s{4}{2}$.  

Figures \ref{724invert1} and \ref{8212invert1} exhibit an isotopies which shows that $(-1,-1,1,\id)\in \Sigma(L)$ for each link, which means $\s{4}{2}$ is their symmetry group.
\end{proof}

\begin{claim} 
Links $7^2_8, 8^2_{10},$ and $8^2_{15}$ have symmetry group $\s{4}{2}$.
\end{claim}
\begin{proof}
These links are purely invertible and comprised of different knot types; thus, Lemma~\ref{lemma:selfwrithe} implies $\s{2}{1} < \Sigma(L) < \s{8}{3}$.  

First, we use the Jones polynomial to rule out mirror symmetry, i.e., the element $\gamma=(-1,1,1,\id) \in \s{8}{3}$ does not lie in $\Sigma(L)$.  That means order of the subgroup $\Sigma(L)$ is between 2 and 7; hence it is either a 2 or 4 element subgroup.  Here are the Jones polynomials:

\begin{align*}
\Jones(7^2_8) = & -\frac{1}{z^{9/2}}+\frac{1}{z^{7/2}}-\frac{2}{z^{5/2}}+\frac{1}{z^{3/2}}+\frac{1}{z^{11/2}}-\frac{2}{\sqrt{z}} \\
\Jones\left((7^2_8)^\gamma \right) = & -z^{9/2}+z^{7/2}-2 z^{5/2}+z^{3/2}+z^{11/2}-2 \sqrt{z} \\
& \\
\Jones(8^2_{10}) = & -z^{9/2}+3 z^{7/2}-\frac{1}{z^{7/2}}-5 z^{5/2}+\frac{2}{z^{5/2}}+5 z^{3/2}-\frac{4}{z^{3/2}}-6 \sqrt{z}+\frac{5}{\sqrt{z}} \\
\Jones\left((8^2_{10})^\gamma \right) = & -\frac{1}{z^{9/2}}-z^{7/2}+\frac{3}{z^{7/2}}+2 z^{5/2}-\frac{5}{z^{5/2}}-4 z^{3/2}+\frac{5}{z^{3/2}}+5 \sqrt{z}-\frac{6}{\sqrt{z}} \\
& \\
\Jones(8^2_{15}) = & -\frac{1}{z^{7/2}}-z^{5/2}+\frac{1}{z^{5/2}}+z^{3/2}-\frac{1}{z^{3/2}}-2 \sqrt{z}+\frac{1}{\sqrt{z}} \\
\Jones\left((8^2_{15})^\gamma \right) = & -z^{7/2}+z^{5/2}-\frac{1}{z^{5/2}}-z^{3/2}+\frac{1}{z^{3/2}}+\sqrt{z}-\frac{2}{\sqrt{z}}
\end{align*}

Next, for each link we depict an isotopy which reverses the orientation of just one component, i.e., we show either $\gamma=(1,-1,1,\id)$ or $\gamma=(1,1,-1,\id)$ lies in $\Sigma(L)$.  This means $\s{4}{2}$ is the symmetry group for these three links. 

Figures~\ref{fig:rotate-y}, \ref{8210invert1}, and \ref{8n2a} show these isotopies for $7^2_8$, $8^2_{10}$, and $8^2_{15}$, respectively.
\end{proof}

\begin{claim} 
Links  $7^2_6$ and $8^2_{13}$ have symmetry group $\s{4}{2}$.
\end{claim}
\begin{proof}
These links are purely invertible and have nonzero self-writhe; thus Lemma~\ref{lemma:selfwrithe} implies $\s{2}{1} < \Sigma(L) <  \s{8}{1}$.  

We take the satellites, $L_1,~L_2$ of the first and second component of $L$, for $L\in\{7^2_6,~8^2_{13}\}$, and compute the Jones polynomial for each.  

{\small
\begin{align*}
\Jones\left(\left(8^2_{13}\right)_1\right) & =1 -\frac{1}{z^6} + \frac{2}{z^5} - \frac{3}{z^4} + 		   \frac{4}{z^3} - \frac{2}{z^2} + \frac{1}{z} - 2 z + 4 z^2 - 2 z^3 + 3 z^4 - z^5 \\
\Jones\left(\left(8^2_{13}\right)_2\right) & = 1 + \frac{1}{z^{11}} - \frac{2}{z^{10}} + \frac{2}{z^8} -  \frac{2}{z^7} + \frac{1}{z^6} + \frac{1}{z^5} - \frac{2}{z^4} + \frac{2}		{z^3} - \frac{1}{z^2} + \frac{1}{z} - 2 z + 4 z^2 - z^3 + 2 z^5 - z^6 \\
& \\
\Jones\left(\left(7^2_6\right)_1\right) & =2 + 1/z^{10} - 2/z^9 + 1/z^8 - 1/z^6 + 2/z^5 - 1/z^4 + 2/z^3 - z + z^2 + z^3 - z^4 \\
\Jones\left(\left(7^2_6\right)_2\right) & =3 + 1/z^7 - 2/z^6 + 2/z^5 - 2/z^4 + 2/z^3 + 1/z^2 - 2 z + 2 z^2 - z^3
\end{align*}
}

Since the Jones polynomials the two different satellites are not equal for either link $L$, we have that $L$ is not isotopic to $L_\gamma$ for $\gamma=(1,1,1,(12))$ by  Lemma~\ref{lem:satellitelemma}.  Thus, $(1,1,1,(12))\notin\Sigma(L)$.

\begin{figure}[ht]
\begin{center}
\scalebox{0.65}{\includegraphics{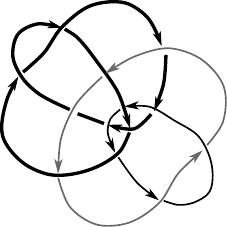}}
\hspace{1in}\scalebox{0.65}{\includegraphics{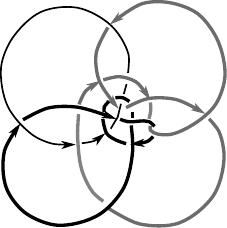}}

\caption{Gray knot is satellite for the first and second component of $7^2_6$, respectively}\label{726sat}
\end{center}
\end{figure}

\begin{figure}[ht]
\begin{center}
\scalebox{0.65}{\includegraphics{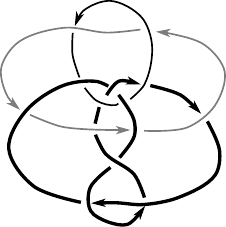}}
\hspace{1in}\scalebox{0.65}{\includegraphics{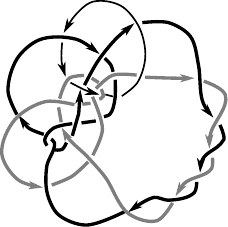}}

\caption{Gray knot is satellite for the first and second component of $8^2_{13}$, respectively}\label{8213sat}
\end{center}
\end{figure}

Figures \ref{726invert1} and \ref{8213invert1} exhibit isotopies that show $(1,-1,1,e)\in \Sigma(L)$ for both of these links.  Therefore, we conclude $\Sigma(L)=\Sigma_{4,2}$ for these two links. 
\end{proof}

\subsubsection{Links with symmetry group $\s{8}{1}$}
\begin{claim}
Links $5^2_1$ and $7^2_3$ have symmetry group $\s{8}{1}$.
\end{claim}
\begin{proof}
These two links have the pure exchange and pure invertibility symmetries, and their self-writhes are nonzero; thus Lemma~\ref{lemma:selfwrithe} implies $\s{4}{1} < \Sigma(L) <  \s{8}{1}$.  We show, in figures~\ref{fig:rotate-z} and \ref{fig:rotate-y}, that each symmetry group includes either $(1, 1, -1, \id)$ or $(1,-1,1, \id)$, neither of which is an element of $\s{4}{1}$.  Therefore, we conclude $\Sigma(L)=\Sigma_{8,1}$ for these two links. 
\end{proof}

\subsubsection{Links with symmetry group $\s{8}{2}$}
\begin{claim}
Links  $2^2_1, \; 6^2_2,$ and $8^2_8$ have symmetry group $\s{8}{2}$.
\end{claim}
\begin{proof}
These three links have the pure exchange and pure invertibility symmetries, and their linking numbers are nonzero; thus Lemma~\ref{lemma:selfwrithe} implies $\s{4}{1} < \Sigma(L) <  \s{8}{2}$.  

Figures~\ref{221mi}, \ref{622mi}, and \ref{828a} display the isotopies which show $(-1,1,-1,e)$ lies in the symmetry group for each of these three links.  Since this element is not in $\s{4}{1}$, we may conclude all three links have symmetry group $\Sigma(L)=\Sigma_{8,2}$.
\end{proof}

\medskip
\section{Three-component links}

There are 14 three-component links with 8 or fewer crossings.  In this section, we determine the symmetry group for each one; Table~\ref{tab:3symm} summarizes the results.  We obtain 11 different symmetry subgroups inside $\Gamma_3$, which represent 7 different conjugacy classes of subgroups (out of the 131 possible).  

For each link, our first task is to calculate the linking matrix.  Then, we utilize Table~\ref{tab:stabs} to determine the stabilizer of this matrix within $\Gamma_3$;  we know that the symmetry group $\Sigma(L)$ must be a subgroup of this stabilizer.  From there, we proceed by ruling out other elements using polynomial invariants and by exhibiting isotopies to show that certain symmetries do lie in $\Sigma(L)$ until we can discern the symmetry group.   


Here are the results, listed in terms of generators for each group.  We use the following notation for common group elements:
\begin{itemize}
\item PI, for pure invertibility, i.e., the element $(1,-1,-1,-1,e)$
\item PE, for having all pure exchanges, i.e., all elements $(1,1,1,1,p)$ where $p \in S_3$
\end{itemize}

\begin{table}[ht]
\begin{center}
\begin{tabular}{cccl}
\toprule
Link(s) & $|\Sigma(L)|$ & $\Sigma$ isomorphic to & Generators  \\
\midrule
$8^3_1 \, (8a18), \; 8^3_8\, (8n4)$ & 4 & $\Z_2 \times \Z_2 \cong D_2$ & PI, $(1,1,-1,-1,(23))$ \\
\addlinespace[0.1em]
$8^3_2\, (8a17), \;8^3_7 \, (8n3)$ & 4 & $\Z_2 \times \Z_2 \cong D_2$ & PI, $(1,1,1,1,(23))$ \\
\addlinespace[0.1em]
$8^3_{10} \, (8n6) $ & 4 &  $\Z_2 \times \Z_2 \cong D_2$ & PI, $(1,1,1,1,(12))$ \\
\addlinespace[0.1em]
$8^3_{4} \, (8a20)$ & 4 & $\Z_2 \times \Z_2 \cong D_2$  & PI, $(-1,1,1,1,(12))$ \\
\addlinespace[0.1em]
$8^3_{5}\, (8a16)$ & 4 & $\Z_2 \times \Z_2 \cong D_2$ & $(1,1,-1,-1,e), ~(1,-1,-1,-1,(23))$ \\
\addlinespace[0.1em]
\midrule
\addlinespace[0.2em]
$8^3_{6} \, (8a19)$ & 8 & $(\Z_2)^3$& PI, $(-1,1,-1,-1,e),~(1,1,-1,-1,(23))$ \\
\addlinespace[0.1em]
$8^3_{9} \, (8n5)$ & 8 & $(\Z_2)^3$ & PI, $(1,1,1,1,(23)), ~(1,1,-1,-1,e)$ \\
\addlinespace[0.1em]
\midrule
\addlinespace[0.2em]
$6^3_{1}\, (6a5), \; 8^3_3 \, (8a15)$ & 12 & $D_6$  & PI, PE \\
\addlinespace[0.1em]
$6^3_{3}\, (6n1)$ & 12 & $D_6$ & $(1,1,1,1,(12)), (1,-1,1,1,(123))$ \\
\addlinespace[0.1em]
$7^3_{1}\, (7a7)$ & 12 & $D_6$ & $(1,1,1,1,(23)),  (1,1,-1,1,(123))$ \\
\addlinespace[0.1em]
\midrule
\addlinespace[0.2em]
$6^3_{2}\, (6a4)$ & 48 & $\Z_2 \times \left((\Z_2)^2 \rtimes S_3 \right)$ & 
	\begin{tabular}{ll}
	$(-1,1,1,1,e)$ & $(1,1,1,1,(123))$ \\
	$(1,-1,1,1,(12))$ &  \\
	\end{tabular} \\
\addlinespace[0.1em]
\bottomrule
\end{tabular}
\end{center}
\vspace*{0.5em}
\caption[Three-component link symmetry groups, by crossing number.]{Three-component link symmetry groups, by crossing number.  The first three groups are all conjugate to each other; also, all 12 element groups above are conjugate to each other.}
\label{tab:3symm}
\end{table}

We note that all but two of these links are purely invertible, even though PI might not be part of a minimal set of generators.  Neither the Borromean rings $\left(6^3_2\right)$ or the link $8^3_5$ are purely invertible.  Both of these links are, however, invertible.  The Borromean rings can be inverted using any odd permutation $p$, i.e., they admit the symmetry $\gamma=(1,-1,-1,-1,p)$; the link $8^3_5$ is invertible using $p=(23)$.

\begin{claim}
The subgroup $\Sigma(6^3_1) < \Gamma_3$ is the $12$ element group isomorphic to $D_6$ generated by pure exchanges and pure invertibility.
\end{claim}

\begin{proof}
The linking matrix for $6^3_1$ is
\begin{center}\begin{tabular}{c|ccc}
&1&2&3\\ \hline
1&0&-1&-1\\
2&-1&0&-1\\
3&-1&-1&0\\
\end{tabular}\end{center}
which is in the standard form $(a,a,a)$. We know that $\Sigma(6^3_1)$ is a subgroup of the stabilizer of this matrix under the action of $\Gamma_3$ on linking matrices. Consulting Table~\ref{tab:stabs}, we see that this stabilizer is the group in the claim. We must now show that all these elements are in the group. Figures \ref{fig:pe-z} and \ref{631(23)} show that
\begin{equation*}
(1,1,1,1,(123)),~(1,1,1,1,(23))\in\Sigma(6^3_1)
\end{equation*}
Since any 3-cycle and 2-cycle generate $S_3$, we have the rest of the pure exchanges as well.
Figure~\ref{631invert} shows that this link is purely invertible, completing the proof.
\end{proof}


\begin{claim}
The subgroup $\Sigma(6^3_2) < \Gamma_3$ is the 48 element group where $(\epsilon_0, \epsilon_1, \epsilon_2, \epsilon_3, p)$ is in the group if either 
\begin{itemize}
\item[(a)] $\epsilon_1 \epsilon_2\epsilon_3 = 1$ and $p$ is an even permutation, or
\item[(b)] $\epsilon_1 \epsilon_2\epsilon_3 = -1$ and $p$ is odd.
\end{itemize}
\end{claim}


\begin{proof}
Figures \ref{fig:rotate-z}, \ref{632_1}, and \ref{632_5} tell us that $\Sigma(6^3_2)$ contains the elements
\begin{equation*}
(1,-1,1,-1,(132)), ~(-1,-1,1,1,(13)), ~(-1,1,1,1,e),
\end{equation*}
which clearly obey the rules in the claim. In fact, they generate a group of $48$ such elements.  Since the order of $\Sigma(6^3_1)$ must divide $|\Gamma_3| = 96$, it is either these 48 elements or it is all of $\Gamma_3$.  But in \cite{MR0380802}, Montesinos proves that $6^3_2$ is not purely invertible.  Thus, $(1,-1,-1,-1,e)$ cannot be in $\Sigma(6^3_2)$, which completes the proof.
\end{proof}

\begin{claim}
The subgroup $\Sigma(6^3_3) < \Gamma_3$ is the 12 element group 
isomorphic to $D_6$ given by $f^{-1}(S(a,a,-a))$.
\end{claim}

\begin{proof}
Figures~ \ref{6n1_2} and \ref{6n1_3} imply that $\Sigma(6^3_3)$ contains
\begin{equation*}
\{(1,1,-1,1,(132)),~(1,-1,-1,-1,(12))\}.
\end{equation*}
These three elements generate the 12 element group of the claim.  Now the linking matrix for $6^3_3$ is
\begin{center}\begin{tabular}{c|ccc}
&1&2&3\\ \hline
1&0&1&-1\\
2&1&0&-1\\
3&-1&-1&0\\
\end{tabular}\end{center}
 which is in the standard form $(a,a,-a)$.  Consulting Table~\ref{tab:stabs}, we see that this means $|\Sigma(6^3_3)|$ divides 12, the order of the stabilizer.  Since we already have 12 elements in the symmetry group, it must equal the stabilizer, which completes the proof.
\end{proof}

\begin{claim} \label{claim:731}
The subgroup $\Sigma(7^3_1) < \Gamma_3$ is the 12 element subgroup isomorphic to $D_6$ that is conjugate to $f^{-1}(S(a,a,-a))$ by $(1,1,1,1,(13))$.
\end{claim}

\begin{proof}
The linking matrix for $7^3_1$ is
\begin{center}\begin{tabular}{c|ccc}
&1&2&3\\ \hline
1&0&-1&-1\\
2&-1&0&1\\
3&-1&1&0\\
\end{tabular}\end{center}
which corresponds to the triple (1,-1, -1) and has the same orbit type as $(a, a, -a)$.  Thus, the stabilizer of this linking matrix is a 12 element group conjugate to the stabilizer $f^{-1}(S(a,a,-a))$ by $(1,1,1,1,(13))$.  Figures~\ref{731+-+(132)} and \ref{731+++(23)} show that
\begin{equation*}
(1,1,-1,1,(123)), ~(1,1,1,1,(23))\in\Sigma(7^3_1).
\end{equation*}
These elements generate a 12 element group, so this stabilizer is the entire symmetry group of $7^3_1$, as claimed. We note that this stabilizer was worked out explicitly as Example~\ref{ex:731stab}.
\end{proof}

%


\begin{claim}
The subgroup $\Sigma(8^3_1) < \Gamma_3$ is the 4 element group isomorphic to $D_2$ generated by pure invertibility and $(1,1,-1,-1,(23))$.
\end{claim}

\begin{proof}
The linking matrix for $8^3_1$ is
\begin{center}\begin{tabular}{c|ccc}
&1&2&3\\ \hline
1&0&-1&1\\
2&-1&0&2\\
3&1&2&0\\
\end{tabular}\end{center}
which is in the standard form $(a,b,-b)$. Consulting Table~\ref{tab:stabs}, we see that the stabilizer of this linking matrix in $\Gamma_3$ has order 4. But Figures \ref{fig:pi-y} and \ref{831(23)} show stabilizer elements
\begin{equation*}
(1,-1,-1,-1,~e), ~(1,-1,1,1,(23))\in\Sigma(8^3_1),
\end{equation*}
which means that we also have $(1,1,-1,-1,(23))\in\Sigma(8^3_1)$.  Therefore these three elements, plus the identity, must form the symmetry group $\Sigma(8^3_1)$ .
\end{proof}

\begin{claim}
The subgroup $\Sigma(8^3_2) < \Gamma_3$ is the 4 element group isomorphic to $D_2$ generated by the pure exchange $(1,1,1,1,(23))$ and pure invertibility.
\end{claim}

\begin{proof}
The linking matrix for $8^3_2$ is
\begin{center}\begin{tabular}{c|ccc}
&1&2&3\\ \hline
1&0&-1&-1\\
2&-1&0&-2\\
3&-1&-2&0\\
\end{tabular}\end{center}
which is in the standard form $(a,b,b)$.  Consulting Table~\ref{tab:stabs}, we see that the stabilizer of $\Lk(8^3_2)$ has order 4.  Figures \ref{fig:rotate-y} and \ref{832(23)} show
\begin{equation*}
(1,-1,-1,-1,(23)), ~(1,1,1,1,(23)) \in\Sigma(8^3_2)
\end{equation*}
and these generate the 4 element subgroup of the claim, which equals the stabilizer.
\end{proof}

\begin{claim}
The subgroup $\Sigma(8^3_3) < \Gamma_3$ is the 12 element group isomorphic to $D_6$ generated by pure exchanges and pure invertibility.
\end{claim}

\begin{proof}
The linking matrix for $8^3_3$ is
\begin{center}\begin{tabular}{c|ccc}
&1&2&3\\ \hline
1&0&-1&-1\\
2&-1&0&-1\\
3&-1&-1&0\\
\end{tabular}\end{center}
which is in the form $(a,a,a)$. Consulting Table~\ref{tab:stabs}, we see that the stabilizer of $\Lk(8^3_3)$ has order 12, and hence $|\Sigma(8^3_3)|$ divides 12. Now Figures \ref{833(12)} and \ref{833(13)} show that
\begin{equation*}
 (1,1,1,1,(12)),~(1,1,1,1,(13))\in\Sigma(8^3_3).
\end{equation*}
Since the cycles $(12)$ and $(13)$ generate all of $S_3$, we know that all $6$ of the pure exchanges are in $\Sigma(8^3_3)$.  Figure~\ref{833invert} shows that $8^3_3$ is purely invertible as well, completing the proof.  

Finally, we note that we have encountered this symmetry group before, as $\Sigma(6^3_1)=\Sigma(8^3_3)$.
\end{proof}

\begin{claim}
The subgroup $\Sigma(8^3_4) < \Gamma(3)$ is the 4 element group isomorphic to $D_2$ generated by pure invertibility and $(-1,1,1,1,(12))$.
\end{claim}

\begin{proof}
The linking matrix for $8^3_4$ is
\begin{center}\begin{tabular}{c|ccc}
&1&2&3\\ \hline
1&0&0&2\\
2&0&0&-2\\
3&2&-2&0\\
\end{tabular}\end{center}
which is in the standard form $(a,-a,0)$. Consulting Table~\ref{tab:stabs}, we see that the stabilizer of $\Lk(8^3_4)$ is the 8 element group $f^{-1}\left(S(a,-a,0)\right)$, which is generated by $(-1,1,1,-1,e), ~(-1,1,1,1,(12))$, and pure invertibility.  Hence, 
$|\Sigma(8^3_4)|$ divides 8.  Figures~\ref{fig:pi-y} and \ref{834} show that the latter two of these generators, namely pure invertibility and $(-1,1,1,1,(12))$, are in $\Sigma(8^3_4)$.

To finish the proof, we now show that the third stabilizer generator $(-1,1,1,-1,e)$ does not lie in the symmetry group.  Applying it to $8^3_4$, we get a link with HOMFLYPT polynomial
\begin{equation*}
a^4+\frac{1}{a^4}-2 a^2 z^2+\frac{a^2}{z^2}-\frac{2 z^2}{a^2}+\frac{1}{a^2 z^2}+z^4-\frac{2}{z^2}-2.
\end{equation*}
However, the base $8^3_4$ has HOMFLYPT polynomial
\begin{equation*}
a^2 z^4+\frac{z^4}{a^2}+3 a^2 z^2+\frac{3 z^2}{a^2}+\frac{a^2}{z^2}+\frac{1}{a^2 z^2}+4 a^2+\frac{4}{a^2}-z^6-5 z^4-10 z^2-\frac{2}{z^2}-8.
\end{equation*}
This means that $(-1,1,1,-1,e) \notin \Sigma(8^3_4)$ and hence that $\Sigma(8^3_4)$ is generated by $(-1,1,1,1,(12))$ and pure invertibility, as claimed.
\end{proof}

\begin{claim}
The symmetry group $\Sigma(8^3_5) < \Gamma_3$ is the four element group isomorphic to $D_2$ generated by $(1,1,-1,-1,e)$ and $(1,-1,1,1,(23))$.
\end{claim}

\begin{proof}
Figures \ref{835} and \ref{835b} imply that the four element group given above is a subgroup of $\Sigma(8^3_5)$. Now the linking matrix for $8^3_5$ is
\begin{center}
\begin{tabular}{c|ccc}
&1&2&3\\ \hline
1&0&0&0\\
2&0&0&-1\\
3&0&-1&0\\
\end{tabular}
\end{center}
which is in the standard form $(a,0,0)$.  This means that the symmetry group must be a subgroup of the 16 element preimage of $\{ \{+1\} \cross \Z_2 \cross \Z_2 \rtimes \{ e, (23) \}$ in~$\Gamma_3$. Computing this preimage, we observe next that if we apply any of the group elements
\begin{equation*}
\begin{array}{llll}
(-1,-1,-1,1,(23)) & (-1,-1,1,-1,(23)) & (-1,1,-1,1,(23)) & (-1,1,1,-1,(23))\\
(-1,-1,-1,1,e) & (-1,-1,1,-1,e) & (-1,1,-1,1,e) & (-1,1,1,-1,e)
\end{array}
\end{equation*}
in this preimage to our link, we get a link with Jones polynomial
\begin{equation*}
-\frac{1}{z^5}+\frac{3}{z^4}-z^3-\frac{3}{z^3}+3 z^2+\frac{6}{z^2}-4
   z-\frac{5}{z}+6
\end{equation*}
while the base link has Jones polynomial
\begin{equation*}
 -\frac{1}{z^6}+\frac{3}{z^5}-\frac{4}{z^4}+\frac{6}{z^3}-z^2-\frac{5}{z^
   2}+3 z+\frac{6}{z}-3.
\end{equation*}
This rules out those 8 elements, leaving us with a subgroup of 8 possible elements remaining. We claim that among these, the pure exchange $(1,1,1,1,(23))$ is actually ruled out. This claim completes the proof, since any proper subgroup of the 8 element group of remaining elements must have order at most 4, demonstrating that the four element subgroup of $\Sigma(8^3_5)$ generated by the isotopies in Figures~\ref{835} and~\ref{835b} must be the entire group.

Detecting that $(1,1,1,1,(23))$ is ruled out will require us to use the Satellite Lemma (\ref{lem:satellitelemma}). Figure~\ref{835sat} gives us two satellites of $8^3_5$:   one which replaces the second component with a Hopf link and one which replaces the third component with a Hopf link. By the lemma, if the pure exchange is in the symmetry group, these satellites must be isotopic. But taking the Jones polynomial for each, we get the following polynomials:
\begin{equation*}
-\frac{3}{z^{9/2}}+\frac{1}{z^{7/2}}-\frac{2}{z^{5/2}}+z^{3/2}-\frac{1}{z^{3/2}}+\frac{1}{z^{21/2}}
   -\frac{2}{z^{19/2}}+\frac{1}{z^{17/2}}-\frac{3}{z^{13/2}}-2 \sqrt{z}+\frac{1}{\sqrt{z}}
\end{equation*}
and
\begin{equation*}
-\frac{2}{z^{9/2}}+z^{7/2}+\frac{1}{z^{7/2}}-2
   z^{5/2}-\frac{3}{z^{5/2}}+z^{3/2}+\frac{1}{z^{3/2}}+\frac{1}{z^{17/2}}-\frac{2}{z^{15/2}}+\frac{
   1}{z^{13/2}}-2 \sqrt{z}-\frac{3}{\sqrt{z}}.
\end{equation*}
Therefore, they are not isotopic and hence $(1,1,1,1,(23))\notin\Sigma(8^3_5)$, which completes the proof.
\end{proof}

\begin{figure}[ht]
\begin{center}
\scalebox{0.85}{\includegraphics{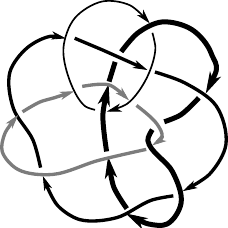}}
\hspace{1in}\scalebox{0.85}{\includegraphics{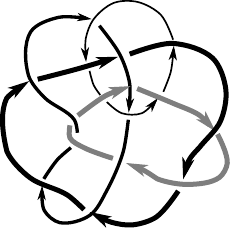}}
\caption[Two satellite links, each with companion link $8^3_5$.]{Two satellite links, each with companion link $8^3_5$. In the left link, component 2 of $8^3_5$ has been replaced by a Hopf link, while in the right link, component 3 of $8^3_5$ has been replaced by a Hopf link. Following Lemma~\ref{lem:satellitelemma}, if the pure exchange symmetry $(1,1,1,1,(23)) \in \Sigma(8^3_5)$, these links must be isotopic. Computing their Jones polynomials shows that they are not.}
\label{835sat}
\end{center}
\end{figure}

\begin{claim}
The subgroup $\Sigma(8^3_6) < \Gamma_3$ is the 8 element group isomorphic to $(\Z_2)^3$ which is generated by
\begin{equation*}
(1,-1,-1,-1,e), ~(1,-1,1,1,(23)),~\text{and\ } (-1,1,-1,-1,e).
\end{equation*}
\end{claim}

\begin{proof}
Figures \ref{fig:pi-y}, \ref{836}, and \ref{836a}, respectively, imply that the three generators above lie in the symmetry group $\Sigma(8^3_6)$.  Now the linking matrix for $8^3_6$ is
\begin{center}\begin{tabular}{c|ccc}
&1&2&3\\ \hline
1&0&1&-1\\
2&1&0&0\\
3&-1&0&0\\
\end{tabular}\end{center}
which corresponds to the triple $(0,-1,1)$ and has the same orbit type as $(a,-a,0)$. Consulting Table~\ref{tab:stabs}, we see that the stabilizer of the linking matrix is precisely the group above, and hence equals $\Sigma(8^3_6)$, which completes the proof.
\end{proof}


\begin{claim}
The subgroup $\Sigma(8^3_7) < \Gamma_3$ is the 4 element group isomorphic to $D_2$ generated by the pure exchange $(1,1,1,1,(23))$ and pure invertibility.
\end{claim}

\begin{proof}
Figures \ref{fig:rotate-y} and \ref{8n3_2} show that $(1,-1,-1,-1,(23),~(1,1,1,1,(23))\in\Sigma(8^3_7)$, which implies that the subgroup above is contained in $\Sigma(8^3_7)$. Now the linking matrix for $8^3_7$ is
\begin{center}\begin{tabular}{c|ccc}
&1&2&3\\ \hline
1&0&1&1\\
2&1&0&-2\\
3&1&-2&0\\
\end{tabular}
\end{center}
which is in the standard form $(a,b,b)$. Consulting Table~\ref{tab:stabs}, we see that $|\Sigma(8^3_7)|$ divides 4. Since we already have 4 elements in the group, this completes the proof.

Finally, we note that we have encountered this symmetry group before, as $\Sigma(8^3_2)=\Sigma(8^3_7)$.
\end{proof}

\begin{claim}
The subgroup $\Sigma(8^3_8) < \Gamma_3$ is the four element subgroup isomorphic to $D_2$ generated by pure invertibility and $(1,1,-1,-1,(23))$.
\end{claim}

\begin{proof}
Figures \ref{8n4_1} and \ref{8n4_2} show that $(1,-1,1,1,(23),~(1,1,-1,-1,(23))\in\Sigma(8^3_8)$, which implies that $\Sigma(8^3_8)$ contains the claimed group. Now the linking matrix for $8^3_8$ is
\begin{center}\begin{tabular}{c|ccc}
&1&2&3\\ \hline
1&0&1&-1\\
2&1&0&-2\\
3&-1&-2&0\\
\end{tabular}
\end{center}
which is in the standard form $(a,b,-b)$. Consulting Table~\ref{tab:stabs}, we see that $|\Sigma(8^3_8)|$ divides 4. This completes the proof, since we already have 4 elements in the subgroup.

Finally, we note that we have encountered this symmetry group before, as $\Sigma(8^3_1)=\Sigma(8^3_8)$.
\end{proof}

\begin{claim}
The subgroup $\Sigma(8^3_9) < \Gamma_3$ is the 8 element group isomorphic to $(\Z_2)^3$ generated by
pure invertibility, the pure exchange $(1,1,1,1,(23))$ and $(1,-1,1,1,e)$.
\end{claim}

\begin{proof}
Figures \ref{fig:rotate-y}, \ref{8n5_2}, and \ref{8n5_3} show that $(1,-1,-1,-1,(23),~(1,1,1,1,(23)),~(1,-1,1,1,e)\in\Sigma(8^3_9)$, which implies that the 8 element subgroup generated by these elements is a subgroup of $\Sigma(8^3_9)$. Now the linking matrix for this link is
\begin{center}\begin{tabular}{c|ccc}
&1&2&3\\ \hline
1&0&0&0\\
2&0&0&2\\
3&0&2&0\\
\end{tabular}
\end{center}
which is in the standard form $(a,0,0)$. Consulting Table~\ref{tab:stabs}, we see that $|\Sigma(8^3_9)|$ divides 16.  Working out these 16 elements as in the case of $8^3_5$, we see that if we apply any of the elements
\begin{equation*}
\begin{array}{llll}
(-1,-1,-1,1,(23)) & (-1,-1,1,-1,(23)) & (-1,1,-1,1,(23)) & (-1,1,1,-1,(23))\\
(-1,-1,-1,1,e) & (-1,-1,1,-1,e) & (-1,1,-1,1,e) & (-1,1,1,-1,e)
\end{array}
\end{equation*}
to $8^3_9$, we get a link with Jones polynomial
\begin{equation*}
2 z^5-2 z^4+4 z^3-2 z^2+3 z+\frac{1}{z}-2,
\end{equation*}
while the Jones polynomial of the base $8^3_9$ link is
\begin{equation*}
z^7-2 z^6+3 z^5-2 z^4+4 z^3-2 z^2+2 z.
\end{equation*}
This leaves only the subgroup claimed. We note that while $\Lk(8^3_5)$ and $\Lk(8^3_9)$ have the same stabilizer, the symmetry group $\Sigma(8^3_5)$ is a proper subgroup of $\Sigma(8^3_9)$.
\end{proof}


\begin{claim}
The subgroup $\Sigma(8^3_{10}) < \Gamma_3$ is the 4 element group isomorphic to $D_2$ generated by pure invertibility and the pure exchange $(1,1,1,1,(12))$.
\end{claim}

\begin{proof}
Figures~\ref{fig:pi-y} and \ref{8n6_2} show that $(1,-1,-1,-1,e)$ and $(1,1,1,1,(12))$ are in
$\Sigma(8^3_{10})$. Now the linking matrix for $8^3_{10}$ is
\begin{center}\begin{tabular}{c|ccc}
&1&2&3\\ \hline
1&0&0&2\\
2&0&0&2\\
3&2&2&0\\
\end{tabular}\end{center}
which has the standard form $(a,a,0)$.  Consulting Table~\ref{tab:stabs}, we see that this matrix has an 8 element stabilizer in $\Gamma_3$ given by the inverse image of $\left<(1,1,1,(12)),(1,1,-1,e)\right>$ under the map $f$. Now we observe that if we apply any of the four elements
\begin{equation*}
\begin{array}{llll}
(-1,-1,-1,1,(12)) & (-1,1,1,-1,(12)) & (-1,-1,-1,1,e) & (-1,1,1,-1,e)
\end{array}
\end{equation*}
in this stabilizer to $8^3_{10}$, we get a link with Jones polynomial
\begin{equation*}
z^{10}+z^6+z^5+z^3,
\end{equation*}
while the base $8^3_{10}$ link has Jones polynomial
\begin{equation*}
z^9+z^7+z^6+z^2.
\end{equation*}
This rules out all but the four element subgroup above, completing the proof.
\end{proof}

\bigskip
\section{Isotopies for four-component links}

There are three prime four-component links with 8 crossings. They are quite similar in appearance, with only some crossing changes distinguishing them. Their symmetry computations are made somewhat more difficult by the fact that we are working in the 768 element group $\Gamma_4$. All three of these links are composed of four unknots linked together so that component $1$ is linked to $2$ and $3$, and $2$ and $3$ are linked to $4$.

Here are the symmetry groups for these links, listed in terms of generators for each group.  Again, we denote the purely invertible symmetry, i.e., element $(1,-1,-1,-1,-1,e)$, by PI; all three links admit this symmetry. 

\begin{table}[ht]
\begin{center}
\begin{tabular}{cccl}
\toprule
Link(s) & $|\Sigma(L)|$ & $\Sigma(L)$ isomorphic to &  Generators  \\
\midrule
$8^4_1 \; (8a21)$ & 16 & $\Z_2 \times D_4$ &  PI, $(1,1,1,1,1,(23)), ~(1,1,1,1,1,(1243))$\\
\addlinespace[0.4em]
$8^4_2 \; (8n7)$ & 16 & $D_8$ &  $(1,1,1,1,1,(13)(24)), ~(1,1,1,-1,1,(1243))$\\
\addlinespace[0.6em]
$8^4_3 \; (8n8)$ & 32 & $\Z_2 \times \Z_2 \times D_4$ &  \begin{tabular}{l}
		PI, $(-1,1,1,1,1,(23)),$ \\
		$(-1,1,1,1,1,(1243)), ~(-1,-1,1,1,-1,e)$\\
		\end{tabular} \\
\bottomrule
\end{tabular}
\end{center}
\vspace*{0.5em}
\caption{Four-component link symmetry groups, by crossing number}
\label{tab:4symm}
\end{table}

\vspace*{-0.2in}
Our approach mimics the one used for three-component links.  After calculating the linking matrix, we utilize Table~\ref{tab:4stabs} to determine the stabilizer of this matrix within $\Gamma_4$;  we know that the symmetry group $\Sigma(L)$ must be a subgroup of this stabilizer.  Next, we show some elements are in the symmetry group by exhibiting isotopies and rule others out using polynomial invariants until we can discern the symmetry group.


\begin{claim} The symmetry subgroup for $8^4_1$ is the 16 element group isomorphic to $\Z_2 \cross D_4$ given by the $S^3$-orientation-preserving elements of the inverse image $f^{-1}(S(a,a,a,a))$.
\end{claim}

\begin{proof}
The linking matrix for $8^4_1$ is
\begin{equation*}
\begin{array}{c|cccc}
 & 1 & 2 & 3 & 4 \\
 \hline
 1 & 0 & 1 & 1 & 0 \\
 2 & 1 & 0 & 0 & 1 \\
 3 & 1 & 0 & 0 & 1 \\
 4 & 0 & 1 & 1 & 0
\end{array}
\end{equation*}
which corresponds to the standard form $(a,a,a,a)$. By Lemma~\ref{lem:stab4}, $\Sigma(8^4_1)$ must be a subgroup of the 32 element stabilizer $f^{-1}(S(a,a,a,a))$ of the linking matrix. 

Figures~\ref{841invert}, \ref{841_2}, and \ref{fig:pe-z} show that pure invertibility and$(1,-1,-1,-1,-1,(23))$ and $(1,1,1,1,1,(1342))$ are all in $\Sigma(8^4_1)$.  Together, these generate the 16 element group of the claim.  

We must show that the 16 $S^3$-orientation-reversing elements of the stabilizer, (i.e., elements of the form $(-1,\epsilon_1, \dots, \epsilon_4,p)$) are not in the symmetry group.  If we apply any of these 16 elements to the base link, we obtain a link with Jones polynomial
\begin{equation*}
-5 z^{9/2}+z^{7/2}-z^{5/2}-z^{21/2}+3 z^{19/2}-6 z^{17/2}+4 z^{15/2}-7 z^{13/2}+4 z^{11/2}.
\end{equation*}
But the Jones polynomial of the base $8^4_1$ is
\begin{equation*}
4 z^{9/2}-6 z^{7/2}+3 z^{5/2}-z^{3/2}-z^{19/2}+z^{17/2}-5 z^{15/2}+4 z^{13/2}-7 z^{11/2},
\end{equation*}
This rules out all 16 of these remaining elements, which proves the claim.
\end{proof}


\begin{claim} The symmetry subgroup for $8^4_2$ is the 16 element group isomorphic to $D_8$ given by the $S^3$-orientation-preserving elements of $f^{-1}(S(a,a,-a,a))$.
\end{claim}

\begin{proof}
The linking matrix for $8^4_2$ is
\begin{equation*}
\begin{array}{c|cccc}
 & 1 & 2 & 3 & 4 \\
 \hline
1 & 0 & 1 & -1 & 0 \\
2 &  1 & 0 & 0 & 1 \\
3 & -1 & 0 & 0 & 1 \\
4 & 0 & 1 & 1 & 0
\end{array}
\end{equation*}
which corresponds to the standard form $(a,a,-a,a)$ (remember that the ordering of elements given by~\eqref{eq:gamma4tolink4} is not obvious).  By Lemma~\ref{lem:stab4} the stabilizer of this linking matrix in $\Gamma_4$ is a 32 element subgroup isomorphic to $\Z_2 \cross D_8$. 

Figures~\ref{842_2} and \ref{842_3} show that  $(1,1,1,-1,1,(14))$ and $(1,1,1,1,1,(13)(24))$ are in $\Sigma(8^4_1)$.  Together, these generate the 16 element group of the claim.  

We must show that the 16 $S^3$-orientation-reversing elements of the stabilizer, (i.e., elements of the form $(-1,\epsilon_1, \dots, \epsilon_4,p)$) are not in the symmetry group.  If we apply any of these 16 elements to the base link, we obtain a link with Jones polynomial
\begin{equation*}
2 z^{9/2}-4 z^{7/2}+z^{5/2}-4 z^{3/2}-3 z^{11/2}+\sqrt{z}-\frac{1}{\sqrt{z}},
\end{equation*}
while the base link has Jones polynomial
\begin{equation*}
-4 z^{9/2}+z^{7/2}-4 z^{5/2}+2 z^{3/2}-z^{13/2}+z^{11/2}-3 \sqrt{z}.
\end{equation*}
This rules out these 16 remaining elements, so the claim is proven.
\end{proof}


\begin{claim} The symmetry subgroup for $8^4_3$ is the 32 element group isomorphic to $\Z_2 \cross \Z_2 \cross D_4$ given by $f^{-1}(S(a,-a,-a,a))$.
%
\end{claim}

\begin{proof}
The linking matrix for $8^4_3$ is
\begin{equation*}
\begin{array}{c|cccc}
 & 1 & 2 & 3 & 4 \\
 \hline
1 & 0 & 1 & -1 & 0 \\
2 & 1 & 0 & 0 & -1 \\
3 & -1 & 0 & 0 & 1 \\
4 & 0 & -1 & 1 & 0
\end{array}
\end{equation*}
which corresponds to the standard form $(a,-a,-a,a)$.  By Lemma~\ref{lem:stab4} the stabilizer of this linking matrix in $\Gamma_4$ is a 32 element subgroup isomorphic to $\Z_2 \cross \Z_2 \cross D_4$.  This group is generated by the isotopies in Figures~\ref{843invert}, \ref{843_1}, \ref{843_2}, and \ref{843_3}, which show that pure invertibility, along with elements
\begin{equation*}
(-1,1,1,1,1,(23)),~~(-1,-1,1,1,-1,\id),~~(-1,1,1,1,1,(1342)) \in \Sigma(8^4_3).
\end{equation*}
\end{proof}

\section{Comparison of intrinsic symmetry groups with ordinary symmetry groups for links}

We now compare our results on intrinsic symmetry groups to the existing literature on symmetry groups for links. Henry and Weeks~\cite{MR1164115,MR1241189} report $\Sym(L)$ groups for hyperbolic links up to 9 crossings, while Boileau and Zimmerman~\cite{MR891592} computed $\Sym(L)$ groups for nonelliptic Montesinos links with up to 11 crossings, and Bonahon and Siebenmann computed $\Sym(L)$ for the Borromean rings link ($6^3_2$) as an example of their methods in~\cite[Theorem 16.18]{bs}.

Comparing all this data with ours, we see that
\begin{lemma}
Among all links of 8 and fewer crossings with known $\Sym(L)$ groups, the Whitten symmetry group $\Sigma(L)$ is not isomorphic to $\Sym(L)$ only for the links in Table~\ref{tab:noteq}.
\renewcommand{\arraystretch}{1.2}
\begin{table}[ht]
\begin{center}
\begin{tabular}{ccccc} 
\toprule
Link & $\Sigma(L)$ & $|\Sigma(L)|$ & $\Sym(L)$ & $|\Sym(L)|$ \\ 
\midrule
$6^2_3$ & $D_2$ & $4$ & $(\Z_2)^3$ & $8$ \\
$8^2_5$ &  $\Z_2$ & $2$ & $D_2$ & $4$ \\
$8^2_4$ & $D_2$ & $4$ & $(\Z_2)^3$ & $8$ \\
$8^2_6$ & $D_2$ & $4$ & $(\Z_2)^3$ & $8$ \\
$8^2_7$ & $D_2$ & $4$ & $(\Z_2)^3$ & $8$ \\
$8^2_9$ &  $\Z_2$ & $2$ & $D_2$ & $4$ \\
$8^2_{11}$  & $\Z_2$ & $2$ & $D_2$ & $4$ \\
$8^2_{14}$  & $\Z_2$ & $2$ & $D_2$ & $4$ \\
$8^2_{16}$  & $\Z_2$ & $2$ & $(\Z_2)^3$ & $8$ \\
$8^3_4$ & $D_2$ & $4$  & $D_4$ & $8$ \\
\bottomrule
\end{tabular}
\caption{Among links with 8 or fewer crossings with known $\Sym(L)$ groups, precisely 10 have $\Sigma(L) \ncong \Sym(L)$. \label{tab:noteq}}
\end{center}
\end{table}
\end{lemma}

Our results provide some data on symmetry groups of torus links as well.
\begin{lemma}
For the $(2,4)$, $(2,6)$ and $(2,8)$ torus links ($4^2_1$, $6^2_1$, and $8^2_1$), we have $\Sym(L) \simeq \Sigma(L) \simeq D_2$. For the Hopf link, the $(2,2)$ torus link, we know that $\Sym(L) \simeq \Sigma(L) \simeq (\Z_2)^3$.
\end{lemma}
\begin{proof}
Goldsmith computed~\cite{MR672923} the ``motion groups'' $M(S^3,L_{(np,nq)})$ for torus knots and links in $S^3$. Combining her Corollary 1.13 and Theorem 3.7, we see that for the $(np,nq)$ torus link $L_{(np,nq)}$, the subgroup $\Sym^+$ of orientation preserving symmetries is homeomorphic to $M(S^3,L_{(np,nq)})$ if $p+q$ is odd, and an index two quotient group of $M(S^3,L_{(np,nq)})$ if $p+q$ is even. The motion group itself is either $D_2$ or an 8-element quaternion group. But in either case $\Sym^+ \simeq D_2$. 

Now the motion group by itself does not provide any information about the orientation reversing elements in $\Sym$. However, any such element in $\Sym \setminus \Sym^+$ would map to an element in $\Sigma(L_{(np,nq)})$ which reversed orientation on $S^3$. Since we have already shown that there are no such elements in $\Sigma(L_{(np,nq)})$ for $n=1$, $p=1$, $q = 2, 3, 4$, we see that for these links $\Sym = \Sym^+ = \Sigma$. For the Hopf link such an orientation reversing element does exist in $\Sym(L)$. So, $\Sym(L)$ is a $Z_2$ extension of $D_2$ and thus is isomorphic to $(\Z_2)^3$
\end{proof}

By Proposition~\ref{prop:image}, the group $\Sigma(L)$ is the image of $\Sym(L)$ under a homomorphism, and hence a quotient group of $\Sym(L)$. Further, if $\Sigma(L)$ has only orientation-preserving elements (on $S^3$) then $\Sym(L)$ does as well. Thus we know something about the $\Sym(L)$ groups of all our links; Table~\ref{tab:nonhyperbolic} summarizes the new information provided by our approach.
\renewcommand{\arraystretch}{1.2}
\begin{table}[ht]
\begin{tabular}{lcc} \toprule
Link       & $|\Sym(L)|$ divisible by & $\Sym(L)$ has quotient\\ 
\midrule
$6^3_1$    & 12 & $D_6$ \\
$6^3_3$    & 12 & $D_6$ \\
$8^3_6$    & 8 & $(\Z_2)^3$ \\
$8^3_7$    & 4 & $D_2$ \\
$8^3_8$    & 4 & $D_2$  \\
$8^3_9$    & 8 & $(\Z_2)^3$ \\
$8^3_{10}$ & 4 & $D_2$ \\
\bottomrule
\end{tabular}
\smallskip
\caption[New results about $\Sym(L)$ symmetry groups]{New results about $\Sym(L)$ symmetry groups:  we know from Proposition~\ref{prop:image} that $\Sigma(L)$ is a quotient group of $\Sym(L)$. Among the links for which we have computed $\Sigma(L)$ are a collection of links for which we can find no information in the literature on $\Sym(L)$. Our results imply that the unknown $\Sym(L)$ groups listed above must have a certain quotient.  \label{tab:nonhyperbolic}}
\end{table}
\renewcommand{\arraystretch}{1}

\section{Future directions}
We have now presented explicit computations of the Whitten symmetry groups for all links with 8 and fewer crossings. In all the cases we studied, the most difficult part of the computation was obtaining explicit isotopies to generate the symmetry group; ruling out the remaining elements of $\Gamma_n$ was generally done by the application of classical invariants. The most difficult of these cases required us to use the ``satellite lemma'' and study the classical invariants of satellites of our original link.

While we have presented conventional proofs of all our results, we used computer methods extensively in determining the right line of attack for each link -- our method was to use the Mathematica package \texttt{KnotTheory} to systematically apply all possible Whitten group elements to each link and then check the knot types of the components, the linking matrix, the Jones polynomial, and the HOMFLYPT polynomial in an attempt to distinguish the new link from the original one. We then checked the computer calculations by hand. This automated method clearly cannot compute $\Sigma(L)$, but it does provide a subgroup $\Sigma'(L)$ of $\Gamma(L)$ which is known to contain $\Sigma(L)$. While we do not currently intend to generate isotopies for links with higher crossing number, we intend to present our computationally-obtained $\Sigma'(L)$ groups for 9, 10, and 11 crossing links in a future publication.

Some of the most natural questions about the Whitten symmetry groups remain unanswered by this type of explicit enumeration: which groups can arise as Whitten groups? Does every subgroup of a given $\Gamma_n$ arise as a $\Sigma(L)$? We have observed 6 different subgroups of $\Gamma_2$, 11 different subgroups of $\Gamma_3$, and $3$ different subgroups of $\Gamma_4$ so far. This subject is certainly worth further exploration: can one generate a carefully chosen link with a given symmetry group?

\section{Acknowledgements}
We are grateful to all the members of the UGA Geometry VIGRE group who contributed their time and effort over many years towards this project. In particular, we'd like to acknowledge the contributions of graduate students Ted Ashton, Yang Liu, Steve Lane, Laura Nunley,  Gregory Schmidt, Jae-Ho Shin, and Joe Tenini and undergraduate students Daniel Cellucci, James Dabbs, Alex Moore, and Emmanuel Obi. We also had a series of helpful conversations with Erica Flapan.  Jeremy Rouse assisted with our Magma computations of subgroups of $\Gamma_\mu$.  We are also grateful to the anonymous referee, who made a number of helpful suggestions as we revised the paper. Our work was supported by the UGA VIGRE grants DMS-07-38586 and DMS-00-89927 and by the UGA REU site grant DMS-06-49242.

\bibliographystyle{amsalpha}
\bibliography{drl,symmetrypaper}

\newpage
\appendix

\section{Guide to link isotopy figures} \label{app:guide}
The Appendices contain figures representing the 101 isotopies presented in our paper; they are organized as follows.
\begin{itemize}
\item Appendix~\ref{app:guide} explains the moves depicted in these isotopies.
\item Appendix~\ref{app:rotate} exhibits 41 isotopies which can be found by simply rotating about one axis.
\item Appendix~\ref{app:2isotopy} contains the remaining isotopies for 2-component links
\item Appendix~\ref{app:3isotopy} contains the remaining isotopies for 3-component links
\item Appendix~\ref{app:4isotopy} contains the remaining isotopies for 4-component links
\end{itemize}

\medskip
For coordinates on our figures, we assume the diagrams appear in the $xy$-plane with the $z$-axis coming out of the page.
With this orientation, Figure~\ref{fig:rotations} shows symbols for various rotations around lines.  We will also denote moves during an isotopy with arrows, as shown in Figure~\ref{fig:arrowmoves}.

\medskip
\begin{figure}[ht]
\begin{center} 
\hfill
\begin{tabular}{c}
\includegraphics[width=0.7in]{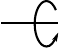} \\
$x$-axis
\end{tabular}
\hfill
\begin{tabular}{c}
\includegraphics[height=0.7in]{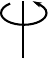} \\
$y$-axis
\end{tabular}
\hfill
\begin{tabular}{c}
\includegraphics[height=0.65in]{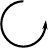} \\
$z$-axis
\end{tabular}
\hfill
\begin{tabular}{c}
\includegraphics[height=0.75in]{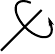} \\
line $y=x$
\end{tabular}
\hfill
\begin{tabular}{c}
\includegraphics[height=0.75in]{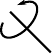} \\
line $y=-x$
\end{tabular}
\hfill
\end{center}
\caption{Rotation by $\pi$ radians around various axes is denoted by the symbols above.\label{fig:rotations}}
\end{figure}

\bigskip
\begin{figure}[ht]
\begin{center}
\begin{tabular}{c}
\includegraphics[width=3.4in]{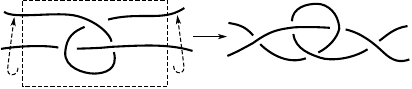} 
\end{tabular}
\hspace{0.25in}
\begin{tabular}{c}
\includegraphics[width=2.5in]{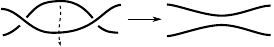}
\end{tabular}
\end{center}
\caption[Examples of moves in isotopy diagrams]{The left figure shows a ``flip'' move. In this move, we take the portion of the link in the dotted box and flip it in the
direction indicated by the curved arrows, resulting in the diagram shown on the right. Arrows denote transitions between stages
of the isotopy. The right figure shows a simpler transformation. Here the arrow shows a portion of the link. Whether this portion 
moves over or under intervening portions of the link should be clear from the next drawing. 
\label{fig:arrowmoves}}
\end{figure}
In all of our diagrams, the component numbered $1$ is drawn with a thin (1 pt) line, while other components are denoted by
thicker (2, 3, or 4 pt lines). We also number the components explicitly to prevent any confusion on this point. As usual, arrows 
denote orientation on the link components.


\newpage
\section{Isotopy Figures found by rotations}  \label{app:rotate}

In this section, we display the simplest type of isotopies:  mere rotation of a link about a coordinate axis.  Such a rotation suffices for 41 of the 101 isotopies we present in this paper.  

We display these in three subsections.  The first covers the pure exchange isotopies, which only permute the components; these correspond to Whitten group elements $\gamma=(1,1,\ldots,1,p)$.   The second subsection covers the pure invertibility isotopies, which only invert each components and correspond to $\gamma= (1,-1,-1, \ldots,-1,e)$.  The third subsection contains 10 other isotopies found by rotating the figure.

\bigskip

\subsection{Pure exchange isotopies found by rotation} \label{app:rotate-pe}

\begin{center}
\begin{figure}[ht]
\begin{tabular}[c]{m{1.45in}m{1.45in}m{1.45in}m{1.45in}}
\begin{center}
\fbox{\includegraphics[scale=3]{y_axis.pdf}}
\end{center}
&
\begin{overpic}[scale=0.9]{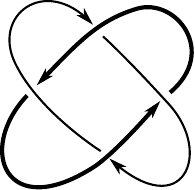}
\put(12,80){1}
\put(103,80){2}
\end{overpic}

\begin{center}
$4^2_1$
\end{center}
&
\begin{overpic}[scale=0.9]{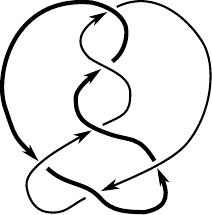}
\put(90,92){1}
\put(3,92){2}
\end{overpic}

\begin{center}
$6^2_2$
\end{center}
&
\begin{overpic}[scale=0.9]{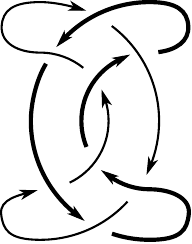}
\put(5,59){1}
\put(83,80){2}
\end{overpic}

\begin{center}
$8^2_4$
\end{center}
\\
&&&\\
\begin{center}
\begin{overpic}[scale=0.9]{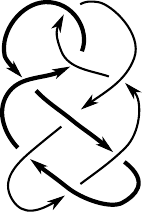}
\put(68,80){1}
\put(-9,80){2}
\end{overpic}

$8^2_5$
\end{center}
&
&
&
\end{tabular}
\caption{Pure exchange isotopies found by rotation about the $y$-axis.} \label{fig:pe-y}
\end{figure}
\end{center}

\newpage

\begin{center}
\begin{figure}[ht]
\begin{tabular}[c]{m{1.45in}m{1.45in}m{1.45in}m{1.45in}}
\begin{center}
\fbox{\includegraphics[scale=3]{z_axis.pdf}}
\end{center}
&
\begin{overpic}[scale=0.9]{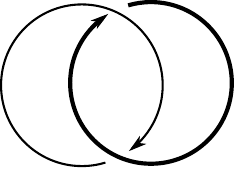}
\put(14,72){1}
\put(80,72){2}
\end{overpic}

\begin{center}
$2^2_1$
\end{center}
&
\begin{overpic}[scale=0.9]{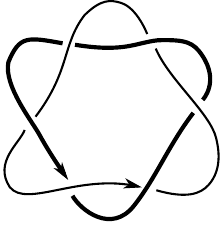}
\put(65,95){1}
\put(10,87){2}
\end{overpic}

\begin{center}
$6^2_1$
\end{center}
&
\begin{center}
\begin{overpic}[scale=0.9]{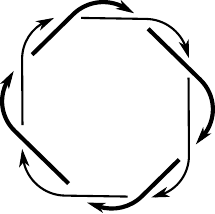}
\put(5,85){1}
\put(47,104){2}
\end{overpic}

$8^2_1$
\end{center}
\\
&&&\\
\begin{overpic}[scale=0.9]{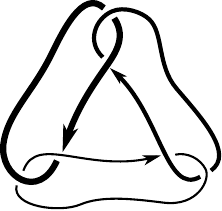}
\put(71,86){1}
\put(11,65){2}
\put(55,-6){3}
\end{overpic}

& 
\begin{overpic}[scale=0.78]{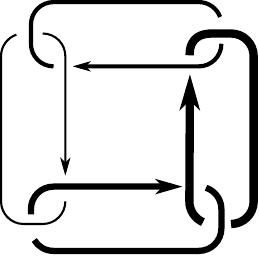}
\put(-6,50){1}
\put(50,101){2}
\put(50,6){3}
\put(102,50){4}
\end{overpic}
&& \\
\begin{center}
$6^3_1, \, (123)$
\end{center}
&
\begin{center}
$8^4_1, \,  (1342)$
\end{center}
&
&
\end{tabular}
\caption[Pure exchange isotopies found by through rotation about the $z$-axis]{Pure exchange isotopies found by rotation about the $z$-axis. In the case of links with more than two components, all possible exchanges are listed underneath the respective figures.} \label{fig:pe-z}
\end{figure}
\end{center}

\newpage

\subsection{Pure invertibility isotopies found by rotation} \label{app:rotate-pi}

\begin{center}
\begin{table}[ht]
\begin{tabular}[c]{m{1.45in}m{1.45in}m{1.45in}m{1.45in}}

\fbox{\includegraphics[scale=3]{y_axis.pdf}}
&
\begin{center}
\begin{overpic}[scale=0.9]{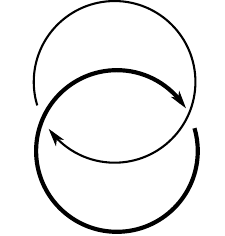}
\put(11,90){1}
\put(82,5){2}
\end{overpic}

$2^2_1$
\end{center}
&
\begin{overpic}[scale=0.9]{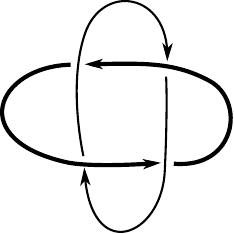}
\put(28,89){1}
\put(87,20){2}
\end{overpic}

\begin{center}
$4^2_1$
\end{center}
&
\begin{center}
\begin{overpic}[scale=0.9]{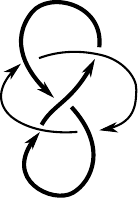}
\put(-10,53){1}
\put(55,85){2}
\end{overpic}

$5^2_1$
\end{center}
\\
&&&\\

\begin{overpic}[scale=0.9]{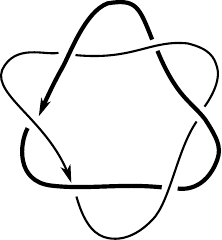}
\put(10,83){1}
\put(65,92){2}
\end{overpic}

\begin{center}
$6^2_1$
\end{center}
&
\begin{overpic}[scale=0.9]{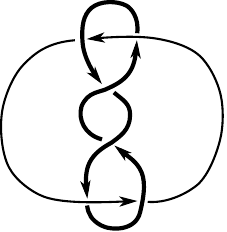}
\put(10,80){1}
\put(63,92){2}
\end{overpic}

\begin{center}
$6^2_3$
\end{center}
&
\begin{overpic}[scale=0.9]{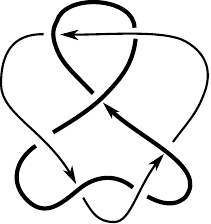}
\put(-2,43){1}
\put(56,99){2}
\end{overpic}

\begin{center}
$7^2_1$
\end{center}
&
\begin{overpic}[scale=0.9]{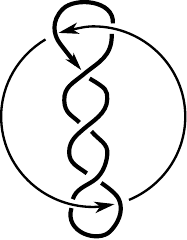}
\put(5,80){1}
\put(50,60){2}
\end{overpic}

\begin{center}
$7^2_3$
\end{center}
\\
&&&\\

\begin{overpic}[scale=0.9]{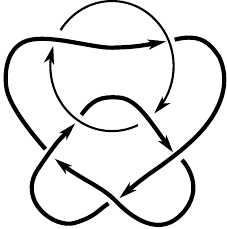}
\put(40,102){1}
\put(85,87){2}
\end{overpic}

\begin{center}
$7^2_4$
\end{center}
&
\begin{overpic}[scale=0.9]{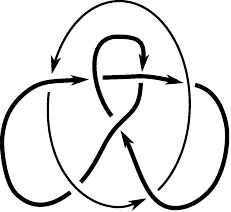}
\put(30,90){1}
\put(97,53){2}
\end{overpic}

\begin{center}
$7^2_5$
\end{center}
&
\begin{overpic}[scale=0.9]{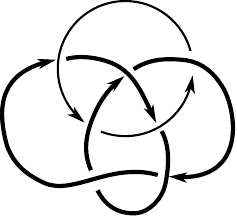}
\put(40,93){1}
\put(8,8){2}
\end{overpic}

\begin{center}
$7^2_6$
\end{center}
&
\begin{overpic}[scale=0.9]{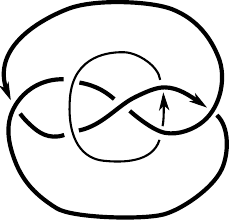}
\put(50,76){1}
\put(85,92){2}
\end{overpic}

\begin{center}
$7^2_7$
\end{center}
\\
&&&\\
\begin{overpic}[scale=0.9]{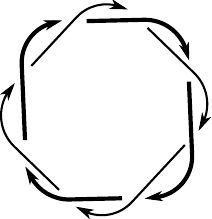}
\put(45,102){1}
\put(80,90){2}
\end{overpic}

\begin{center}
$8^2_1$
\end{center}
&
\begin{overpic}[scale=0.9]{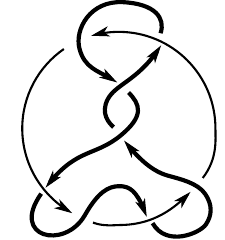}
\put(84,75){1}
\put(50,102){2}
\end{overpic}

\begin{center}
$8^2_3$
\end{center}
&
\begin{overpic}[scale=0.9]{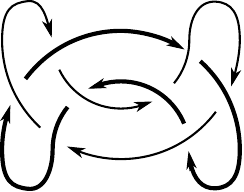}
\put(10,82){1}
\put(50,70){2}
\end{overpic}

\begin{center}
$8^2_4$
\end{center}
&
\begin{overpic}[scale=0.9]{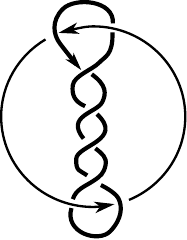}
\put(15,93){1}
\put(81,47){2}
\end{overpic}

\begin{center}
$8^2_6$
\end{center}

\end{tabular}
\end{table}
\end{center}

\newpage

\begin{center}
\begin{figure}[ht]
\begin{tabular}[c]{m{1.45in}m{1.45in}m{1.45in}m{1.45in}}

\begin{overpic}[scale=0.9]{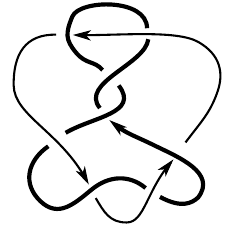}
\put(8,86){1}
\put(5,5){2}
\end{overpic}

\begin{center}
$8^2_7$
\end{center}
&
\begin{overpic}[scale=0.9]{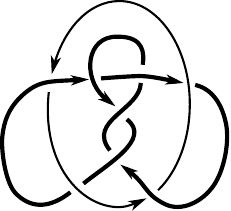}
\put(70,86){1}
\put(99,49){2}
\end{overpic}

\begin{center}
$8^2_9$
\end{center}
&
\begin{overpic}[scale=0.9]{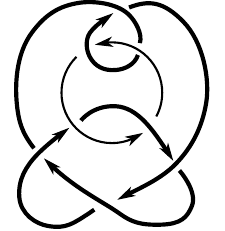}
\put(20,63){1}
\put(10,92){2}
\end{overpic}

\begin{center}
$8^2_{11}$
\end{center}
&
\begin{overpic}[scale=0.9]{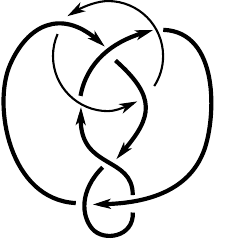}
\put(50,102){1}
\put(85,85){2}
\end{overpic}

\begin{center}
$8^2_{13}$
\end{center}
\\
&&&\\
\begin{overpic}[scale=0.9]{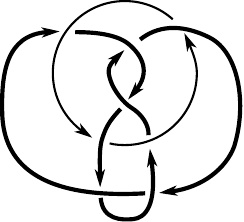}
\put(39,93){1}
\put(89,80){2}
\end{overpic}

\begin{center}
$8^2_{14}$
\end{center}
&
\begin{overpic}[scale=0.9]{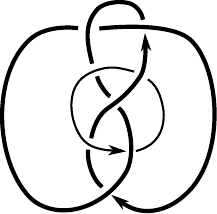}
\put(23,52){1}
\put(80,90){2}
\end{overpic}

\begin{center}
$8^2_{16}$
\end{center}
&
\begin{overpic}[scale=0.9]{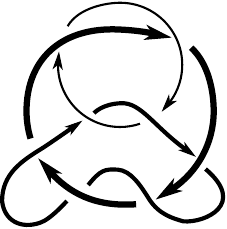}
\put(50,102){1}
\put(-5,23){2}
\put(94,76){3}
\end{overpic}

\begin{center}
$8^3_1$
\end{center}
&
\begin{center}
\begin{overpic}[scale=0.9]{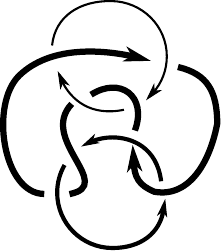}
\put(59,97){1}
\put(69,3){2}
\put(87,70){3}
\end{overpic}

$8^3_4$
\end{center}
\\
&&&\\
\begin{overpic}[scale=0.9]{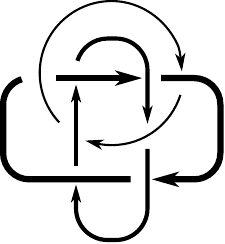}
\put(72,93){1}
\put(61,1){2}
\put(4,16){3}
\end{overpic}
&
\begin{center}
\begin{overpic}[scale=0.9]{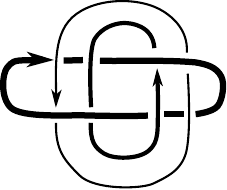}
\put(55,87){1}
\put(52,17){2}
\put(99,58){3}
\end{overpic}
\end{center}
&&\\

\begin{center}
$8^3_6$
\end{center}
&\begin{center}
$8^3_{10}$
\end{center}
&
\end{tabular}
\caption{Pure inversion isotopies found by rotation about the $y$-axis.} \label{fig:pi-y}
\end{figure}
\end{center}

\newpage
\subsection{Other isotopies found by rotation}

\begin{center}
\begin{figure}[ht]
\begin{tabular}[c]{m{1.45in}m{1.45in}m{1.45in}m{1.45in}}
\begin{center}
\fbox{\includegraphics[scale=3]{y_axis.pdf}}
\end{center}
&
\begin{overpic}[scale=1.1]{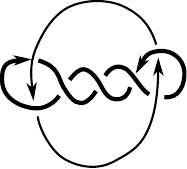}
\put(15,80){1}
\put(50,58){2}
\end{overpic}

\begin{center}
$7^2_3, \, (1,-1,1,e)$
\end{center}
&
\begin{overpic}[scale=0.95]{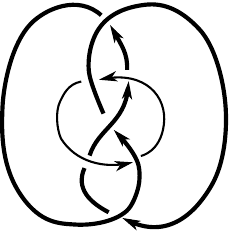}
\put(17,50){1}
\put(5,90){2}
\end{overpic}

\begin{center}
$7^2_8, \, (1,-1,1,e)$
\end{center}
&
\begin{center}
\begin{overpic}[scale=0.9]{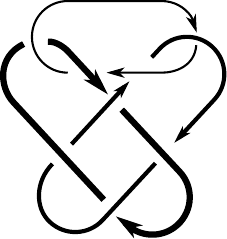}
\put(48,89){1}
\put(99,70){2}
\put(-7,70){3}
\end{overpic}

$8^3_2, \, (1,\minus 1\minus 1,\minus 1,(23))$
\end{center}
\\
&&&\\

\begin{overpic}[scale=0.9]{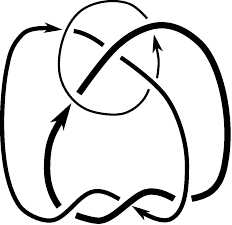}
\put(45,101){1}
\put(10,90){2}
\put(80,90){3}
\end{overpic}

$8^3_7, \, (1,\minus 1,\minus 1,\minus 1,(23))$
&

\begin{overpic}[scale=0.9]{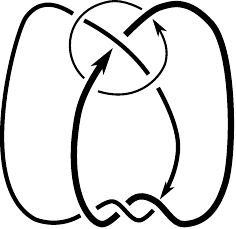}
\put(45,101){1}
\put(5,95){2}
\put(91,95){3}
\end{overpic}

$8^3_9, \, (1,\minus 1,\minus 1,\minus 1,(23))$
&&\\
\end{tabular}
\caption[Whitten group elements effected through rotation about the $y$-axis]{Whitten group elements effected through rotation about the $y$-axis. Group elements are listed under their respective figure.} \label{fig:rotate-y}
\end{figure}
\end{center}

\vspace*{0.5in}
\begin{center}
\begin{figure}[ht]
\begin{tabular}[c]{m{1.45in}m{1.45in}m{1.45in}m{1.45in}}

\fbox{\includegraphics[scale=3]{z_axis.pdf}}
&
\begin{center}
\begin{overpic}[scale=0.9]{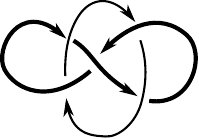}
\put(50,73){1}
\put(8,12){2}
\end{overpic}
\end{center}

&
\begin{overpic}[scale=0.9]{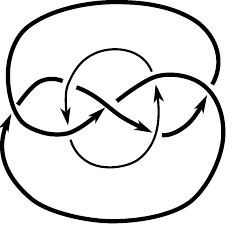}
\put(47,15){1}
\put(87,89){2}
\end{overpic}

&
\begin{overpic}[scale=0.9]{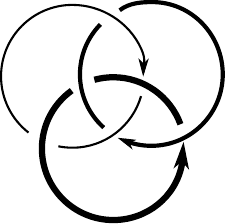}
\put(5,92){1}
\put(96,86){2}
\put(82,10){3}
\end{overpic}
\\
&
\begin{center}
$5^2_1, \, (1,1,-1,e)$
\end{center}
&
\begin{center}
$7^2_8, \, (1,1,-1,e)$
\end{center}
&
\begin{center}
$6^3_2,$ {\small $(1,-1,1,-1,(132))$}
\end{center}
\end{tabular}
\caption[Whitten group elements effected through rotation about the $z$-axis]{Whitten group elements effected through rotation about the $z$-axis. Group elements are listed under their respective figures.} \label{fig:rotate-z}
\end{figure}
\end{center}

\newpage
\section{Isotopy Figures for two-component links} \label{app:2isotopy}

This appendix contains all of the isotopy figures for two-component links that are not contained in Appendix~\ref{app:rotate}.  Section~\ref{app:pe} contains the pure exchange isotopies.  Section~\ref{app:pi} contains the pure exchange isotopies.  

\vspace*{0.5in}

\begin{figure}[ht]
\begin{center}
\begin{overpic}{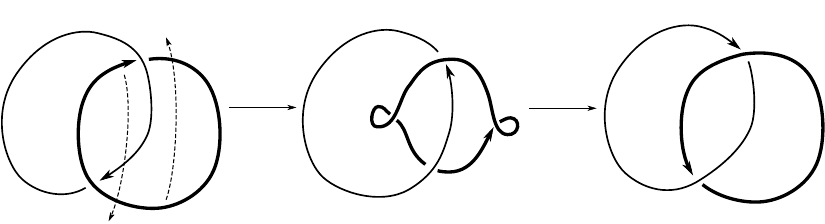}
\put(5,23){1}
\put(22,21){2}
\put(40,22){1}
\put(57,20){2}
\put(80,25){1}
\put(95,21){2}
\end{overpic}
\caption{$(2^2_1)^\gamma,~~\gamma=(-1,1,-1,e)$}
\label{221mi}
\end{center}
\end{figure}

\vspace*{1in}
\begin{figure}[ht]
\begin{center}
\begin{overpic}{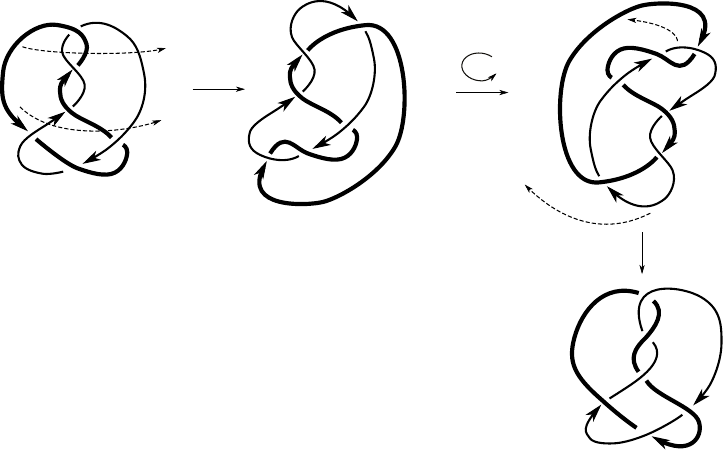}
\put(15,60){1}
\put(5,60){2}
\put(33,46){1}
\put(54,58){2}
\put(100,50){1}
\put(79,58){2}
\put(95,23){1}
\put(85,23){2}
\end{overpic}
\caption{$(6^2_2)^\gamma,~~\gamma=(-1,1,-1,e)$}
\label{622mi}
\end{center}
\end{figure}

\newpage
\begin{figure}[ht]
\begin{center}
\begin{overpic}{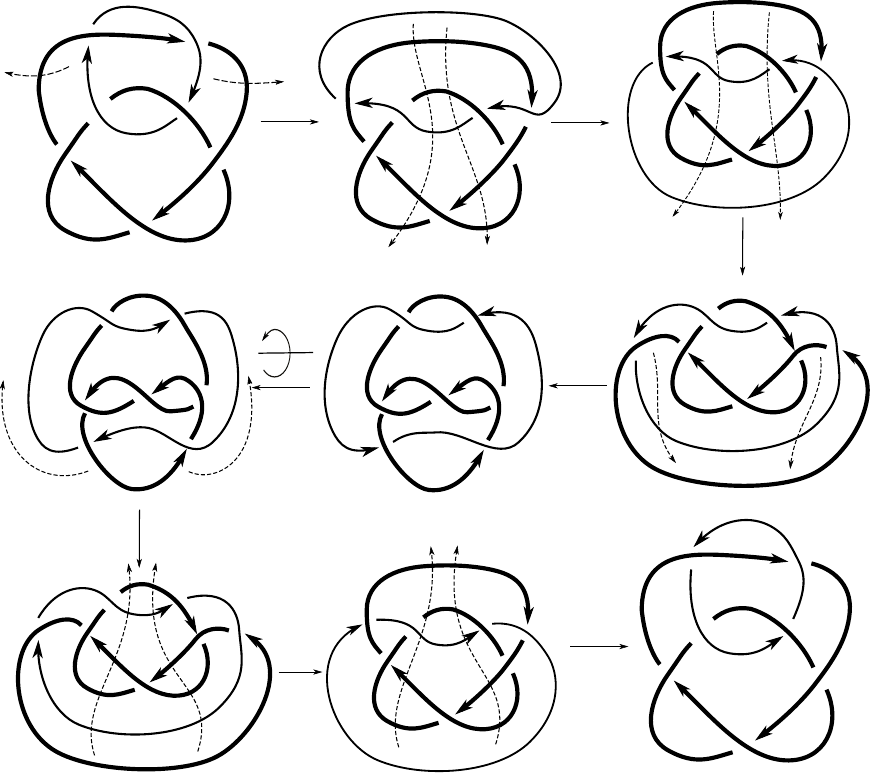}
\put(15,89){1}
\put(27,84){2}
\put(50,89){1}
\put(61,65){2}
\put(71,80){1}
\put(85,90){2}
\put(76,55){1}
\put(87,55){2}
\put(42,55){1}
\put(50,56){2}
\put(7,55){1}
\put(15,56){2}
\put(22,21){1}
\put(31,14){2}
\put(39,18){1}
\put(45,24){2}
\put(90,28){1}
\put(97,23){2}
\end{overpic}
\caption{$(7^2_4)^\gamma,~~\gamma=(1,-1,1,e)$}
\label{724invert1}
\end{center}
\end{figure}

\newpage
\begin{figure}[ht]
\begin{center}
\begin{overpic}{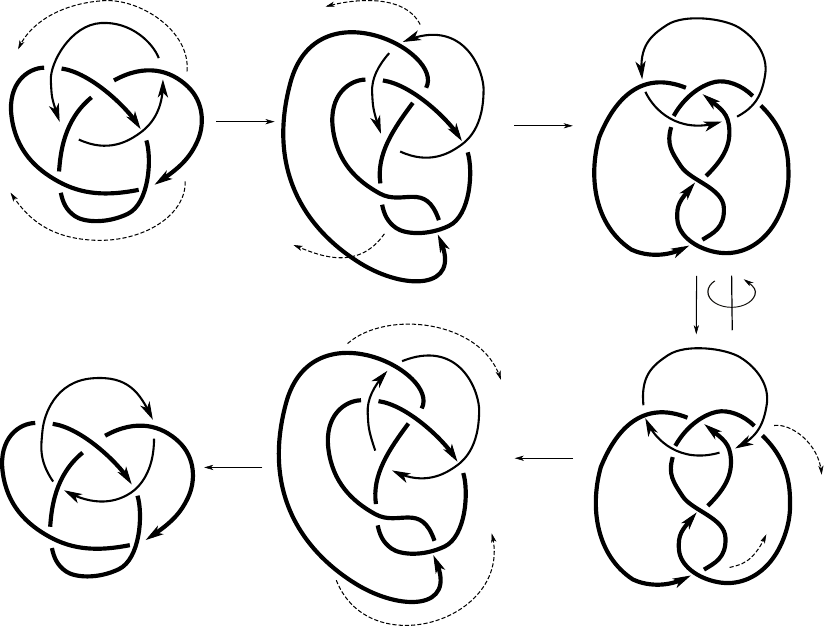}
\put(7,72){1}
\put(24,66){2}
\put(55,72){1}
\put(57,47){2}
\put(85,75){1}
\put(97,55){2}
\put(92,32){1}
\put(97,10){2}
\put(59,25){1}
\put(58,14){2}
\put(10,31){1}
\put(20,9){2}
\end{overpic}
\caption{$(7^2_6)^\gamma,~~\gamma=(1,-1,1,e)$}
\label{726invert1}
\end{center}
\end{figure}

\newpage
\begin{figure}[ht]
\begin{center}
\begin{overpic}{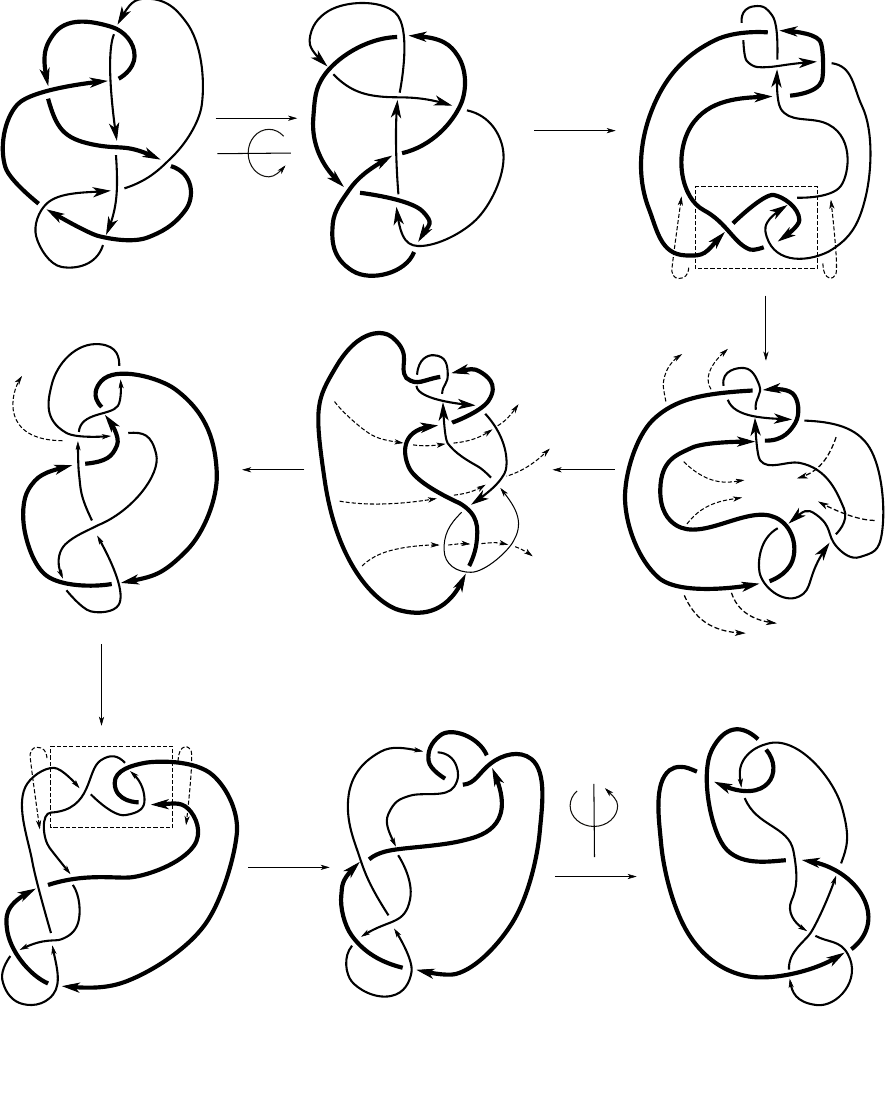}
\put(18,98){1}
\put(3,98){2}
\put(36.5,99){1}
\put(26,91){2}
\put(57,92){2}
\put(79,92){1}
\put(3.2,62){1}
\put(19,62){2}
\put(47,60){1}
\put(26,60){2}
\put(56,60){2}
\put(75,63){1}
\put(22,24){2}
\put(1,24){1}
\put(46,32.5){2}
\put(33,33){1}
\put(74,32){1}
\put(60,31.5){2}
\end{overpic}
\caption{$(8^2_8)^\gamma,~~\gamma=(-1,1,-1,e)$}
\label{828a}
\end{center}
\end{figure}

\newpage
\begin{figure}[ht]
\begin{center}
\begin{overpic}{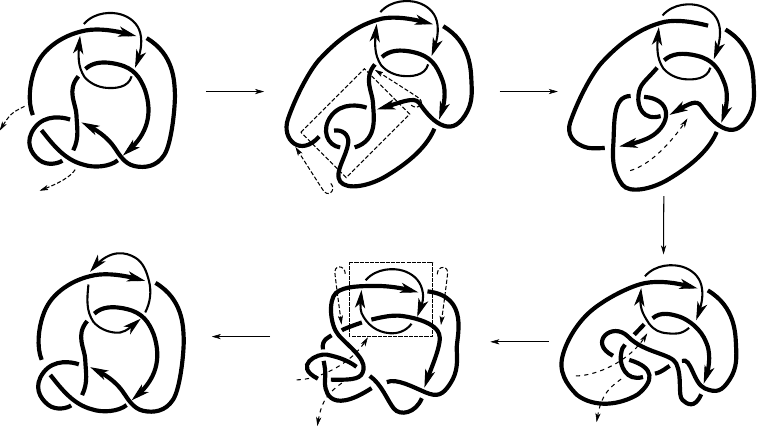}
\put(15,56){1}
\put(24,47){2}
\put(53,57){1}
\put(62,50){2}
\put(90,57){1}
\put(101,45){2}
\put(90,22){1}
\put(97,16){2}
\put(51,23){1}
\put(61,16){2}
\put(15,24){1}
\put(24,17){2}
\end{overpic}
\caption{$(8^2_{10})^\gamma,~~\gamma=(1,-1,1,e)$}
\label{8210invert1}
\end{center}
\end{figure}

\newpage
\begin{figure}[ht]
\begin{center}
\begin{overpic}{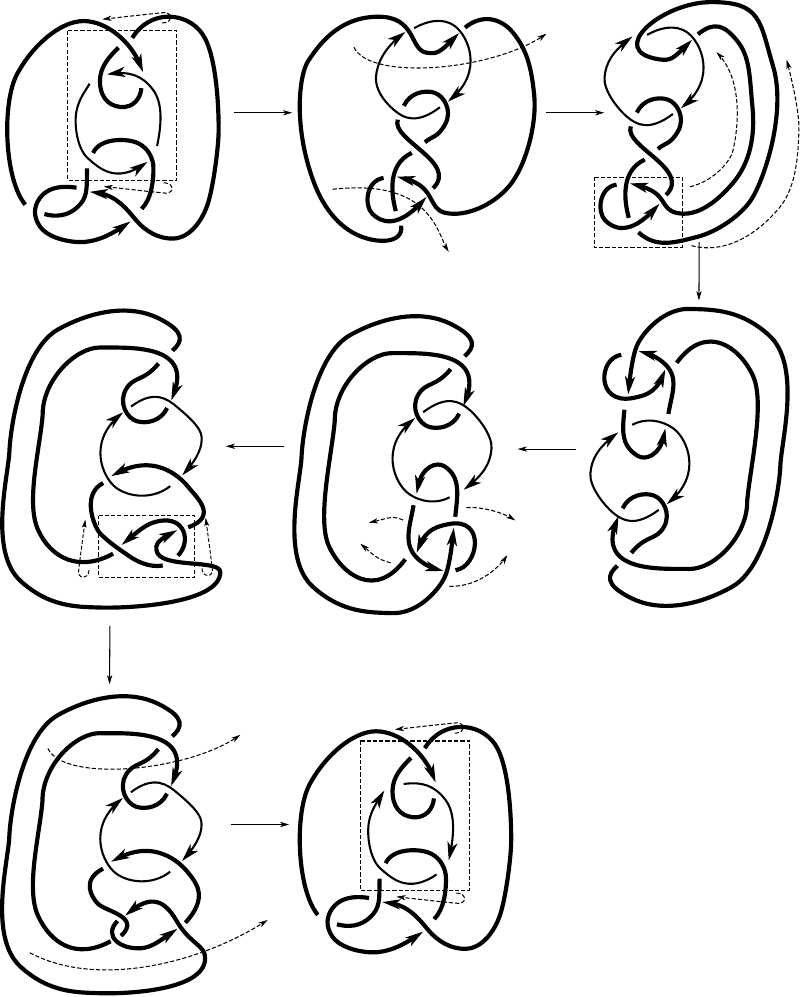}
\put(5,87){1}
\put(20,99){2}
\put(42,98){1}
\put(50,99){2}
\put(59,93){1}
\put(70,101){2}
\put(58,53){1}
\put(58,62){2}
\put(50,55){1}
\put(48,65){2}
\put(21,57){1}
\put(20,67){2}
\put(21,19){1}
\put(19,28){2}
\put(34,15){1}
\put(34,27){2}
\end{overpic}
\caption{$(8^2_{12})^\gamma,~~\gamma=(1,-1,1,e)$}
\label{8212invert1}
\end{center}
\end{figure}

\newpage
\begin{figure}[ht]
\begin{center}
\begin{overpic}{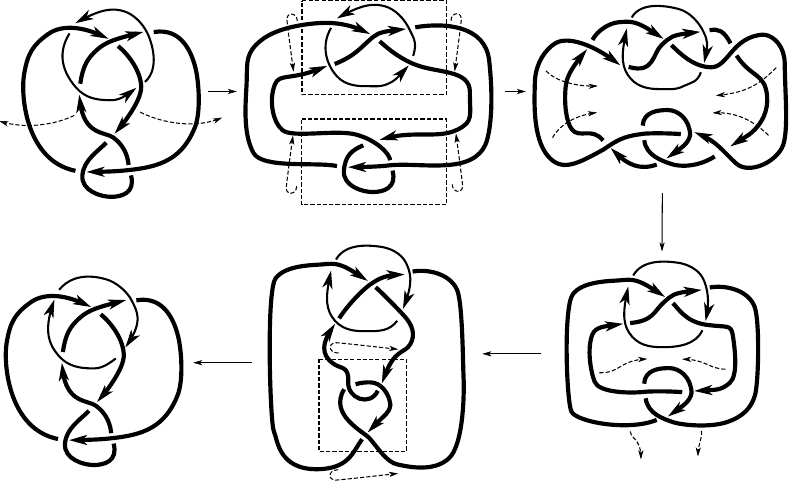}
\put(15,61){1}
\put(22,58){2}
\put(47,62){1}
\put(62,57){2}
\put(84,61){1}
\put(95,58){2}
\put(87,29){1}
\put(95,25){2}
\put(47,31){1}
\put(55,28){2}
\put(11,27){1}
\put(21,23){2}
\end{overpic}
\caption{$(8^2_{13})^\gamma,~~\gamma=(1,-1,1,e)$}
\label{8213invert1}
\end{center}
\end{figure}

\vspace*{0.75in}
\begin{figure}[ht]
\begin{center}
\begin{overpic}{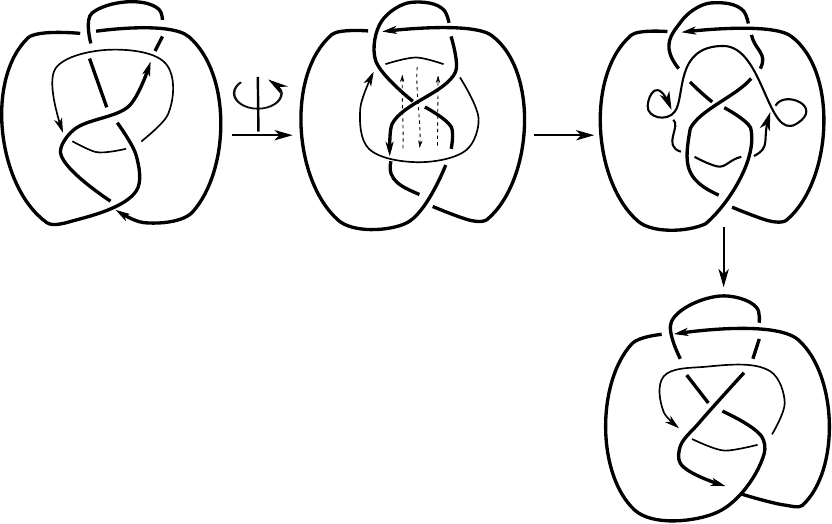}
\put(4,52){1}
\put(24,58){2}
\put(41.5,52){1}
\put(62,58){2}
\put(76,52){1}
\put(98,58){2}
\put(95,24){2}
\put(77,15){1}
\end{overpic}
\caption{$(8^2_{15})^\gamma,~~\gamma=(1,1,-1,e)$}
\label{8n2a}
\end{center}
\end{figure}
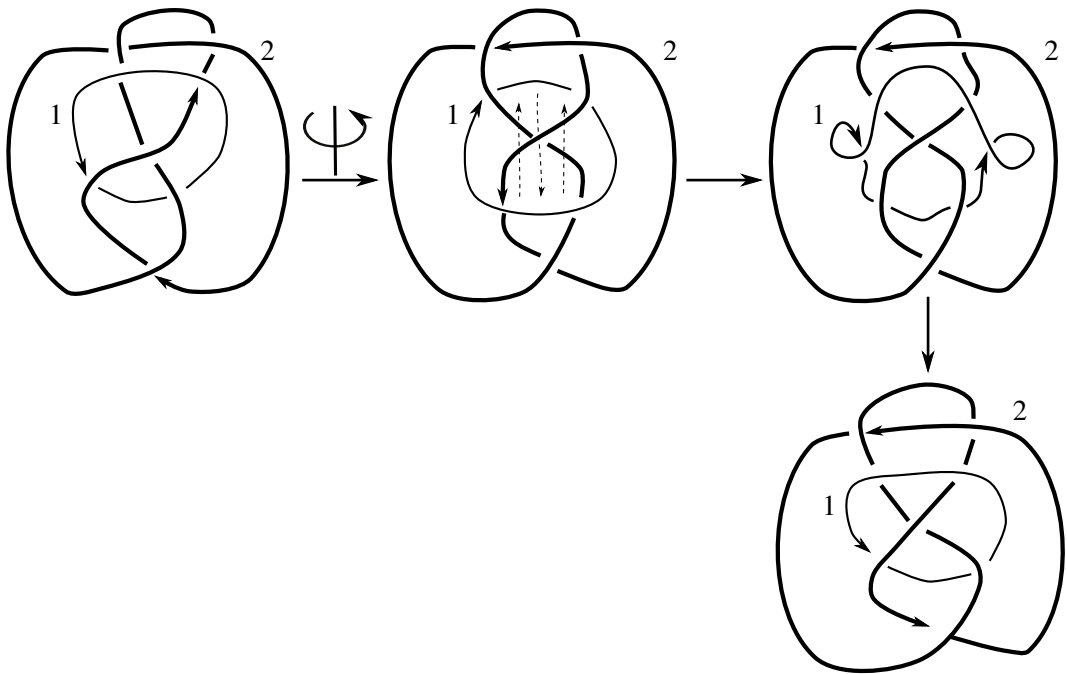

\newpage
\subsection{Isotopies showing pure exchange symmetries for two-component links}\label{app:pe}

Along with Appendix~\ref{app:rotate-pe}, Figures~\ref{fig:721pe}-\ref{fig:twisttype} demonstrate that each of the 17 two-component links listed in Lemma~\ref{lemma:pe} has pure exchange symmetry.  We demonstrate in section~\ref{sec2a} that the 13 remaining two-component links with $\leq 8$ crossings do not admit this symmetry.

\vspace*{0.5in}



\begin{figure}[ht]
\begin{center}
\begin{overpic}[scale=0.9]%
{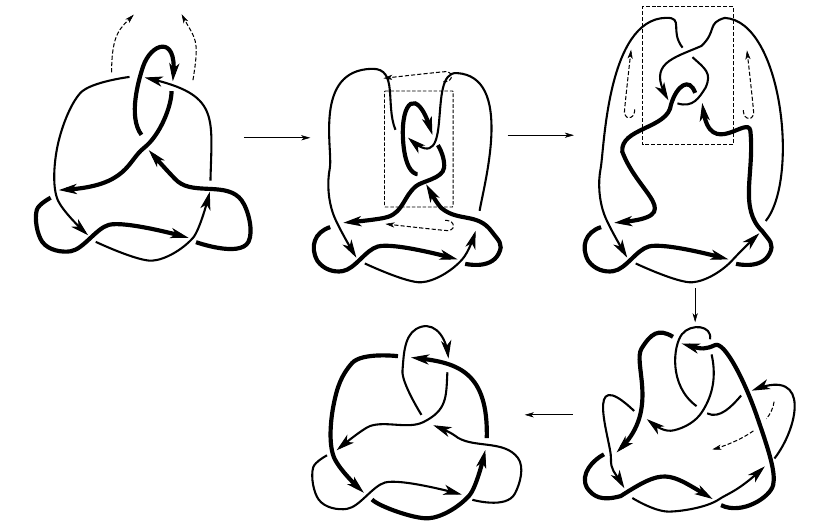}
\put(6,50){1}
\put(2,33){2}
\put(38,50){1}
\put(35,33){2}
\put(72,50){1}
\put(68,33){2}
\put(72,16){1}
\put(68,4){2}
\put(35,4){1}
\put(38,17){2}

\end{overpic}
\caption{$(7^2_1)^\gamma,~~\gamma=(1,1,1,(12))$}
\label{fig:721pe}
\end{center}
\end{figure}

\newpage
\begin{figure}[ht]
\begin{center}
\begin{overpic}[scale=0.9]%
{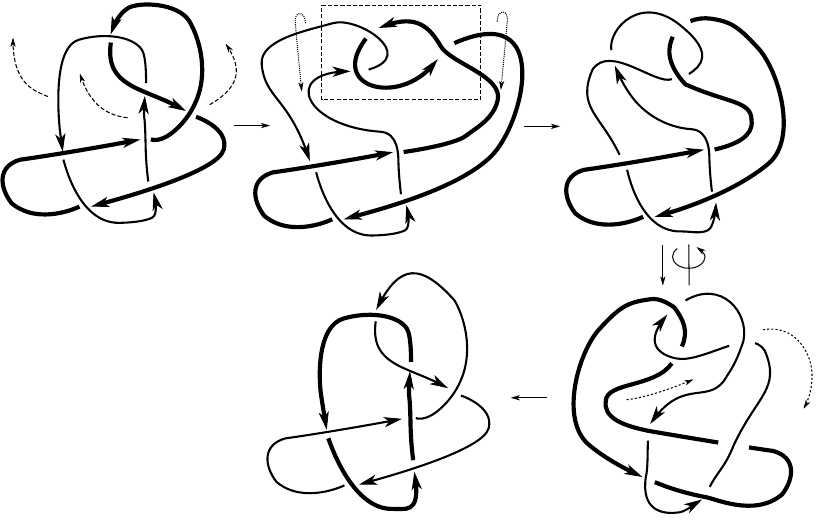}
\put(5,55){1}
\put(24,60){2}
\put(29.5,55){1}
\put(62,60){2}
\put(69,55){1}
\put(91,60){2}
\put(91,24){1}
\put(69,22){2}
\put(57,24){1}
\put(37,22){2}
\end{overpic}
\caption{$(7^2_2)^\gamma,~~\gamma=(1,1,1,(12))$}
\end{center}
\end{figure}


\vspace*{1in}
\begin{figure}[ht]
\begin{center}
\begin{overpic}[scale=0.9]%
{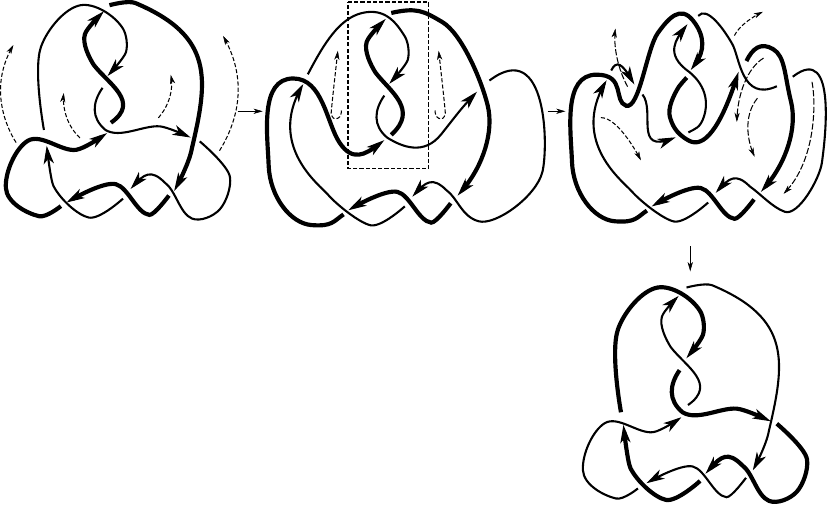}
\put(4,58){1}
\put(-1.5,37){2}
\put(31,37){2}
\put(38,58){1}
\put(67,37){2}
\put(87.5,58){1}
\put(68,3){1}
\put(73,22){2}
\end{overpic}
\caption{$(8^2_2)^\gamma,~~\gamma=(1,1,1,(12))$}
\end{center}
\end{figure}

\newpage
\vspace{1cm}
\begin{figure}[ht]
\begin{center}
\begin{overpic}[scale=0.9]%
{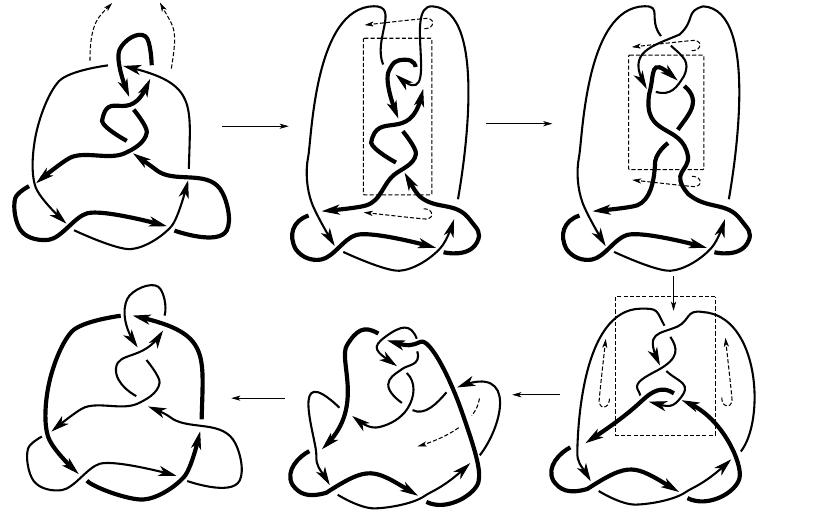}
\put(4,52){1}
\put(-0.5,35){2}

\put(36.5,52){1}
\put(33,35){2}

\put(70,52){1}
\put(67,35){2}

\put(71,22){1}
\put(65,4){2}

\put(37.5,17){1}
\put(33,4){2}
\put(3.5,17){2}
\put(1.5,4){1}
\end{overpic}
\caption{$(8^2_3)^\gamma,~~\gamma=(1,1,1,(12))$}
\end{center}
\end{figure}


\vspace*{0.5in}
\begin{figure}[ht]
\begin{center}
\begin{overpic}[scale=0.9]%
{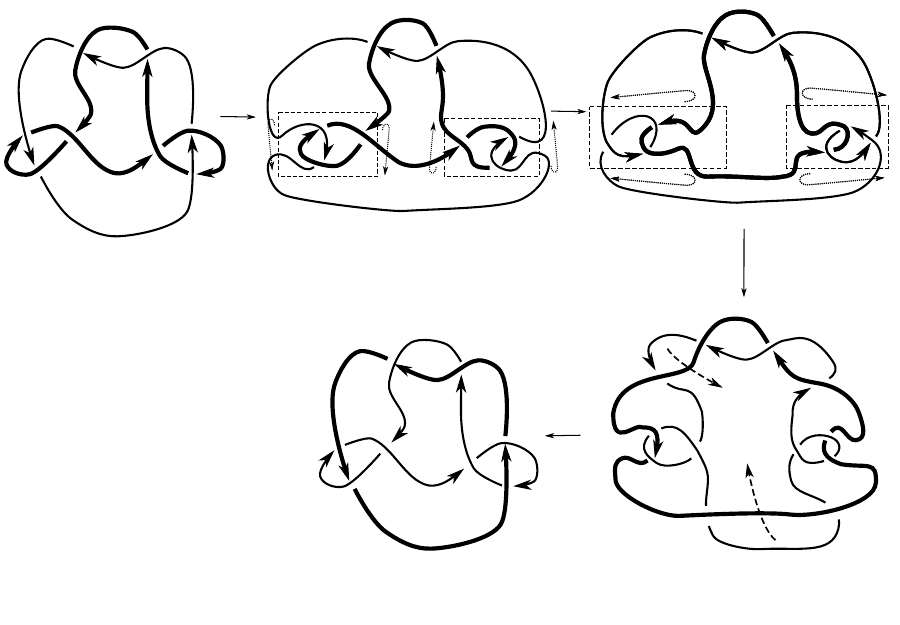}
\put(0,60){1}
\put(12,68){2}
\put(28,60){1}
\put(42,68){2}
\put(85,50){1}
\put(93,49){2}
\put(65,60){1}
\put(85,68){2}
\put(71,32){1}
\put(85,33){2}
\put(47,33){1}
\put(35.5,28){2}
\end{overpic}
\caption{$(8^2_7)^\gamma,~~\gamma=(1,1,1,(12))$}
\end{center}
\end{figure}

\newpage
\vspace{1cm}
\begin{figure}[ht]  
\begin{center}
\begin{overpic}[scale=0.9]%
{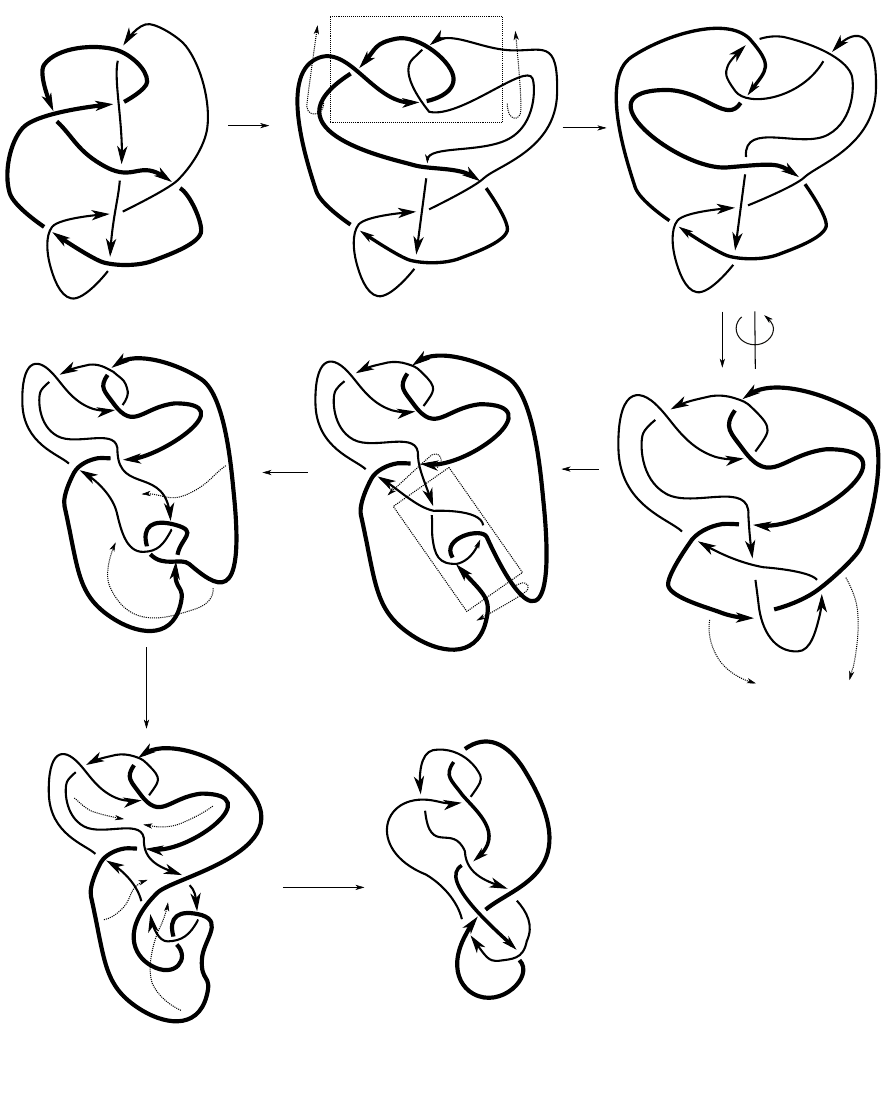}
\put(2,92){2}
\put(16,97){1}
\put(25,92){2}
\put(48,96){1}
\put(54,92){2}
\put(72,97){1}
\put(0.5,63){1}
\put(17,67){2}
\put(27,63){1}
\put(44,67.5){2}
\put(55,63){1}
\put(75,64.5){2}
\put(3,27){1}
\put(17,33){2}
\put(33,24){1}
\put(46,33){2}
\end{overpic}
\caption{$(8^2_8)^\gamma,~~\gamma=(1,1,1,(12))$}
\label{fig:828pe}
\end{center}
\end{figure}

\newpage
\vspace{1cm}
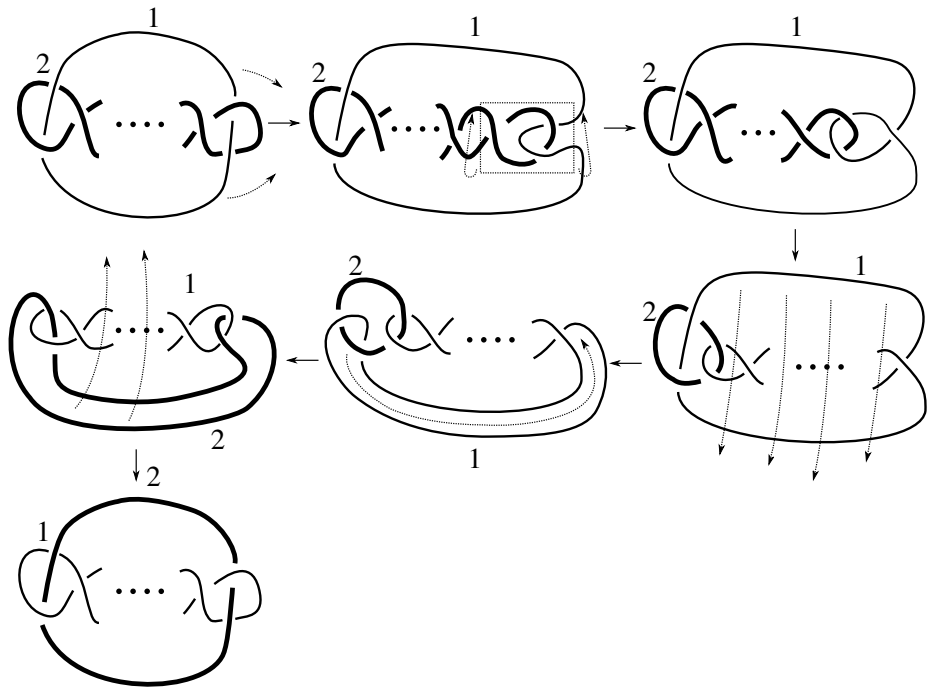
\begin{figure}[ht]  
\begin{center}
\begin{overpic}[scale=0.9]%
{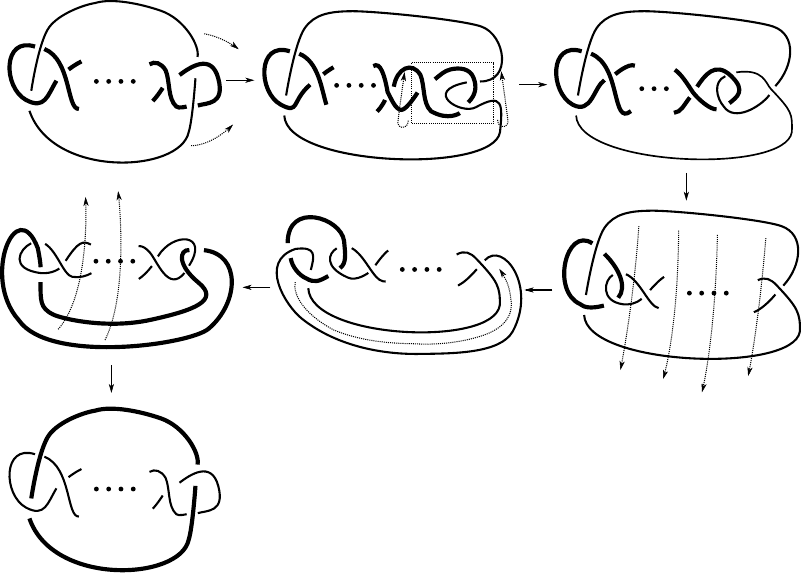}
\put(15,72){1}
\put(3,67){2}
\put(50,71){1}
\put(33,66){2}
\put(85,71){1}
\put(69,66){2}
\put(19,43){1}
\put(22,26){2}
\put(50,24){1}
\put(37,45){2}
\put(92,45){1}
\put(69,40){2}
\put(3,16){1}
\put(15,22){2}
\end{overpic}
\caption[Pure exchange isotopies for $5^2_1$, $6^2_3$, $7^2_3$ and $8^2_6$.]{The links $5^2_1$, $6^2_3$, $7^2_3$ and $8^2_6$ share the common form
above, with 1, 2, 3, and 4 crossings replacing the dots in the central ``twisted'' region of component 2 above. The pure exchange symmetry for each can be accomplished in a similar way, by shifting twists from
component 2 to component 1 as shown. Since the orientations vary between links, we do not show arrows above. Some examples require a final flip or twist to match
orientations, but in each case it is not difficult to figure out the
required move.}
\label{fig:twisttype}
\end{center}
\end{figure}

\newpage

\subsection{Isotopies showing pure invertibility for two-component links}  \label{app:pi}

Along with the 21 isotopies exhibited in Appendix~\ref{app:rotate-pi},  Figures~\ref{fig:622pi}-\ref{fig:8215pi} demonstrate that, as stated in Lemma~\ref{lemma:pi}, all 30 of the two-component links with eight or fewer crossings are purely invertible.  (To obtain invertibility for $7^2_8$, combine the results of Figures~\ref{fig:rotate-y} and \ref{fig:rotate-z}, which show that each of its components can be individually inverted.)

\vspace*{0.5in}

\begin{figure}[ht]
\begin{center}
\begin{overpic}[scale=0.9]%
{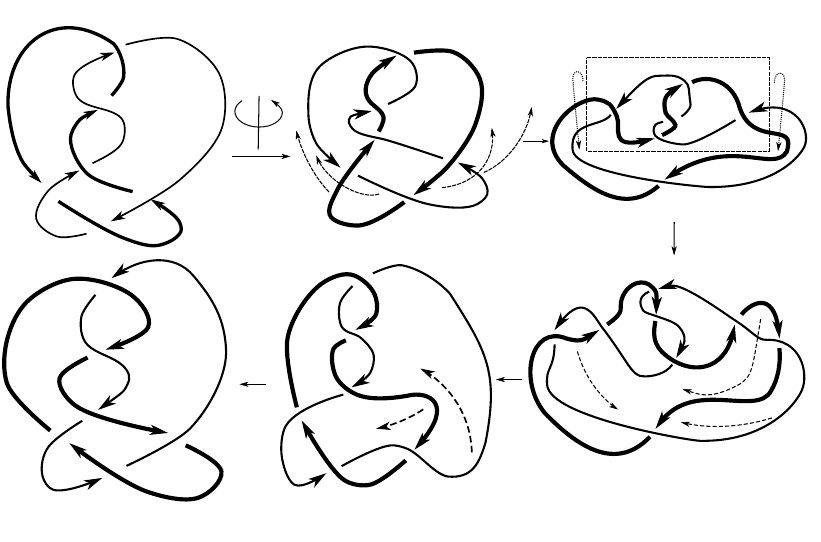}
\put(2,39){1}
\put(1,61){2}
\put(37,58){1}
\put(37,39){2}
\put(76,56){1}
\put(90,55){2}
\put(85,31){1}
\put(74,31){2}
\put(35,31){2}
\put(32,7){1}
\put(3,6){1}
\put(1,30){2}
\end{overpic}
\caption{$(6^2_2)^\gamma,~~\gamma=(1,-1,-1,e)$}
\label{fig:622pi}
\end{center}
\end{figure}

\newpage
\vspace{1cm}
\begin{figure}[ht]
\begin{center}
\begin{overpic}[scale=0.9]%
{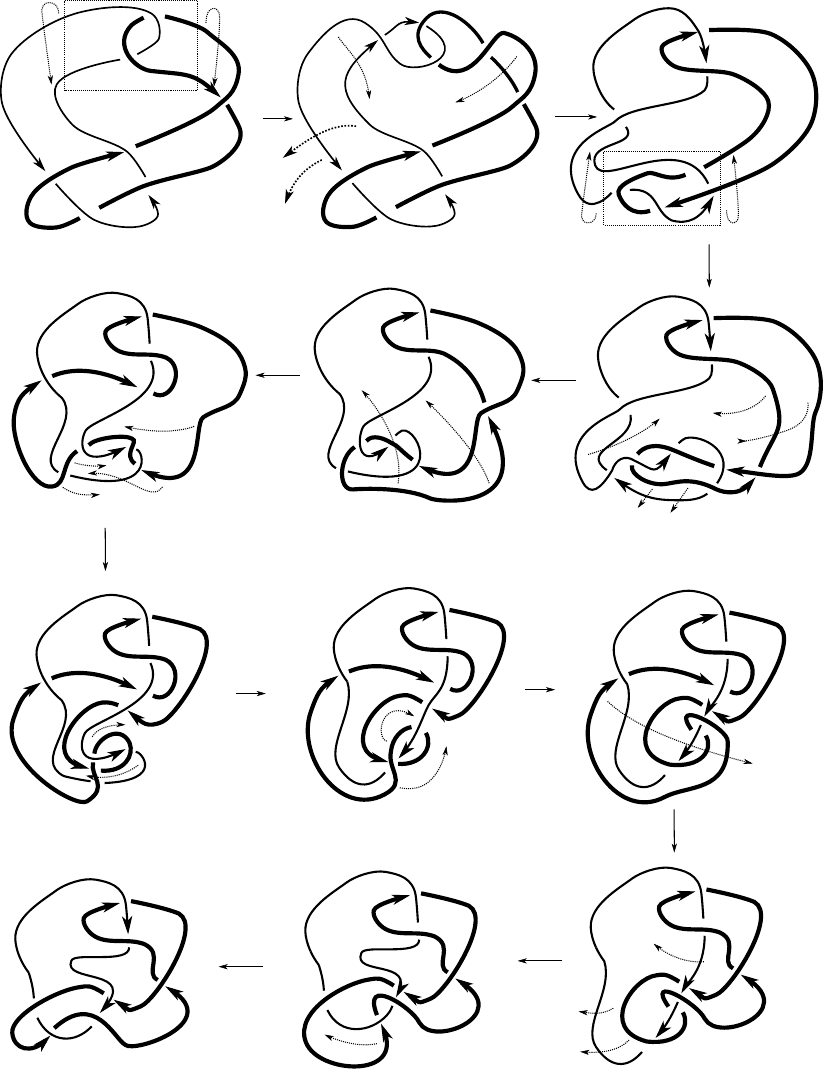}
\put(1,96.5){1}
\put(1,80){2}
\put(28,96.5){1}
\put(28,80){2}
\put(53,80){1}
\put(73.5,97){2}
\put(1,54){2}
\put(4,71){1}
\put(30,71){1}
\put(42,71){2}
\put(57,71){1}
\put(70,71){2}
\put(1,27){2}
\put(4,42.5){1}
\put(31,42.5){1}
\put(28,27){2}
\put(57.5,42.5){1}
\put(54,27){2}
\put(4,17){1}
\put(1,0){2}
\put(30,17){1}
\put(27.5,0){2}
\put(56,17){1}
\put(54,0){2}
\end{overpic}
\caption{$(7^2_2)^\gamma,~~\gamma=(1,-1,-1,e)$}
\end{center}
\end{figure}

\newpage
\begin{figure}[ht]
\begin{center}
\begin{overpic}[scale=0.9]%
{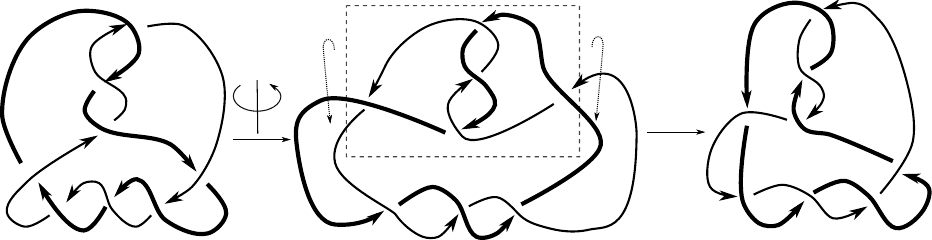}
\put(23,24){1}
\put(25,1){2}
\put(68,2){1}
\put(60,22){2}
\put(96,22){1}
\put(100,2){2}
\end{overpic}
\caption{$(8^2_2)^\gamma,~~\gamma=(1,-1,-1,e)$}
\end{center}
\end{figure}

\vspace*{0.75in}

\begin{figure}[ht]
\begin{center}
\begin{overpic}[scale=0.9]%
{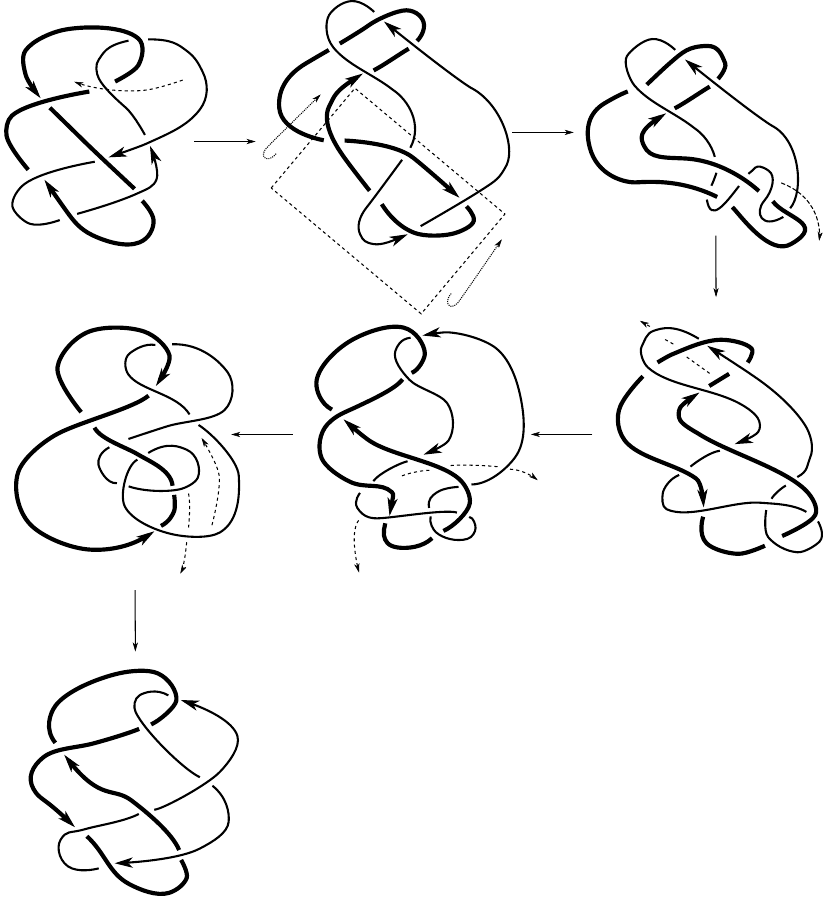}
\put(5,97){2}
\put(22,96){1}
\put(33,95){2}
\put(52,92){1}
\put(64,88){2}
\put(85,89){1}
\put(7,64){2}
\put(22,63){1}
\put(36,64){2}
\put(58,60){1}
\put(66,54){2}
\put(88,57){1}
\put(6,24){2}
\put(25,21){1}
\end{overpic}
\caption{$(8^2_5)^\gamma,~~\gamma=(1,-1,-1,e)$}
\end{center}
\end{figure}

\newpage
\begin{figure}[ht]
\begin{center}
\begin{overpic}[scale=0.9]%
{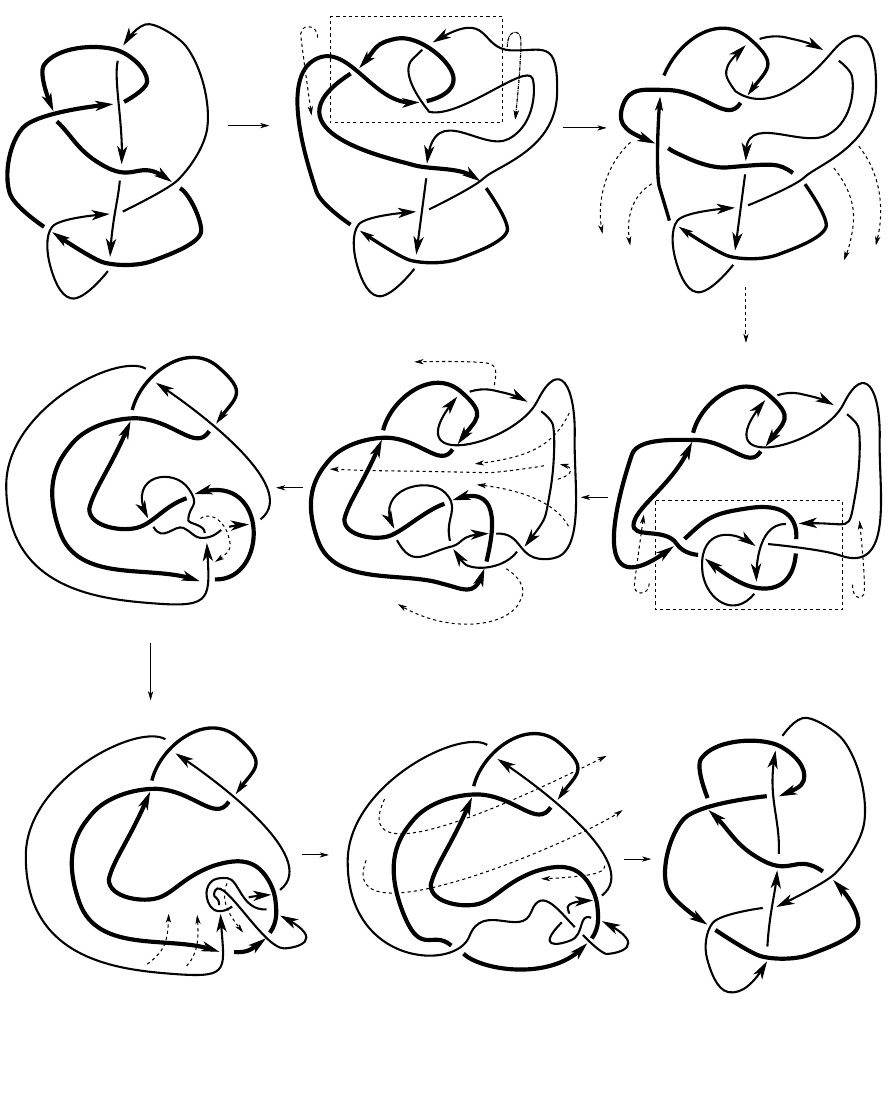}
\put(18,95){1}
\put(5,96){2}
\put(51,95){1}
\put(25,95){2}
\put(77,98){1}
\put(60,97){2}
\put(75,65){1}
\put(65,66){2}
\put(50,67){1}
\put(35,65){2}
\put(5,66){1}
\put(22,65){2}
\put(5,31){1}
\put(20,35){2}
\put(35,31){1}
\put(50,35){2}
\put(75,35){1}
\put(65,34){2}
\end{overpic}
\caption{$(8^2_8)^\gamma,~~\gamma=(1,-1,-1,e)$}
\end{center}
\end{figure}

\newpage
\begin{figure}[ht]
\begin{center}
\begin{overpic}[scale=0.9]%
{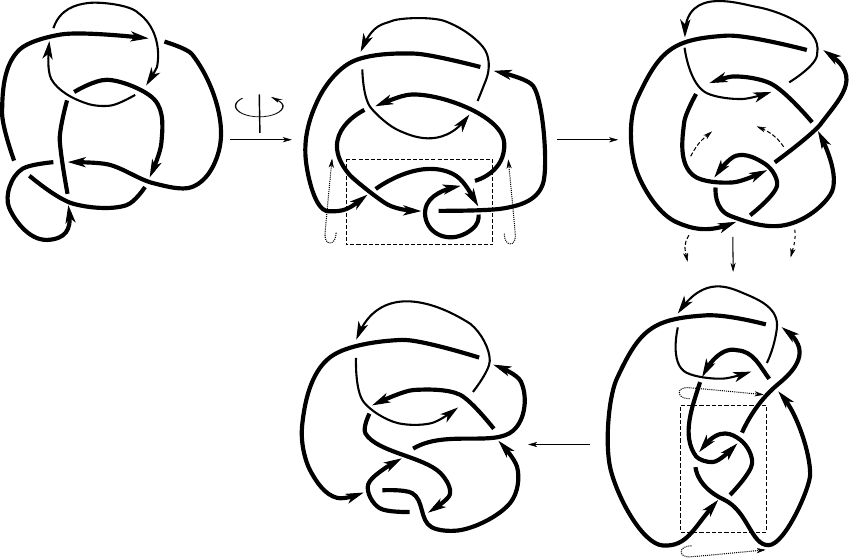}
\put(19,62){1}
\put(-2.5,55){2}
\put(56,62){1}
\put(35.5,55){2}
\put(96,62){1}
\put(73,55){2}

\put(50,31){1}
\put(35,22){2}
\put(87,32){1}
\put(70,22){2}
\end{overpic}
\caption{$(8^2_{10})^\gamma,~~\gamma=(1,-1,-1,e)$}
\end{center}
\end{figure}

\vspace*{1in}
\begin{figure}[ht]
\begin{center}
\begin{overpic}[scale=0.9]%
{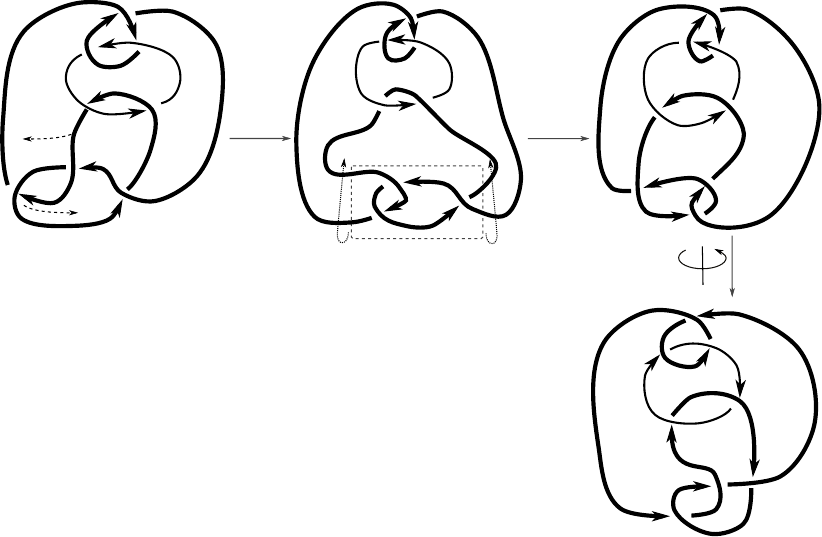}
\put(6,55){1}
\put(-1,60){2}
\put(41.5,55){1}
\put(37,60){2}
\put(76.5,55){1}
\put(72.5,60){2}
\put(76.5,18){1}
\put(72.5,26){2}
\end{overpic}
\caption{$(8^2_{12})^\gamma,~~\gamma=(1,-1,-1,e)$}
\end{center}
\end{figure}

\newpage
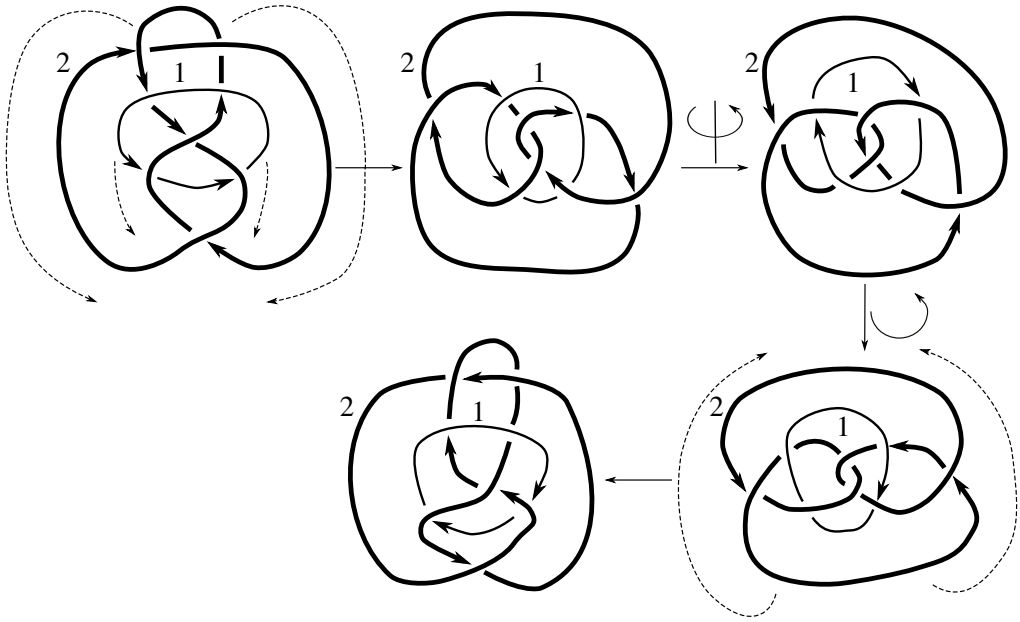
\begin{figure}[ht]
\begin{center}
\begin{overpic}[scale=0.9]%
{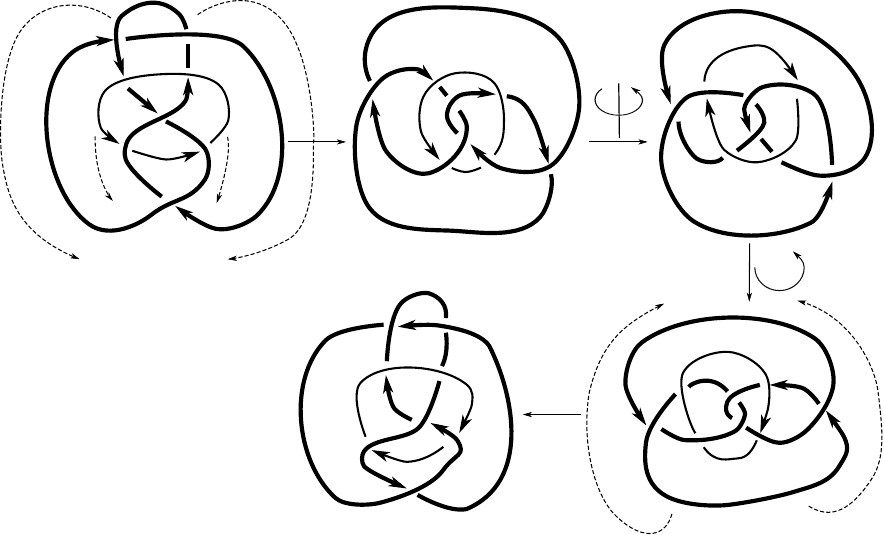}
\put(16.5,53){1}
\put(5,54){2}
\put(52,53){1}
\put(39,54){2}
\put(83,52){1}
\put(73,54){2}
\put(69.5,20){2}
\put(46,19.5){1}
\put(82,18){1}
\put(33,20){2}
\end{overpic}
\caption{$(8^2_{15})^\gamma,~~\gamma=(1,-1,-1,e)$}\label{fig:8215pi}
\end{center}
\end{figure}

\newpage
\section{Isotopy Figures for three-component links} \label{app:3isotopy}

This appendix contains all of the isotopy figures for three-component links that are not contained in Appendix~\ref{app:rotate}.

\vspace{0.125in}
\begin{figure}[ht]
\begin{center}
\begin{overpic}[scale=0.9]%
{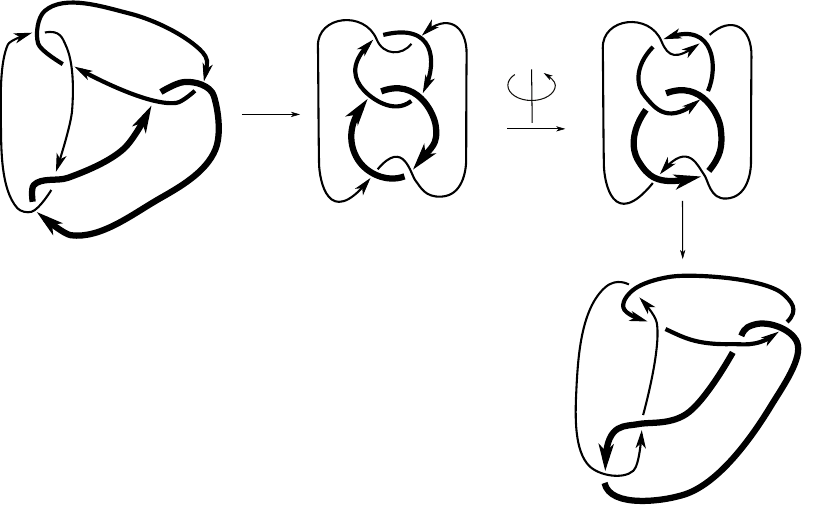}
\put(-2,53){1}
\put(15,63){2}
\put(18,36){3}
\put(36.5,53){1}
\put(50,55){2}
\put(50,46.5){3}
\put(71,53){1}
\put(84,54){2}
\put(85,46){3}
\put(68.5,22){1}
\put(85,31){2}
\put(88,4){3}
\end{overpic}
\caption{$(6^3_1)^\gamma,~~\gamma=(1,-1,-1,-1,e)$}
\label{631invert}
\end{center}
\end{figure}

\vspace{0.5in}
\begin{figure}[ht]
\begin{center}
\begin{overpic}[scale=0.9]{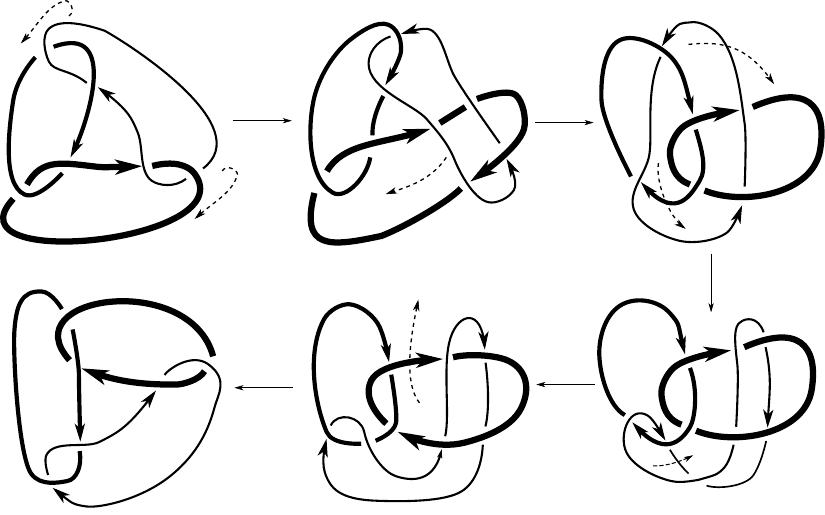}
\put(17,57){1}
\put(0,52){2}
\put(15,31){3}
\put(54,57){1}
\put(36,52){2}
\put(48,31){3}
\put(85,61){1}
\put(70.5,52){2}
\put(97,51.5){3}
\put(89,24.5){1}
\put(70.5,21){2}
\put(97,22){3}
\put(55,24.5){1}
\put(36,21){2}
\put(62,19.5){3}
\put(22,25){3}
\put(18.5,1){1}
\put(-1,15){2}
\end{overpic}
\caption{$(6^3_1)^\gamma,~~\gamma=(1,1,1,1,(23))$}
\label{631(23)}
\end{center}
\end{figure}

\newpage
\begin{figure}[ht]
\begin{center}
\begin{overpic}[scale=0.9]{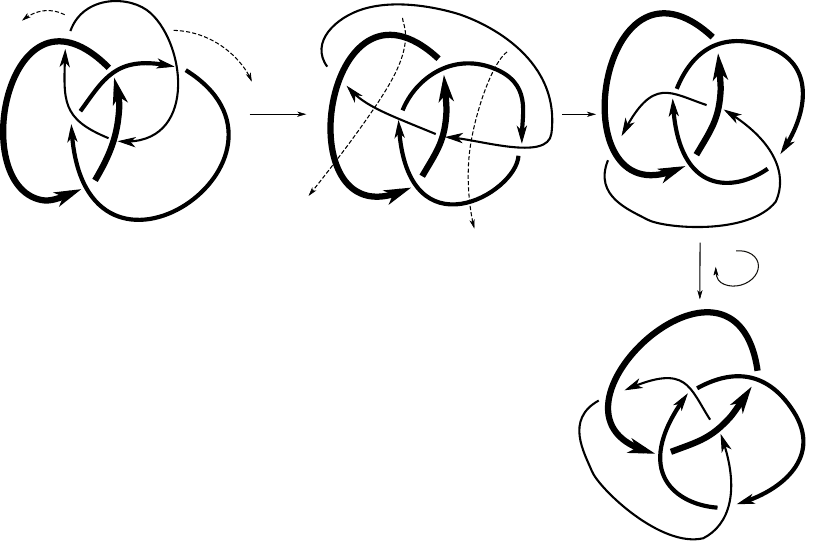}
\put(19,64.5){1}
\put(-1.5,54){3}
\put(24.5,40){2}

\put(62,62){1}
\put(61,40){2}
\put(38,54){3}

\put(95,40){1}
\put(93,62){2}
\put(71.5,54){3}

\put(72,4){1}
\put(97,17){2}
\put(79.5,27){3}
\end{overpic}
\caption{$(6^3_2)^\gamma,~~\gamma=(-1,-1,1,1,(13))$}\label{632_1}
\end{center}
\end{figure}

\vspace{1in}
\begin{figure}[ht]
\begin{center}
\begin{overpic}[scale=0.9]%
{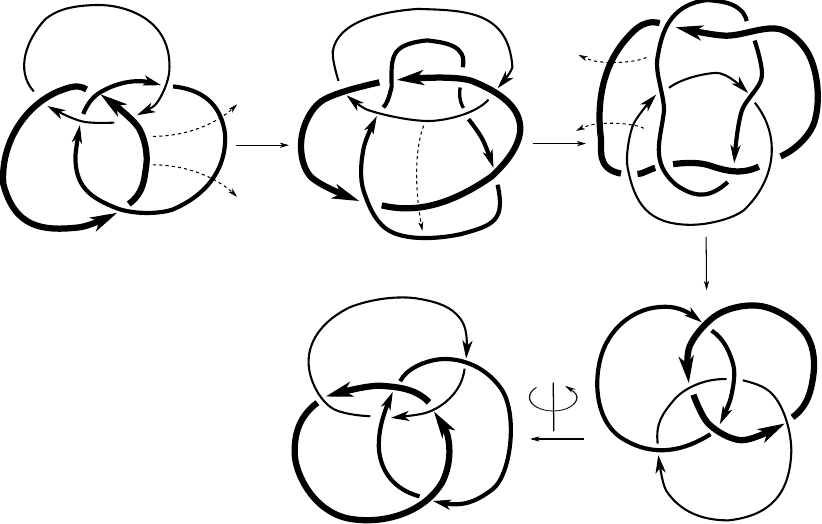}
\put(6,33){3}
\put(20,60){1}
\put(25.5,37.5){2}

\put(61,60){1}
\put(62,37.5){2}
\put(35,42){3}

\put(92,37){1}
\put(85,64.5){2}
\put(71,53){3}
\put(78.5,3){1}
\put(70.5,17){2}
\put(92,27.5){3}
\put(47,28){1}
\put(62,4){2}
\put(34,6){3}
\end{overpic}
\caption{$(6^3_2)^\gamma,~~\gamma=(-1,1,1,1,e)$}\label{632_5}
\end{center}
\end{figure}

\newpage
\begin{figure}[ht]
\begin{center}
\begin{overpic}[scale=.9]
{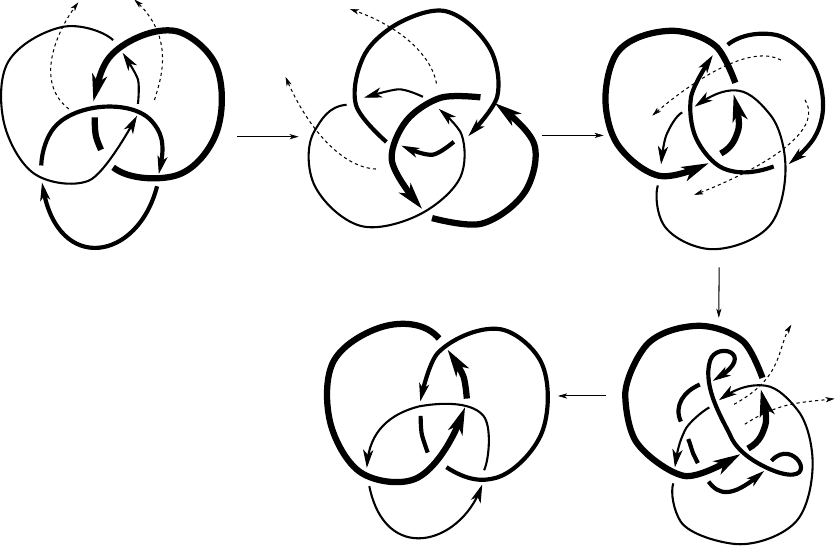}
\put(3,62){1}
\put(12,33){2}
\put(27,58){3}
\put(36,40){1}
\put(42.5,60){2}
\put(64,41){3}

\put(94,40){1}
\put(97,60){2}
\put(72,60){3}

\put(97,1.5){1}
\put(81,19.5){2}
\put(74.5,23){3}
\put(38,22){3}
\put(57,1.5){1}
\put(62,26){2}
\end{overpic}
\caption{$(6^3_3)^\gamma,~~\gamma=(1,1, -1,1, (132))$}\label{6n1_2}
\end{center}
\end{figure}

\vspace{1in}
\begin{figure}[ht]
\begin{center}
\begin{overpic}[scale=.9]
{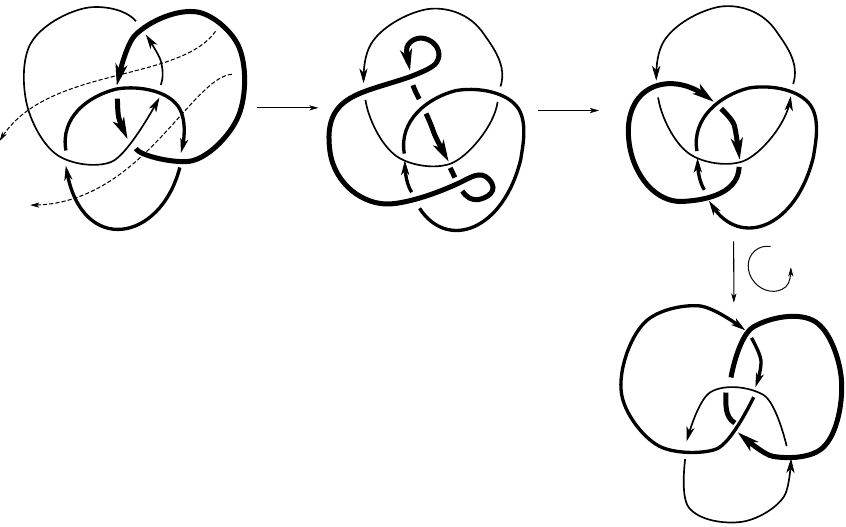}
\put(2,58){1}
\put(12,32){2}
\put(29,58){3}

\put(58,58){1}
\put(62.5,43){2}
\put(38,40){3}

\put(94,57){1}
\put(96,40){2}
\put(73,40){3}
\put(92,0){1}
\put(73,23){2}
\put(98,24){3}
\end{overpic}
\caption{$(6^3_3)^\gamma,~~\gamma=(1,-1, -1,-1, (12))$}\label{6n1_3}
\end{center}
\end{figure}

\newpage
\begin{figure}[ht]
\begin{center}
\begin{overpic}[scale=.9]
{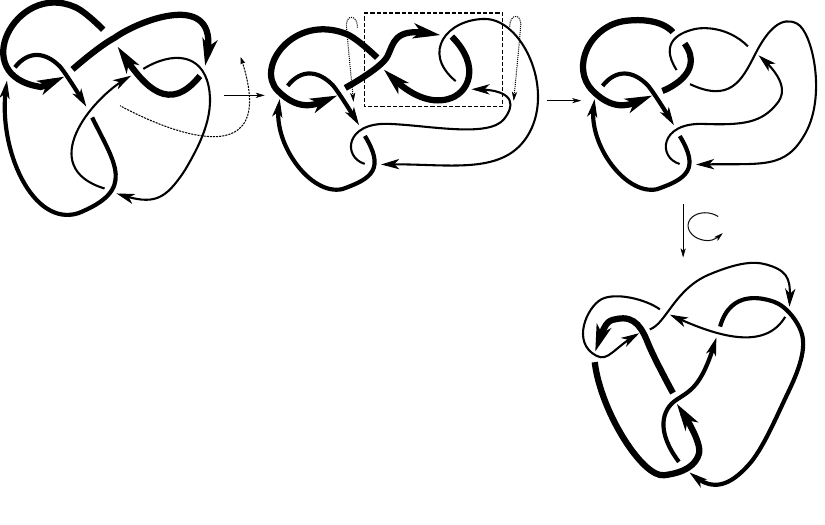}
\put(22,40){1}
\put(0,40){2}
\put(12,61.5){3}

\put(61.5,42){1}
\put(33,42){2}
\put(40,61){3}
\put(98,45){1}
\put(71,42){2}
\put(72,61){3}

\put(79,27.5){1}
\put(96,13){2}
\put(71,13){3}
\end{overpic}
\caption{$(7^3_1)^\gamma,~~\gamma=(1,1,-1,1,(123))$}\label{731+-+(132)}
\end{center}
\end{figure}

\vspace{1in}

\begin{figure}[ht]
\begin{center}
\begin{overpic}[scale=.9]
{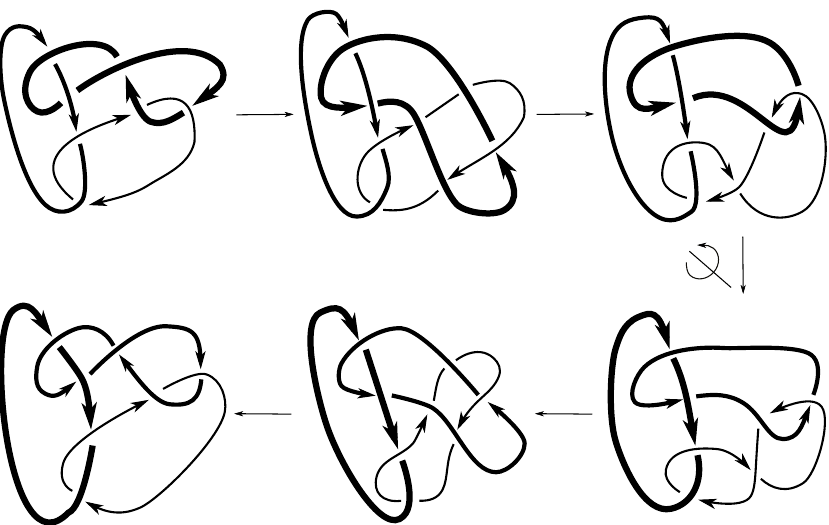}
\put(21,40){1}
\put(0,61){2}
\put(19.5,58){3}

\put(51,36){1}
\put(34.5,60){2}
\put(52.5,58){3}

\put(93,34.5){1}
\put(71,55){2}
\put(89,60){3}

\put(93,2){1}
\put(87,23){2}
\put(71.5,20){3}

\put(60.5,20){1}
\put(47,25){2}
\put(35,20){3}

\put(23.5,5){1}
\put(18,25){2}
\put(-2,20){3}
\end{overpic}
\caption{$(7^3_1)^\gamma,~~\gamma=(1,1,1,1,(23))$}\label{731+++(23)}
\end{center}
\end{figure}

\vspace{0.75in}
\begin{figure}[ht]
\begin{center}
\begin{overpic}[scale=.9]
{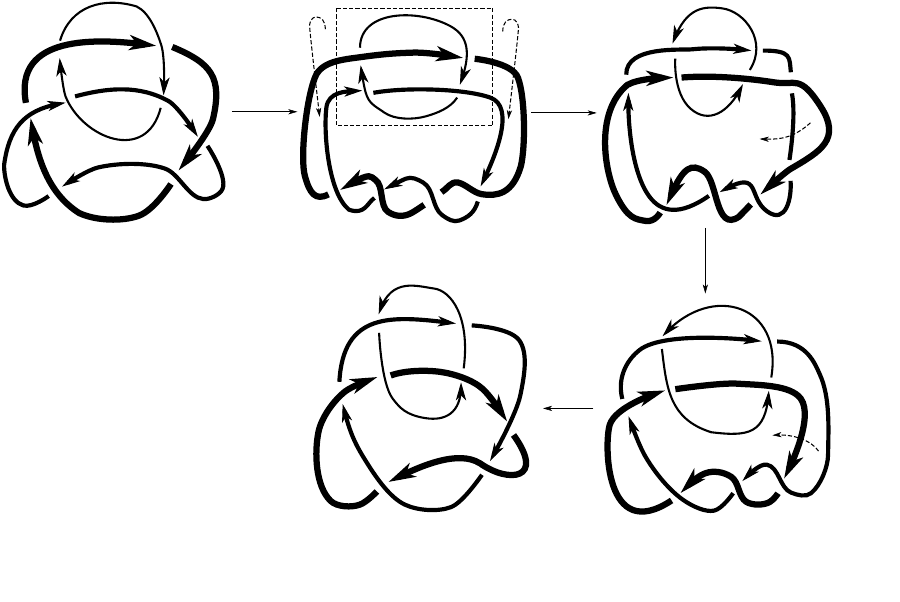}
\put(6,65){1}
\put(26,46){2}
\put(22,61){3}
\put(51,63.5){1}
\put(51,40){2}
\put(32,56){3}
\put(73,64){1}
\put(88.5,60){2}
\put(66,45){3}
\put(84,32){1}
\put(67,25){2}
\put(66,10){3}
\put(46,36){1}
\put(59,27){2}
\put(33,13){3}
\end{overpic}
\caption{$(8^3_1)^\gamma,~~\gamma=(1,-1,1,1,(23))$}\label{831(23)}
\end{center}
\end{figure}

\newpage
\begin{figure}[ht]
\begin{center}
\begin{overpic}[scale=.9]
{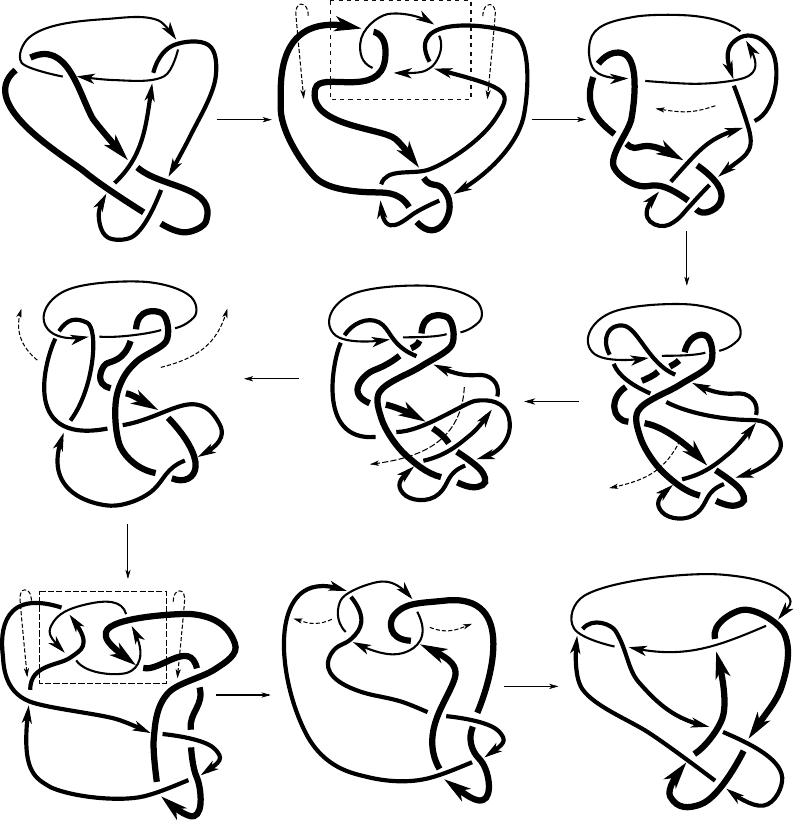}
\put(12,99){1}
\put(24,81){2}
\put(4,80){3}
\put(47,99){1}
\put(61,80){2}
\put(36,76){3}
\put(83,99){1}
\put(91,82){2}
\put(72,77){3}

\put(13,67){1}
\put(4,45){2}
\put(17,53.5){3}

\put(46,66.5){1}
\put(59,55){2}
\put(48,52.2){3}

\put(77,64.5){1}
\put(92,51.5){2}
\put(74,46){3}

\put(12,29){1}
\put(1,9){2}
\put(26,27){3}

\put(45,31){1}
\put(34,12){2}
\put(61,22){3}
\put(80,32){1}
\put(96,21){3}
\put(73,11){2}
\end{overpic}
\caption{$(8^3_2)^\gamma,~~\gamma=(1,1,1,1,(23))$}\label{832(23)}
\end{center}
\end{figure}

\newpage
\vspace{1cm}
\begin{figure}[ht]
\begin{center}
\begin{overpic}[scale=.9]
{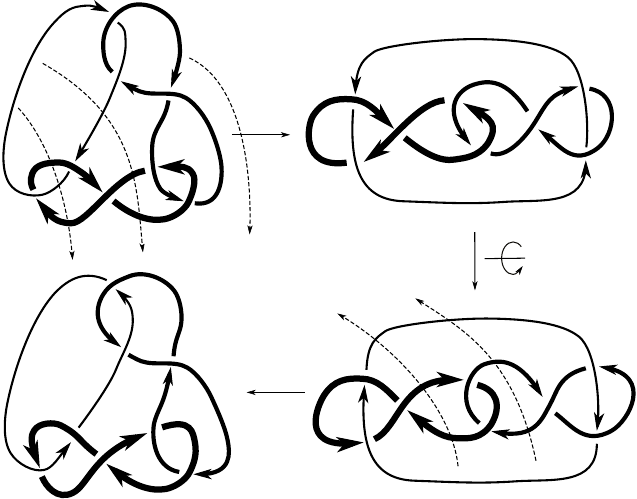}
\put(3,70){1}
\put(30,70){2}
\put(15,42){3}
\put(75,73){1}
\put(97,58){2}
\put(50,49){3}
\put(75,29){1}
\put(101,14){2}
\put(50,5){3}

\put(3,28){1}
\put(30,28){2}
\put(14,-1){3}
\end{overpic}
\caption{$(8^3_3)^\gamma,~~\gamma=(1,-1,-1,-1,e)$
\label{833invert}}
\end{center}
\end{figure}

\vspace{1in}
\begin{figure}[ht]
\begin{center}
\begin{overpic}[scale=.9]
{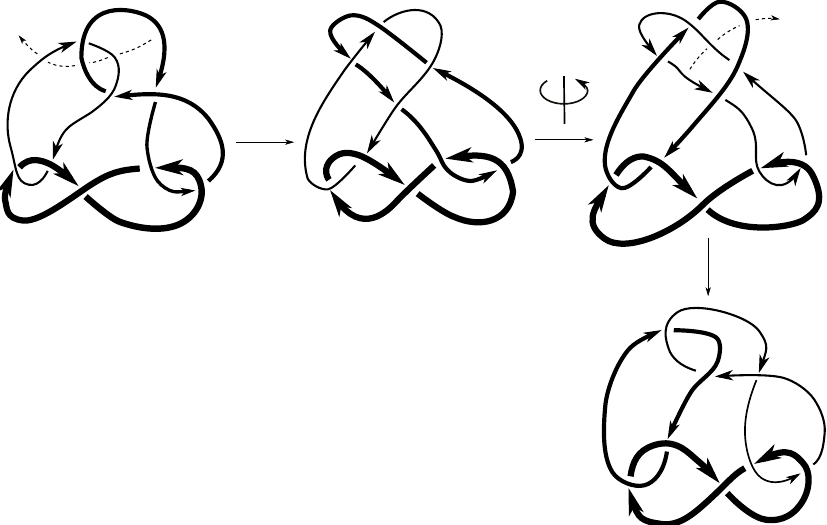}
\put(0,54){1}
\put(20,60){2}
\put(10,36){3}
\put(54,60){1}
\put(38,60){2}
\put(47,35){3}

\put(96,50){1}
\put(74.5,53){2}
\put(88,33){3}

\put(91,25){1}
\put(73,20){2}
\put(87.5,0){3}
\end{overpic}
\caption{$(8^3_3)^\gamma,~~\gamma=(1,1,1,1,(12))$}\label{833(12)}
\end{center}
\end{figure}

\newpage
\begin{figure}[ht]
\begin{center}
\begin{overpic}[scale=.9]
{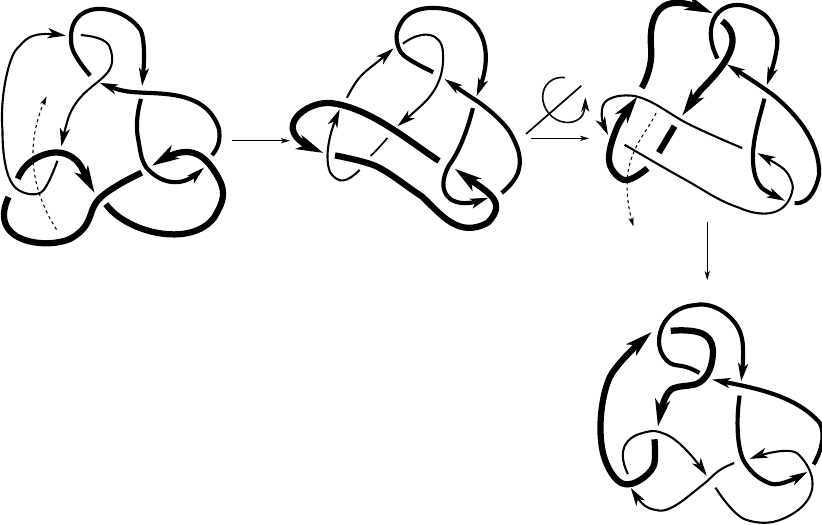}
\put(1,59){1}
\put(18,60){2}
\put(20,32){3}
\put(43,56){1}
\put(60,60){2}
\put(45,40){3}
\put(84,38){1}
\put(96,60){2}
\put(76,60){3}
\put(86,0){1}
\put(92,22){2}
\put(73,20){3}
\end{overpic}
\caption{$(8^3_3)^\gamma,~~\gamma=(1,1,1,1,(13))$}\label{833(13)}
\end{center}
\end{figure}

\newpage
\begin{figure}[ht]
\begin{center}
	\begin{overpic}[scale=.9]{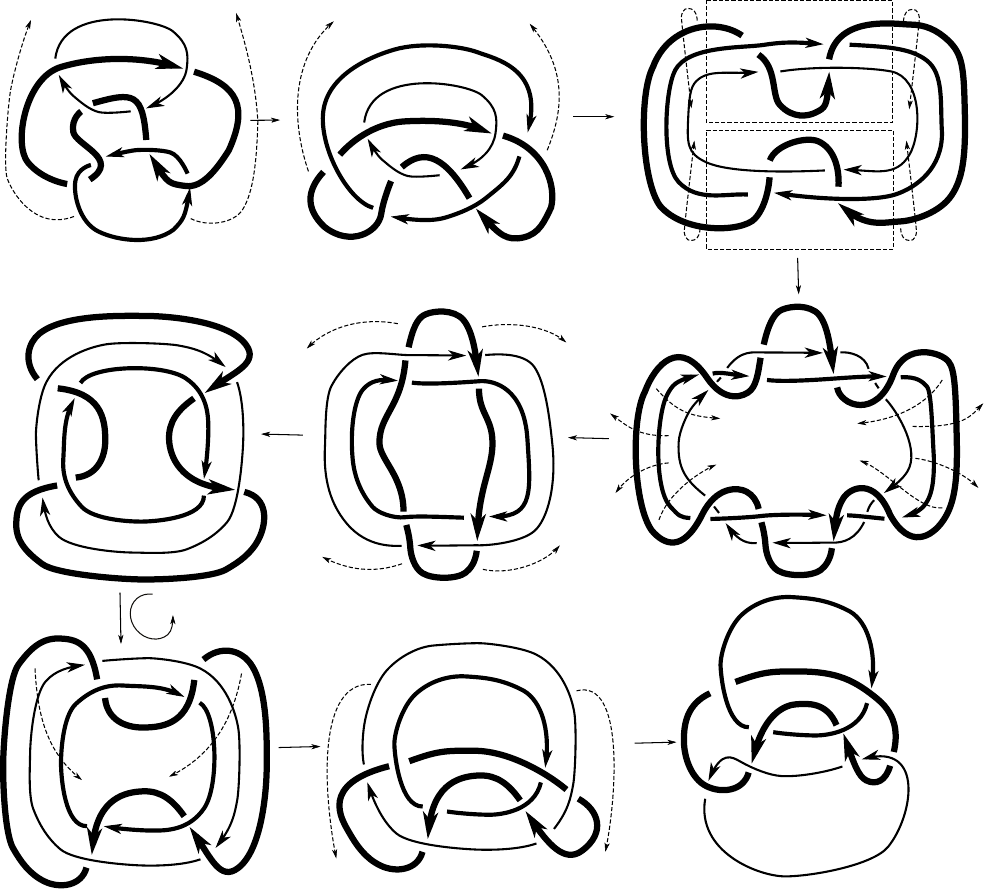}
	\put(11,89.5){1}
	\put(15,64){2}
	\put(4,75){3}
	\put(42,83){1}
	\put(54.5,82){2}
	\put(57,66){3}
	\put(72.5,80){1}
	\put(80,87){2}
	\put(96,88.5){3}
	\put(74,55){1}
	\put(81,49){2}
	\put(94,56){3}
	\put(35,55){1}
	\put(44,48){2}
	\put(45,60){3}
	\put(2,47){1}
	\put(14,50){2}
	\put(24,58){3}
	\put(13,0){1}
	\put(7,13){2}
	\put(25,25){3}
	\put(41,25){1}
	\put(47,18.5){2}
	\put(60.5,2){3}
	\put(72,1){1}
	\put(72,26){2}
	\put(92,18){3}
	\end{overpic}
\caption{$(8^3_4)^\gamma,~~\gamma=(-1,1,1,1,(12))$}\label{834}
\end{center}
\end{figure}

\newpage
\begin{figure}[ht]
\begin{center}
	\begin{overpic}[scale=0.9]{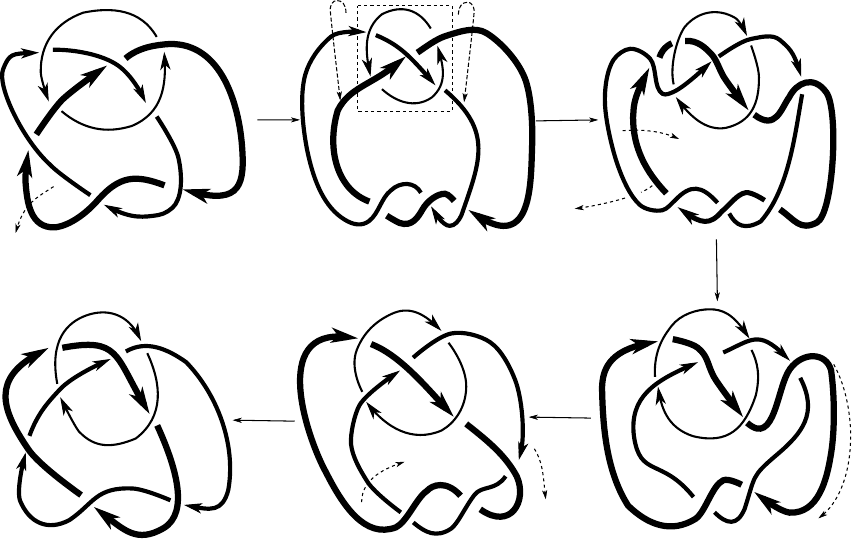}
	\put(5,62){1}
	\put(20,35.5){2}
	\put(3,33.5){3}
	\put(45,64){1}
	\put(35,58){2}
	\put(61,57){3}
	\put(85,63){1}
	\put(70,57){2}
	\put(98,52){3}
	\put(80,27.5){1}
	\put(91,24){2}
	\put(69,20){3}
	\put(47,28){1}
	\put(57,25){2}
	\put(34,22){3}
	\put(10,28){1}
	\put(23,22){2}
	\put(-2,18){3}
	\end{overpic}
\caption{$(8^3_5)^\gamma,~~\gamma=(1,-1,1,1,(23))$}\label{835}
\end{center}
\end{figure}

\vspace{1in}
\begin{figure}[ht]
\begin{center}
	\begin{overpic}[scale=0.9]{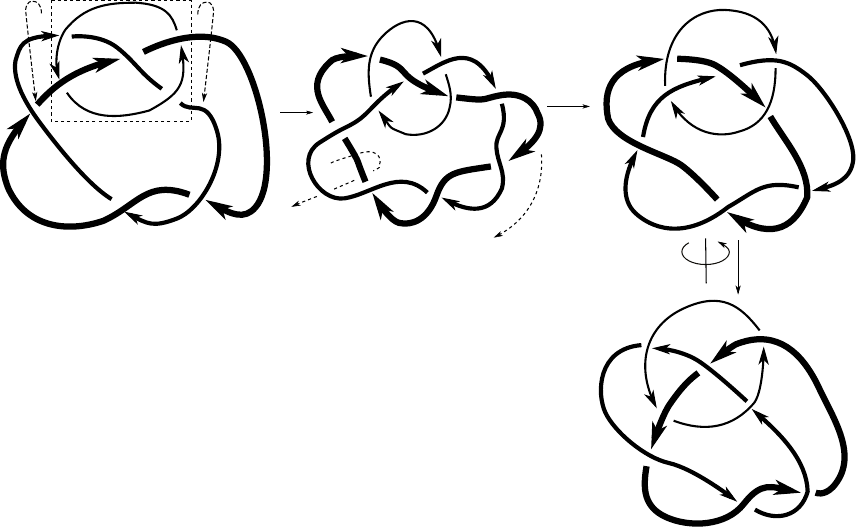}
	\put(10,63){1}
	\put(-1,56){2}
	\put(28.5,56){3}
	\put(45,60.5){1}
	\put(55,56.5){2}
	\put(35.5,55){3}
	\put(85,62){1}
	\put(96,54){2}
	\put(71,55){3}
	\put(76,25){1}
	\put(67,18){2}
	\put(96,19){3}
	\end{overpic}
\caption{$(8^3_5)^\gamma,~~\gamma=(1,1,-1,-1,e)$}\label{835b}
\end{center}
\end{figure}

\newpage
\begin{figure}[ht]
\begin{center}
	\begin{overpic}[scale=0.9]{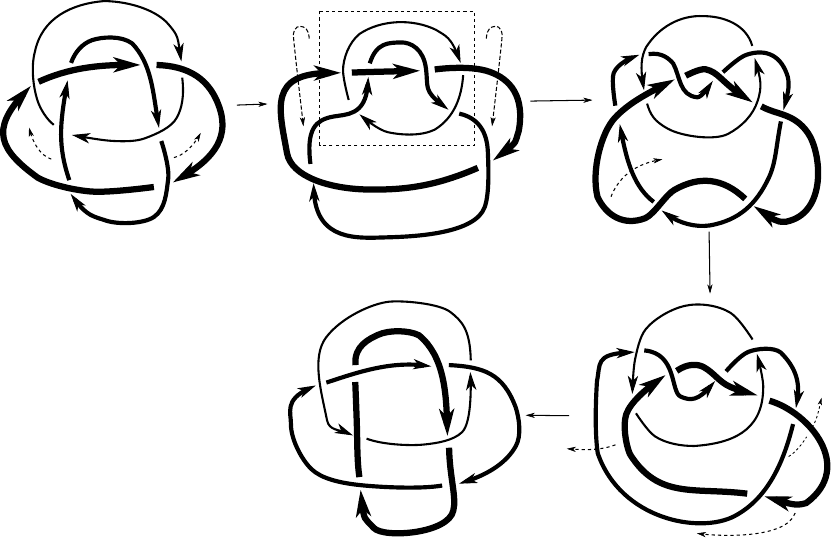}
	\put(12,66){1}
	\put(14,34){2}
	\put(-3,48){3}
	\put(45,66){1}
	\put(47,33){2}
	\put(32,54){3}
	\put(83,65){1}
	\put(96,57){2}
	\put(99.5,42){3}
	\put(78,27){1}
	\put(96,22){2}
	\put(101,7){3}
	\put(38,27){1}
	\put(33,12){2}
	\put(49,-2.5){3}
	\end{overpic}
\caption{$(8^3_6)^\gamma,~~\gamma=(1,-1,1,1,(23))$}\label{836}
\end{center}
\end{figure}

\vspace{1in}
\begin{figure}[ht]
\begin{center}
	\begin{overpic}[scale=0.9]{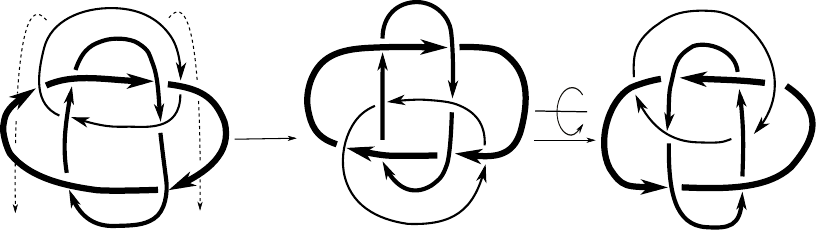}
	\put(12,28.5){1}
	\put(12,-3){2}
	\put(-2,11){3}
	\put(50,-3){1}
	\put(50,28.5){2}
	\put(38,22){3}
	\put(85,28){1}
	\put(87,-3){2}
	\put(101,12){3}
	\end{overpic}
\caption{$(8^3_6)^\gamma,~~\gamma=(-1,1,-1,-1,e)$}\label{836a}
\end{center}
\end{figure}

\newpage
\begin{figure}[ht]
\begin{center}
	\begin{overpic}[scale=0.9]{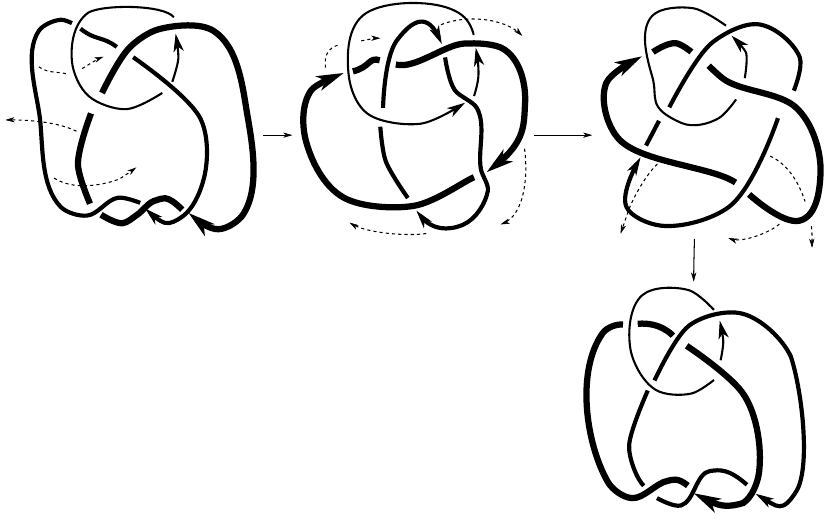}
	\put(15,64){1}
	\put(2,60){2}
	\put(28,60){3}
	\put(49,65){1}
	\put(55,33){2}
	\put(38,38){3}
	\put(81,65){1}
	\put(94,63){2}
	\put(70,56){3}
	\put(77,30){1}
	\put(94,26){2}
	\put(70,26){3}
	\end{overpic}
\caption{$(8^3_7)^\gamma,~~\gamma=(1, 1, 1,1,(23))$}\label{8n3_2}
\end{center}
\end{figure}

\vspace{1in}
\begin{figure}[ht]
\begin{center}
	\begin{overpic}[scale=0.9]{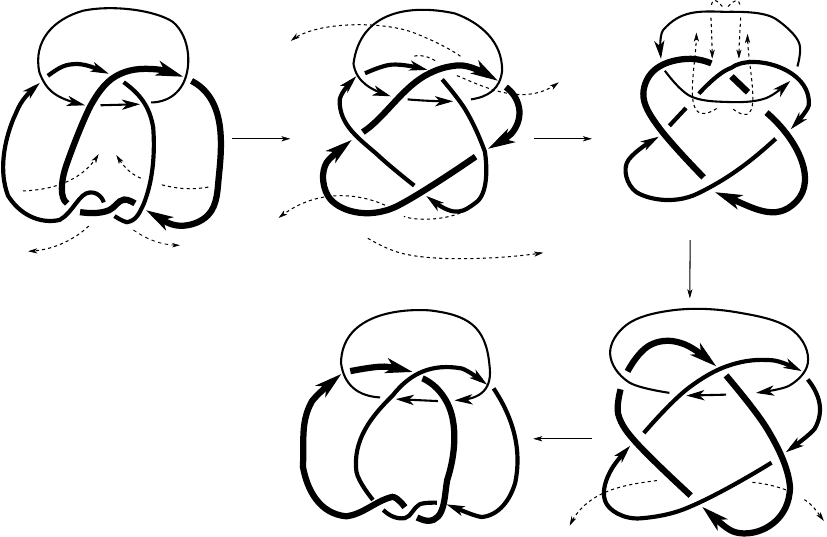}
	\put(12,65){1}
	\put(-3,47){2}
	\put(27,52){3}
	\put(51,65){1}
	\put(38,52){2}
	\put(64,49){3}
	\put(92,64){1}
	\put(98.5,54){2}
	\put(75,54){3}
	\put(87,28){1}
	\put(100,15){2}
	\put(71,15){3}
	\put(50,28){1}
	\put(62,16){2}
	\put(35,15){3}
	\end{overpic}
\caption{$(8^3_8)^\gamma,~~\gamma=(1,-1,1,1,(23))$}\label{8n4_1}
\end{center}
\end{figure}

\newpage
\begin{figure}[ht]
\begin{center}
	\begin{overpic}[scale=0.9]{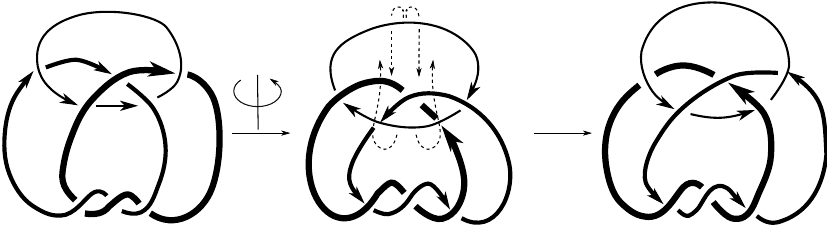}
	\put(12,28){1}
	\put(-3,10){2}
	\put(27,3){3}
	\put(55,25){1}
	\put(64,4){2}
	\put(34,3){3}
	\put(85,28){1}
	\put(101,10){2}
	\put(71,3){3}
	\end{overpic}
\caption{$(8^3_8)^\gamma,~~\gamma=(1, 1, -1,-1, (23))$}\label{8n4_2}
\end{center}
\end{figure}

\vspace{1in}
\begin{figure}[ht]
\begin{center}
	\begin{overpic}[scale=0.9]{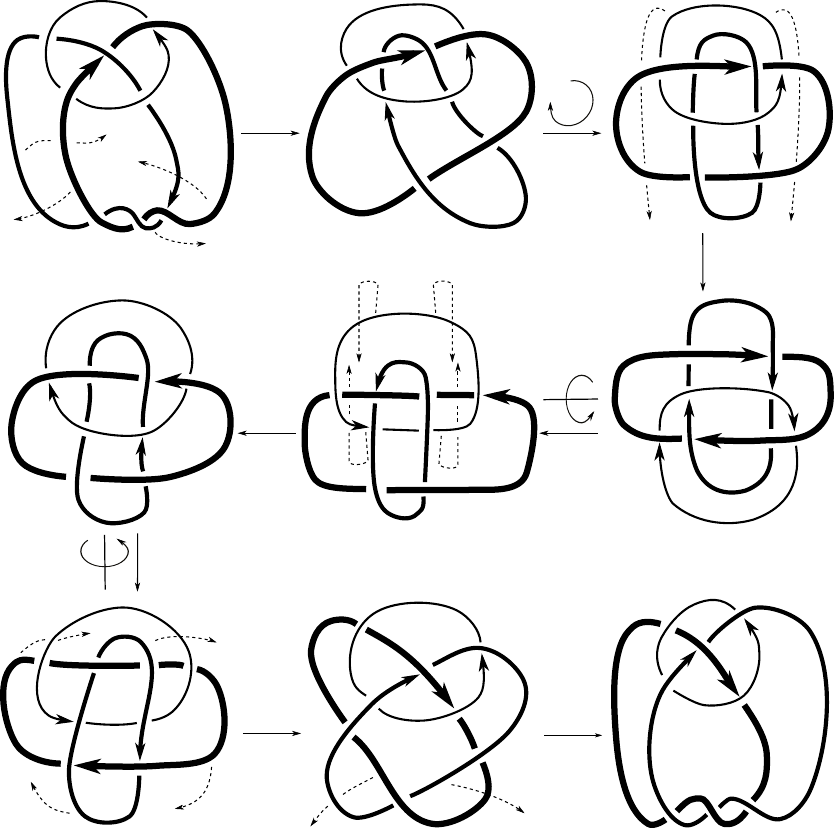}
	\put(10,100){1}
	\put(-3,92){2}
	\put(24,96){3}
	\put(49,100){1}
	\put(63,70){2}
	\put(37,71){3}
	\put(85,100){1}
	\put(91,71){2}
	\put(101,85){3}
	\put(85,33){1}
	\put(91,64){2}
	\put(100,52){3}
	\put(48,65){1}
	\put(48,34){2}
	\put(61,53){3}
	\put(14,65){1}
	\put(6,38){2}
	\put(-3,50){3}
	\put(8,27){1}
	\put(12,-3){2}
	\put(-2,15){3}
	\put(47,28){1}
	\put(60,22){2}
	\put(38,26){3}
	\put(85,28){1}
	\put(93,28){2}
	\put(73,24){3}
	\end{overpic}
\caption{$(8^3_9)^\gamma,~~\gamma=(1, 1, 1,1, (23))$}\label{8n5_2}
\end{center}
\end{figure}

\newpage
\vspace{1cm}
\begin{figure}[ht]
\begin{center}
	\begin{overpic}[scale=0.9]{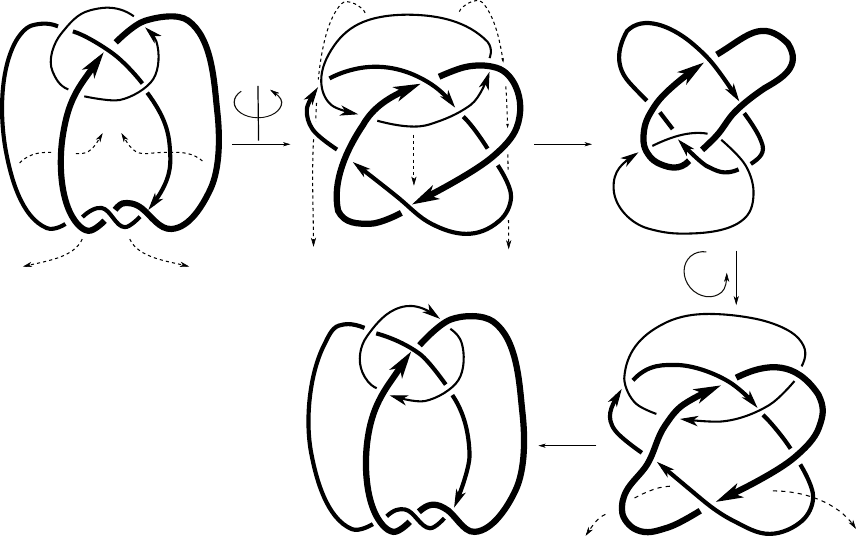}
	\put(11,63){1}
	\put(-3,51){2}
	\put(23,60){3}
	\put(49,63){1}
	\put(56,32){2}
	\put(38,32){3}
	\put(70,34){1}
	\put(70,60){2}
	\put(95,56){3}
	\put(92,25){1}
	\put(69,17){2}
	\put(97,17){3}
	\put(47,28){1}
	\put(35,24){2}
	\put(61,23){3}
	\end{overpic}
\caption{$(8^3_9)^\gamma,~~\gamma=(1,-1, 1,1, e)$}\label{8n5_3}
\end{center}
\end{figure}

\vspace{1in}
\begin{figure}[ht]
\begin{center}
	\begin{overpic}[scale=0.9]{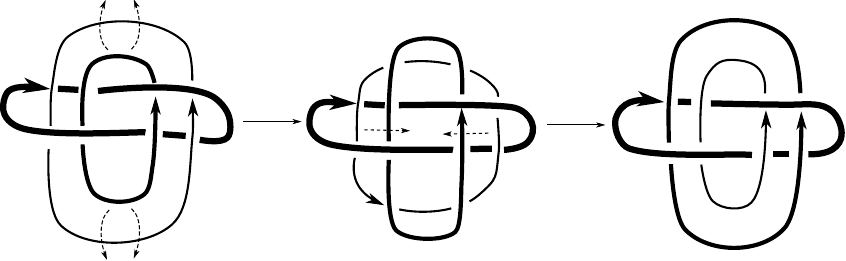}
	\put(23,26){1}
	\put(21,2){2}
	\put(1,12){3}
	\put(57,23){1}
	\put(55,2){2}
	\put(37,11){3}
	\put(85,24){1}
	\put(95,25){2}
	\put(74,10){3}
	\end{overpic}
\caption{$(8^3_{10})^\gamma,~~\gamma=(1, 1, 1,1, (12))$}\label{8n6_2}
\end{center}
\end{figure}
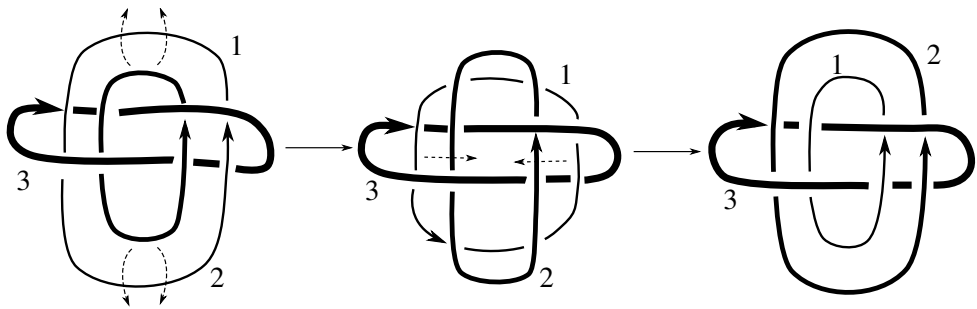

\clearpage
\newpage
\section{Isotopy Figures for four-component links} \label{app:4isotopy}

This appendix contains all of the isotopy figures for three-component links that are not contained in Appendix~\ref{app:rotate}.  

\vspace{0.5in}
\begin{figure}[ht]
\begin{center}
\begin{overpic}[scale=.9]{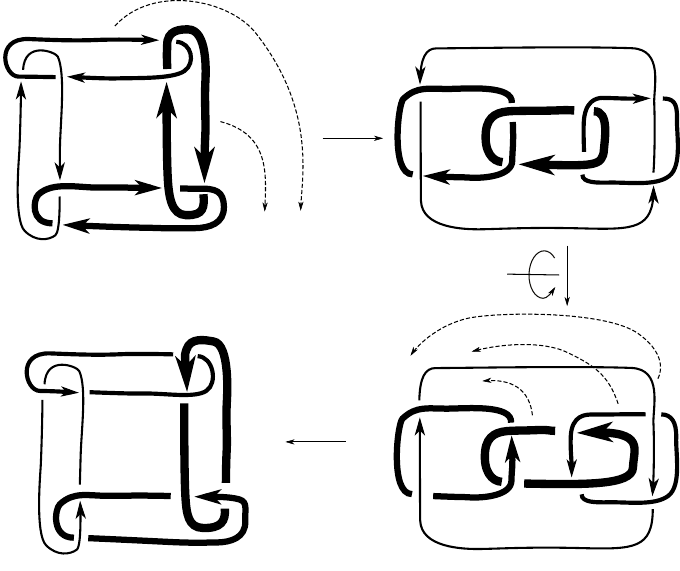}
\put(0,62){1}
\put(15,80){2}
\put(17,46){3}
\put(32.5,70){4}
\put(80,77){1}
\put(90,71){2}
\put(80,55){3}
\put(56,57){4}
\put(66,31){1}
\put(101,16){2}
\put(67,7){3}
\put(79,9){4}
\put(3.5,16){1}
\put(17,33){2}
\put(17,0){3}
\put(35,20){4}
\end{overpic}
\caption{$(8^4_1)^\gamma,~~\gamma=(1,-1,-1,-1,-1,e)$
\label{841invert}}
\end{center}
\end{figure}

\vspace{1in}
\begin{figure}[ht]
\begin{center}
	\begin{overpic}[scale=0.9]{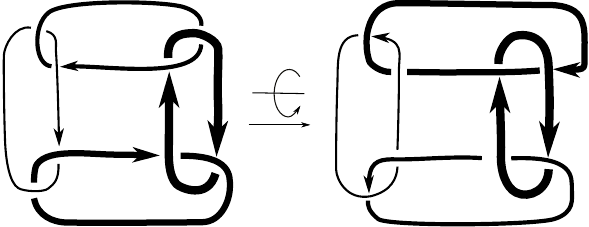}
	\put(-4,18){1}
	\put(15,29){2}
	\put(15,-4){3}
	\put(32,20){4}
	\put(69,16){1}
	\put(76,-4){2}
	\put(76,32){3}
	\put(97,15){4}
	\end{overpic}
\caption{$(8^4_1)^\gamma,~~\gamma=(1,-1,-1,-1,-1, (23))$}\label{841_2}
\end{center}
\end{figure}

\newpage
\vspace{1cm}
\begin{figure}[ht]
\begin{center}
\begin{overpic}[scale=.9]{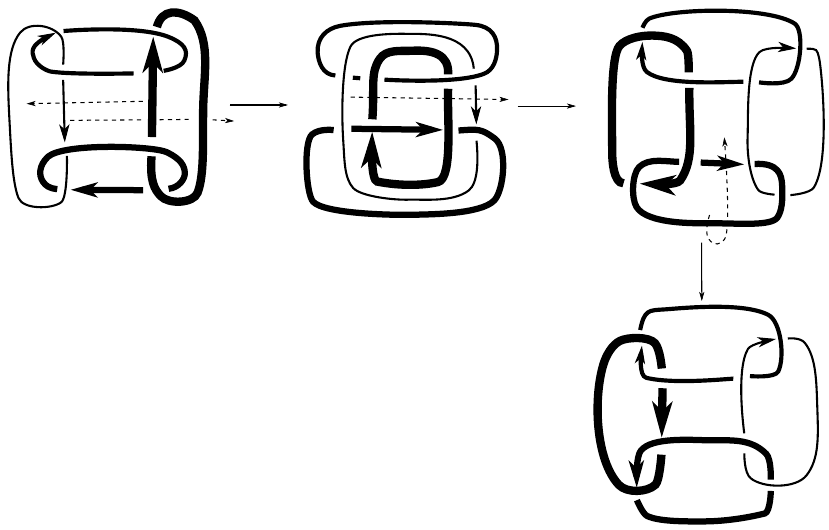}
\put(-1,50){1}
\put(12,62){2}
\put(12,37){3}
\put(26.5,55){4}
\put(39,51.5){1}
\put(50,62){2}
\put(45,35){3}
\put(49,43){4}
\put(100,54){1}
\put(95,62.5){2}
\put(92,34){3}
\put(73,39){4}
\put(99,22){1}
\put(80,28){2}
\put(85,-2){3}
\put(69,15){4}
\end{overpic}
\caption{$(8^4_{2})^\gamma,~~\gamma=(1,1,1,-1,1,(14))$}\label{842_2}
\end{center}
\end{figure}

\vspace{1in}
\begin{figure}[ht]
\begin{center}
\begin{overpic}[scale=.9]{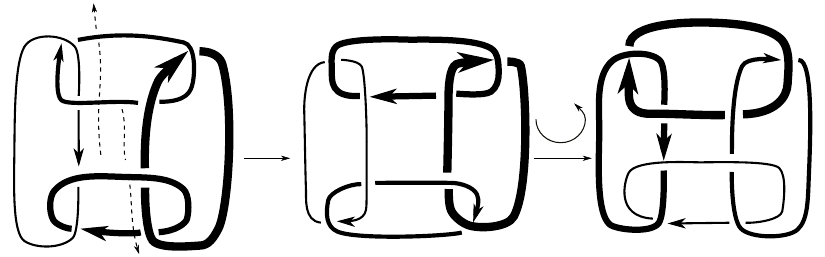}
\put(0,15){1}
\put(15,28){2}
\put(11.5,-1){3}
\put(29,16){4}
\put(35,16){1}
\put(50,-1){2}
\put(50,27.5){3}
\put(65,20){4}
\put(83,1){1}
\put(99,20){2}
\put(73,.5){3}
\put(95,27){4}
\end{overpic}
\caption{$(8^4_{2})^\gamma,~~\gamma=(1,1,1,1,1,(13)(24))$}\label{842_3}
\end{center}
\end{figure}

\newpage
\begin{figure}[ht]
\begin{center}
\begin{overpic}[scale=.9]{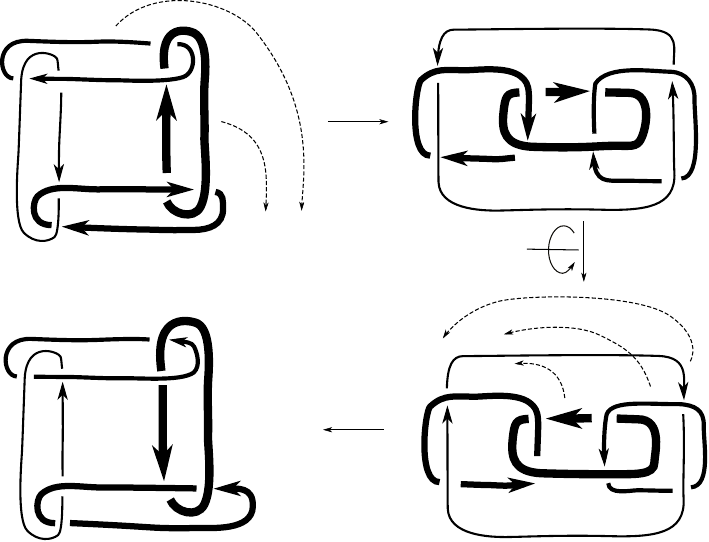}
\put(0,55){1}
\put(10,72){2}
\put(15,40){3}
\put(30.5,64){4}
\put(75,74){1}
\put(99,65){2}
\put(56.5,64){3}
\put(78,52){4}
\put(67,27.5){1}
\put(100,19){2}
\put(59,20){3}
\put(80,6){4}
\put(1,15){1}
\put(15,30){2}
\put(15,-1.5){3}
\put(31,20){4}
\end{overpic}
\caption{$(8^4_{3})^\gamma,~~\gamma=(1,-1,-1,-1,-1,e)$}\label{843invert}
\end{center}
\end{figure}

\vspace{1in}
\begin{figure}[ht]
\begin{center}
\begin{overpic}[scale=.9]{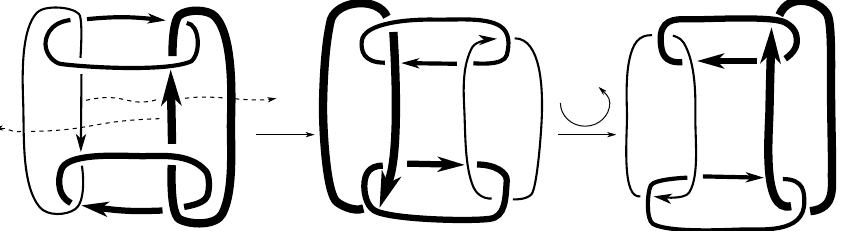}
\put(.5,15){1}
\put(15,27){2}
\put(15,-2){3}
\put(29,19){4}
\put(64,6){1}
\put(50,27){2}
\put(50,-3){3}
\put(35,4){4}
\put(72,22){1}
\put(86,-3){2}
\put(85,26){3}
\put(100,14.5){4}
\end{overpic}
\caption{$(8^4_{3})^\gamma,~~\gamma=(-1,1,1,1,1,(23))$}\label{843_1}
\end{center}
\end{figure}

\newpage
\begin{figure}[ht]
\begin{center}
\begin{overpic}[scale=.9]{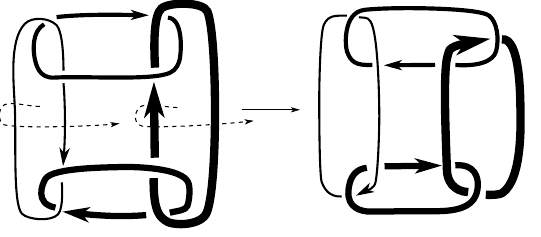}
\put(1,41){1}
\put(12,43){2}
\put(12,-1){3}
\put(38,0){4}
\put(59,41){1}
\put(70,43){2}
\put(86,0){3}
\put(99,35){4}
\end{overpic}
\caption{$(8^4_{3})^\gamma,~~\gamma=(-1,-1,1,1,-1,e)$}\label{843_2}
\end{center}
\end{figure}

\vspace{1in}
\begin{figure}[ht]
\begin{center}
\begin{overpic}[scale=.9]{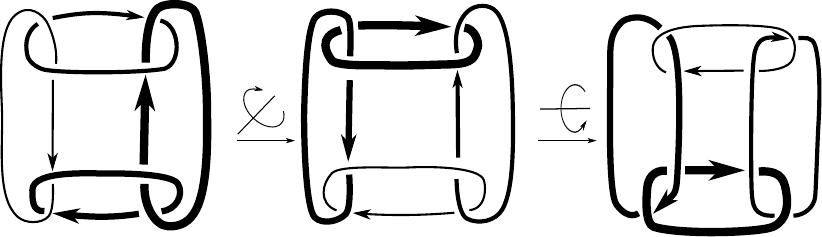}
\put(-2,15){1}
\put(10,28){2}
\put(10,-1){3}
\put(26,22){4}
\put(46,-.5){1}
\put(63,22){2}
\put(35,23){3}
\put(47,28){4}
\put(86,26.5){1}
\put(101,20){2}
\put(71,23){3}
\put(86,-3){4}
\end{overpic}
\caption{$(8^4_{3})^\gamma,~~\gamma=(-1,1,1,1,1,(1342))$}\label{843_3}
\end{center}
\end{figure}

\end{document}